\documentclass[11pt]{article}
\usepackage[utf8]{inputenc}
\usepackage{amscd,amsmath,amsthm,amssymb,amsfonts}
\usepackage{txfonts}
\usepackage{eucal}
\usepackage{bbm}
\usepackage{appendix}
\usepackage{authblk}
\usepackage{graphicx} 
\usepackage{pifont}
\usepackage{hyperref}
\hypersetup{
    colorlinks=true,
    linkcolor=blue,
    filecolor=blue,      
    urlcolor=blue,
    citecolor=blue,
    }
\usepackage{algorithm}
\usepackage{mathrsfs}
\usepackage{derivative}
\usepackage[noend]{algorithmic}
\usepackage[a4paper]{geometry}
\usepackage{tikz-cd}
\usepackage{mathtools}
 \usepackage[nottoc]{tocbibind}    
\usepackage{empheq}
\usepackage{graphics}
\usepackage{indentfirst}
\usepackage[
backend=bibtex,
style=numeric-comp,
maxnames=4
]{biblatex}
\usepackage{color}
\usepackage{tabularx,colortbl}
\numberwithin{equation}{section}

\usepackage{subcaption}
\DeclareCaptionLabelFormat{custom}
{%
      \textbf{#1 #2}
}
\DeclareCaptionLabelSeparator{custom}{ }
\DeclareCaptionFormat{custom}
{%
    #1 #2 #3
}
\usepackage{etoolbox}

\makeindex

\def\s{$\mbox{\c{s}}$}

\def\opn#1#2{\def#1{\operatorname{#2}}} 

\opn\Ker{Ker} \opn\Coker{Coker}  \opn\Hom{Hom} \opn\Im{Im}
\opn\End{End} \opn\Aut{Aut} \opn\defect{def} \opn\ord{ord}
\opn\id{id} \opn\dim{dim} \opn\det{det} \opn\tr{tr} \opn\grad{grad} \opn\lcm{lcm}
\opn\min{min} \opn\max{max} 
\opn\Span{Span}   \opn\rang{rang}  \opn\id{id} \opn\Ass{Ass} \opn\Min{Min}
\opn\GL{GL} \opn\SL{SL} \opn\mod{mod} \opn\diag{diag}
\opn\min{min} \opn\sgn{sgn} \opn\ini{in_<}  \opn\Mon{Mon} \opn\LC{LC_<} \opn\Hom{Hom} \opn\Ext{Ext} \opn\gini{gin_{<_{rev}}} \opn\gin{gin_{<}}
\opn\LT{LT_<}
\opn\s{supp} \opn\Tor{Tor} \opn\link{link} \opn\depth{depth} \opn\pd{pd} \opn\reg{reg} 

\newcommand{\supp}{\textup{supp}}
\newcommand{\der}{\textup{d}}

\date{}
\title{Mass concentration in a spatially inhomogeneous coagulation model with fast sedimentation}
\author[a,1]{Iulia Cristian}
\author[b,2]{Juan J. L. Vel\'{a}zquez}
\affil[a]{Laboratoire Jacques-Louis Lions, Sorbonne University, 4 Place Jussieu, 75005 Paris, France}
\affil[b]{Institute for Applied Mathematics, University of Bonn, Endenicher Allee 60, 53115 Bonn, Germany}
\affil[1]{\href{mailto:cristian@iam.uni-bonn.de}{iulia.cristian@sorbonne-universite.fr}}
\affil[3]{\href{mailto:velazquez@iam.uni-bonn.de}{velazquez@iam.uni-bonn.de}}

\begin{document}
\newtheorem{teo}{Theorem}[section]
\newtheorem{ex}[teo]{Example}
\newtheorem{prop}[teo]{Proposition}
\newtheorem{obss}[teo]{Observations}
\newtheorem{cor}[teo]{Corollary}
\newtheorem{lem}[teo]{Lemma}
\newtheorem{prob}[teo]{Problem}
\newtheorem{conj}[teo]{Conjecture}
\newtheorem{exs}[teo]{Examples}
\newtheorem{ans}[teo]{Ansatz}
\newtheorem{alg}[teo]{\bf Algorithm}

\theoremstyle{definition}
\newtheorem{defi}[teo]{Definition}

\theoremstyle{remark}
\newtheorem{rmk}[teo]{Remark}
\newtheorem{ass}[teo]{Assumption}
\maketitle

\vspace{-1cm}

\begin{abstract}
We study a spatially inhomogeneous coagulation model that contains a transport term in the spatial variable. The transport term models the vertical motion of particles due to gravity, thereby incorporating their fall into the dynamics. Local existence of mass-conserving solutions for a class of coagulation rates for which in the spatially homogeneous case instantaneous gelation (i.e., instantaneous loss of mass) occurs has been proved in \cite{cristianinhom}. In order to obtain some insight into how to prove global existence of solutions, we allow a fast sedimentation speed. For very fast sedimentation speed, we rigorously prove that solutions converge to a Dirac measure in the space variable. We also formally obtain in the limit a one-dimensional coagulation equation with diagonal kernel, i.e., only particles of the same size interact. This provides a physical intuition on how coagulation models with a diagonal kernel emerge. 
\end{abstract}
\tableofcontents
\section{Introduction}
\subsection{Background}
It is known that gelation (loss of mass) occurs for the standard one-dimensional coagulation model,
\begin{align*}
    \partial_{\tau}f(v,\tau)=\frac{1}{2}\int_{(0,v)}K(v-w,w)f(v-w,\tau)f(w,\tau)\der w -\int_{(0,\infty)}K(v,w)f(v,\tau)f(w,\tau)\der w,
\end{align*}
when the coagulation kernel behaves like a power law of homogeneity $\gamma>1$, see \cite{gelationpaper,papergelationlaurencot, esclaumischperth, gelfournier2025}. In addition, for sum kernels of the form
\begin{align}\label{sum kernel}
    K(v,w)=v^{\gamma}+w^{\gamma}, \quad \gamma>1,
\end{align}
gelation happens instantaneously and one can prove that solutions belonging to $L^{1}$ do not exist, see \cite{ballcarr,bookcoagulation,Carr1992,Dongen}. Nonetheless, there are some results concerning the existence of solutions for sum-type kernels as in \eqref{sum kernel} in the spatially inhomogeneous case,
\begin{align}\label{rain model}
    \partial_{\tau}f(x,v,\tau)+v^{\alpha}\partial_{x}f(x,v,\tau)=&\frac{1}{2}\int_{(0,v)}K(v-w,w)f(x,v-w,\tau)f(x,w,\tau)\der w \nonumber\\
    &-\int_{(0,\infty)}K(v,w)f(x,v,\tau)f(x,w,\tau)\der w\\
    =:&\mathbb{K}_{1}[f](x,v,\tau)-\mathbb{K}_{2}[f](x,v,\tau) =:\mathbb{K}[f](x,v,\tau), \quad x\in\mathbb{R}.\nonumber
\end{align}
More precisely, when $\gamma<\alpha+1, \alpha\geq 0$, existence of solutions (which may not preserve mass) for a discrete version of \eqref{rain model} with a coagulation kernel as in \eqref{sum kernel} was proved in \cite{galkin}, while existence of mass-conserving solutions for short times for the model \eqref{rain model} when $\gamma<\alpha+1, \alpha\in(0,1)$ was proved in \cite{cristianinhom}. It is also worth mentioning that the relevance of the model in \eqref{rain model} with coagulation kernels of homogeneity $\gamma>1$ stems from the fact that it was first introduced in order to describe the formation of rain or the behavior of air bubbles in water, see \cite{physicspaper, precipitation, raininitiation, tanaka}. In this case, the coagulation kernel has the form
\begin{align}\label{rain kernel intro}
    K(v,w)=|v^{\alpha}-w^{\alpha}|(v^{\frac{1}{3}}+w^{\frac{1}{3}})^{2}.
\end{align}

The main idea in \cite{cristianinhom} in order to prove existence of mass-conserving solutions for the model \eqref{rain model} is that the coagulation operator $\mathbb{K}[f]$ gives a negligible contribution for sufficiently small times. It is thus natural to ask what is the time range for which mass-conserving solutions exist if we work with a model of the form
\begin{align}\label{original equation}
    \partial_{\tau}f_{\epsilon}(x,v,\tau)+v^{\alpha}\partial_{x}f_{\epsilon}(x,v,\tau)=\epsilon \mathbb{K}[f_{\epsilon}](x,v,\tau),
\end{align}
for some sufficiently small $\epsilon\in(0,1)$ and when $K$ is as in \eqref{sum kernel} or \eqref{rain kernel intro}. We expect that this modification of our model will offer clues on how to prove global in time existence of mass-conserving solutions. The main goal of this paper is thus to prove that mass-conserving solutions of \eqref{original equation} exist for times of order $\mathcal{O}(e^{\frac{1}{\epsilon}})$. While kernels of homogeneity $\gamma\leq 1$ are well-studied in the mathematical literature, see for example \cite{escobedomischler, escobedomischler2, stewart}, this is the first result of existence of mass-conserving solutions involving kernels of homogeneity $\gamma>1$, with the exceptions being the product kernel \cite{menonpego,tranvan} and our previous result in \cite{cristianinhom}.

Notice also that since we can rescale the time $\tau\rightarrow \epsilon \tau$, \eqref{original equation} is equivalent to the following form.
\begin{align}\label{one over epsilon first time}
\partial_{\tau}f_{\epsilon}(x,v,\tau)+\frac{1}{\epsilon}v^{\alpha}\partial_{x}f_{\epsilon}(x,v,\tau)= \mathbb{K}[f_{\epsilon}](x,v,\tau).
\end{align}

\subsection{Connection with coagulation models with diagonal kernel}
The first question is whether the proof in \cite{cristianinhom} can be reproduced in order to extend the existence result for the model \eqref{original equation}. It was argued in \cite[Section 6.4]{phdthesis} that  the steps to tackle
local in time existence cannot be reproduced in order to prove global existence of solutions (or even existence for times of order one) for the model \eqref{original equation}. As such, a new strategy needs to be adopted.

Different forms of solutions for the model \eqref{original equation} have then been considered in \cite[Section 6.5]{phdthesis} (such as polynomial decaying solutions, Dirac measure solutions, and Dirac measure solutions with sufficiently fast decay in the $v$ variable) and it has been noticed at a formal level that the solutions can be expected to portray a self-similar behavior of the form $f_{\epsilon}(x,v,\tau)\approx \frac{1}{\tau}S_{\epsilon}\big(\frac{x}{\tau},v\big)$ as $\tau\rightarrow \infty$.

Motivated by the computations in \cite{phdthesis}, we make the following rescaling for solutions of \eqref{rain model}.
\begin{align}
f(x,v,\tau)=\frac{1}{\tau+1}H(y,v,t), \quad y:=\frac{x}{\tau+1}, \quad t:= \ln(\tau+1),
\end{align}
or, in other words, we have that
\begin{align}
H(y,v,t)=e^{t}f(x,v,\tau), \quad x:=e^{t}y, \quad \tau:=e^{t}-1.
\end{align}

Then $H$ satisfies
\begin{align}\label{v minus y rescaling}
\partial_{t}H(y,v,t)+\partial_{y}[(v^{\alpha}-y)H(y,v,t)]=\mathbb{K}[H](y,v,t).
\end{align}

Since the rain kernel \eqref{rain kernel intro} vanishes on the diagonal $\{v=w\}$, it is straightforward to check that a steady state of the model \eqref{v minus y rescaling} is given by $G(v)\delta(y-v^{\alpha})$.

\begin{ans}\label{ansatz} We thus make the following ansatz in the case when the coagulation kernel in \eqref{v minus y rescaling} is $K_{\epsilon}(v,w)=\epsilon(v^{\gamma}+w^{\gamma})$, for some $\epsilon<1$ small. For large times, solutions of the model \eqref{v minus y rescaling} can be expected to concentrate around $|y-v^{\alpha}|\ll 1$. Notice however that  $G(v)\delta(y-v^{\alpha})$ are not steady states of \eqref{v minus y rescaling} in this case. Nonetheless, our assumptions are motivated by the fact that $K_{\epsilon}$ and the rain kernel \eqref{rain kernel intro}  both behave like sum-type kernels with a small contribution  on the diagonal $\{v=v'\}$.
\end{ans}
 We thus formally assume $H$ is of the form 
\begin{align}\label{y v alpha limit form}
    H(y,v,t)=G(v,t)\delta(y-v^{\alpha}).
\end{align}
 We integrate the coagulation operator $\mathbb{K}$ in the $y$ variable. For the loss term $\mathbb{K}_{2}$ in \eqref{rain model}, we have
\begin{align*}\int_{\mathbb{R}} \mathbb{K}_{2}[H](y,v,t)\der y &=\int_{\mathbb{R}}\int_{(0,\infty)}K(v,w)G(v,t)\delta(y-v^{\alpha})G(w,t)\delta(y-w^{\alpha})\der w \der y\\
    &=   \int_{(0,\infty)}K(v,w)G(v,t)G(w,t)\delta(w^{\alpha}-v^{\alpha})\der w
    \\
    &=   \frac{v^{1-\alpha}}{\alpha}\int_{(0,\infty)}(v^{\gamma}+w^{\gamma})G(v,t)G(w,t)\delta(w-v)\der w
       \\
    &=\frac{2}{\alpha}v^{\gamma+1-\alpha}(G(v,t))^{2}.
\end{align*}
For the gain term $\mathbb{K}_{1}$ in \eqref{rain model}, it holds that
\begin{align}\label{the gain term dirac}
 \int_{\mathbb{R}}\mathbb{K}_{1}[H](y,v,t) \der y  &=\frac{1}{2}\int_{\mathbb{R}}\int_{(0,v)}K(v-w,w)G(v-w,t)\delta(y-(v-w)^{\alpha})G(w,t)\delta(y-w^{\alpha})\der w \der y\nonumber\\
    &=   \frac{1}{2}\int_{(0,v)}K(v-w,w)G(v-w,t)G(w,t)\delta((v-w)^{\alpha}-w^{\alpha})\der w 
   \nonumber \\
    &=   \frac{1}{4\alpha}\bigg(\frac{v}{2}\bigg)^{1-\alpha}\int_{(0,v)}((v-w)^{\gamma}+w^{\gamma})G(v-w,t)G(w,t)\delta(w-\frac{v}{2})\der w
    \nonumber   \\
    &=   \frac{1}{2\alpha}\bigg(\frac{v}{2}\bigg)^{\gamma+1-\alpha}\bigg(G\big(\frac{v}{2},t\big)\bigg)^{2}.
\end{align}
 Thus, $G$ satisfies the equation
\begin{align}\label{first diagonal}
    \partial_{t}G(v,t)=    \frac{2}{\alpha}\bigg[ \frac{1}{4}\bigg(\frac{v}{2}\bigg)^{\gamma+1-\alpha}\bigg(G(\frac{v}{2},t)\bigg)^{2}-v^{\gamma+1-\alpha}(G(v,t))^{2}\bigg].
\end{align}
\begin{rmk}\label{remark homogeneity decreases}
Thus, $G$ satisfies a coagulation equation in which the coagulation kernel is a diagonal kernel of the form $v^{\gamma+1-\alpha}\delta(v-w)$ multiplied by some constant. Notice that this means the kernel has homogeneity $\gamma-\alpha$ and thus the existence of solutions for this type of equation is well-known and has been proved in \cite{diagonal, leyvraz,diagonal2} when $\gamma-\alpha<1$.  
\end{rmk}
\subsection{Self-consistent argument for behavior of solutions as $t\rightarrow\infty$}
Equation \eqref{first diagonal} suggests (at least formally) that the function $G(v,t):=\int H(y,v,t)\der y$, with $H$ as in \eqref{v minus y rescaling}, satisfies a one-dimensional 
 coagulation model with a diagonal kernel of homogeneity $\gamma-\alpha<1$. We denote by  $[\cdot]$ the dimensional properties of a quantity. We analyze the scaling properties of mass-conserving solutions to \eqref{first diagonal}. Mass-conservation $\int  v G(v,t)\der v \equiv constant$ implies that
 \begin{align}\label{mass cons sol}
     [G][v]^{2}=1.
 \end{align}
 From \eqref{first diagonal}, we further deduce that
\begin{align}\label{scaling properties}
\frac{[G]}{[t]}=[v]^{\gamma-\alpha+1}[G]^{2}.
\end{align}
From \eqref{mass cons sol} and \eqref{scaling properties}, we have that 
 \begin{align*}
     [G]=\frac{1}{[v]^{\gamma-\alpha+1}[t]}\textup{ and } [v]=[t]^{\frac{1}{1-(\gamma-\alpha)}}.
 \end{align*} With Ansatz \ref{ansatz}, we have that $y$ should rescale as $v^{\alpha}$ and thus
 \begin{align*}
     [y]=[v]^{\alpha}=[t]^{\frac{\alpha}{1-(\gamma-\alpha)}}.
 \end{align*} Since we also want mass-conservation $\int \int v H(y,v,\tau)\der v \der y \equiv constant$, then the long-time behavior of solutions $H$ in \eqref{v minus y rescaling} is described by $g$, with $g$ having the following form.
\begin{align*}
    H(y,v,t)=\frac{1}{\tau^{\frac{2+\alpha}{1-(\gamma-\alpha)}}}g(Y,V, s), \textup{ with } Y:=\frac{y}{t^{\frac{\alpha}{1-(\gamma-\alpha)}}}, \quad V:=\frac{v}{t^{\frac{1}{1-(\gamma-\alpha)}}}, \quad  s=\ln t.
\end{align*}
Denoting by $\beta:=\frac{1}{1-(\gamma-\alpha)}$, we obtain that $g$ satisfies
\begin{align*}
   \frac{1}{\tau^{\beta(2+\alpha)+1}}& \big[\partial_{s}g(Y,V,s)-\beta(2+\alpha)g-\beta V\partial_{V}g-\alpha\beta Y\partial_{Y}g\big]+\tau^{-\beta(2+\alpha)}\partial_{Y}[(V^{\alpha}-Y)g]= \frac{\tau^{\beta(\gamma+1)}}{\tau^{2\beta(2+\alpha)}}\mathbb{K}[g]
\end{align*}
or alternatively since $\beta(\gamma-\alpha-1)+1=0$, it holds that
\begin{align*}
  \partial_{s}g(Y,V,s)-\beta(2+\alpha)g- \beta V\partial_{V}g-\alpha\beta Y\partial_{Y}g + e^{s}\partial_{Y}[(V^{\alpha}-Y)g]=\mathbb{K}[g](Y,V,s),
\end{align*}
where we remember that we are interested in the behavior for large times.

Notice that the transport term $e^{s}\partial_{Y}[(V^{\alpha}-Y)g]$ can be expected to give a large contribution as $s\rightarrow \infty$. This justifies the study of the equation
\begin{align*}
    \partial_{t}H_{\epsilon}(x,v,t)+\frac{1}{\epsilon}\partial_{x}[(v^{\alpha}-x)H_{\epsilon}(x,v,t)]=\mathbb{K}[H_{\epsilon}](x,v,t)
\end{align*} as a first natural step in order to study the long-time behavior of solutions to \eqref{rain model}.

\subsection{Connection with previous results for spatially inhomogeneous coagulation models}

Under these conditions, the natural approach in order to extend the existence result of mass-conserving solutions is thus to analyze the following rescaled version of the model \eqref{original equation}. Let $f_{\epsilon}$ be a solution of \eqref{original equation}. Take $H_{\epsilon}$ to be such that
\begin{align}\label{rain model to our model}
f_{\epsilon}(x,v,\tau)=\frac{1}{\tau+1}H_{\epsilon}(y,v,t), \quad y:=\frac{x}{\tau+1}, \quad t:=\epsilon \ln(\tau+1),
\end{align}
or, in other words, we have that
\begin{align}\label{our model to rain model}
H_{\epsilon}(y,v,t)=e^{\frac{t}{\epsilon}}f_{\epsilon}(x,v,\tau), \quad x:=e^{\frac{t}{\epsilon}}y, \quad \tau:=e^{\frac{t}{\epsilon}}-1.
\end{align}
Then $H_{\epsilon}$ satisfies
\begin{align}\label{diagonal equation}
\partial_{t}H_{\epsilon}(y,v,t)+\frac{1}{\epsilon}\partial_{y}[(v^{\alpha}-y)H_{\epsilon}(y,v,t)]=\mathbb{K}[H_{\epsilon}](y,v,t).
\end{align}

We analyze equation \eqref{diagonal equation} in more detail. We first remember the existence result for mass-conserving solutions in \cite{cristianinhom}. 
\begin{teo}[Local existence of mass-conserving solutions for small times]\label{main theorem} Let $K$ be as in \eqref{sum kernel} or \eqref{rain kernel intro} with $\alpha\in(0,1), \gamma\in[0,1+\alpha)$. Let $T>0$ be sufficiently small. Let $f_{\textup{in}}\in \textup{C}^{1}(\mathbb{R}\times\mathbb{R}_{> 0})$ such that
\begin{align*}
    f_{\textup{in}}(x,v)\leq \frac{A}{1+|x|^{m}+v^{p}},
\end{align*}
for some $A>0$, $x\in\mathbb{R}$, $v\in(0,\infty)$, and where $m=\alpha p$, $m$ even. Then there exists a mass-conserving solution $f$ for the equation (\ref{original equation}).
\end{teo}
In our case, this translates to

\begin{table}[H]
    \centering
    \begin{tabular}{|c||c|}
    \hline
     $f_{\epsilon}$ satisfying (\ref{original equation})& $H_{\epsilon}$ satisfying (\ref{diagonal equation})\\
        \hline
        \hline
         existence for $\tau\in[0,T]$, with $T$ small &  existence for $t\in[0,\epsilon\ln(T+1)]$, with $T$ small\\
         \hline
    \end{tabular}
    \caption{Connection between the two equations}
    \label{tab:my_label}
\end{table}

\begin{rmk}\label{remark absence of uniform time}
Notice that the existence from the previous paper translates for our equation (\ref{diagonal equation}) to a time of existence of order $\epsilon$. This means that as $\epsilon\rightarrow 0$ (so in the limit) the time of existence goes to zero. Thus, we prove in this paper existence of mass-conserving solutions for equation (\ref{diagonal equation})  for a time of order one which is independent of $\epsilon$. This will imply existence of mass-conserving solutions for times of order $\mathcal{O}(e^{\frac{1}{\epsilon}})$ to the equation \eqref{original equation}.
\end{rmk}

There is a simple heuristic argument that explains why mass-conserving solutions to equation \eqref{diagonal equation} should exist. Assuming that the coagulation operator gives a negligible contribution, we look at the following PDE (for some suitable test function $\varphi$).
\begin{align*}
\partial_{t}\int H_{\epsilon}(y,v,t)\varphi(y,v,t)\der y \der v-\frac{1}{\epsilon}\int (v^{\alpha}-y)H_{\epsilon}(y,v,t)\partial_{y}\varphi(y,v,t)\der y \der v&=0,\\
H_{\epsilon}(y,v,0)&=H_{\textup{in}}(y,v),
\end{align*}
for which we know the limiting behavior (in some weak sense)
\begin{align}\label{behavior is dirac}
\lim_{\epsilon\rightarrow 0}H_{\epsilon}(y,v,t)&=H_{\textup{in}}(v^{\alpha},v).
\end{align}
We can thus assume formally that solutions to \eqref{diagonal equation} converge to a Dirac measure. More precisely, let us denote the limit of $H_{\epsilon}$ (in whichever sense and if it exists) by $H$ and that $H(y,v,t)=G(v,t)\delta(y-v^{\alpha})$. By the computations above, we obtain that $G$ can be expected to satisfy a coagulation equation in which the coagulation kernel is a  kernel of homogeneity $\gamma-\alpha<1$ for which existence is well-known.

It is also worthwhile to mention that despite the fact that one-dimensional coagulation equations with a diagonal collision kernel have already been the subject of study in the mathematical literature \cite{diagonal, diagonal2, leyvraz,bonacini1,bonacinipeak1,bonacinipeak2}, a motivation of how this type of kernels could appear has not yet been made. In addition, we remember that at the moment there are only few coagulation kernels for which the long-time behavior of solutions is known. These kernels are the constant kernel, the product kernel $K(v,w)=vw$, the sum kernel $K(v,w)=v+w$, the diagonal kernel, and perturbations of the constant kernel of the form $K(v,w)=2+\epsilon W(v,w)$, with $W(v,w)$ continuous, symmetric, bounded, and homogeneous of degree zero. The long-time behavior for general classes of kernels for coagulation equations remains an open problem. We refer to \cite{menonpego} for the first three kernels, to \cite{leyvraz, diagonal, diagonal2} for the diagonal kernel, and to \cite{thromcanizo, throm1} for perturbations of the constant kernel. We thus obtain a first physical intuition on how coagulation models with a diagonal kernel could emerge: as the limiting behavior of  coagulation models describing the sedimentation of particles \eqref{original equation} in the case of very fast sedimentation speed.

We conjecture that this limiting behavior can be proved rigorously. In the present paper, we rigorously prove that the sequence $\{H_{\epsilon}\}$ in \eqref{diagonal equation} converges to $H$ as in \eqref{y v alpha limit form} as $\epsilon\rightarrow 0$. The fact that $G$ in \eqref{y v alpha limit form} satisfies \eqref{first diagonal} will be the study of a future work. Nonetheless, as noticed in Remark \ref{remark absence of uniform time}, there is no result in the present literature proving that there is a uniform time (in $\epsilon$) of existence of mass-conserving solutions for \eqref{diagonal equation}. The first step for proving that $H_{\epsilon}$ converges to $H$ as in \eqref{y v alpha limit form}, with $G$ satisfying \eqref{first diagonal}, is thus to prove existence of mass-conserving solutions of \eqref{diagonal equation} for times of order one. This result is also included in the current work.

\subsection{Connection with coagulation models with fast fusion}\label{fast fusion subsection}
It is also worth to analyze the similarities of the model \eqref{original equation} with coagulation models with fusion.
\begin{align}
\partial_{t}P_{\epsilon}(a, v, t) +&\frac{1}{\epsilon}
\partial_{a}[(c_{0}v^{\frac{2}{3}} - a)P_{\epsilon}(a, v, t)] \nonumber\\
= & \int_{(0,a)} \int_{(0,v)}
K(a-a' , v-v', a' , v')P_{\epsilon}(a-a',v-v', t)P_{\epsilon}(a', v',t)\der v'
\der a' \nonumber\\
&- \int_{(0,\infty)} \int_{(0,\infty)}
K(a , v, a' , v')P_{\epsilon}(a,v, t)P_{\epsilon}(a', v',t)\der v'
\der a'.  \label{fusion equation}
\end{align}
The additional variable $a$ represents in this case the total surface area of the cluster and the transport term $\partial_{a}[(c_{0}v^{\frac{2}{3}} - a)P_{\epsilon}(a, v, t)]$  indicates an evolution of the particles towards a spherical shape. The term $c_{0}v^{\frac{2}{3}} - a$  decreases the surface area of the clusters as long as it is larger than that of a sphere $c_{0}v^{\frac{2}{3}}$. Physically, the transport term describes the merging of two colliding particles up to when they become a sphere and it is hence termed \textit{fusion}.

It was proved in \cite{cristian2023fast} that for the model \eqref{fusion equation} we have  that the sequence $\{P_{\epsilon}\}_{\epsilon\in(0,1)}$ converges to $\overline{P} = P(v, t)\delta(a -c_{0}v^{\frac{2}{3}} )$ as $\epsilon\rightarrow 0$,
with $P$ satisfying the weak formulation of the standard one-dimensional coagulation equation
\begin{align}\label{limit fusion}
\int_{(0,\infty)}& P(v, t)\varphi (v)\der v - \int_{(0,\infty)} P(v, 0)\varphi (v)\der v\\
&=
\int_{0}^{t} \int_{(0,\infty)} \int_{(0,\infty)}
K(c_{0}v^{\frac{2}{3}} , v, c_{0}v'^{\frac{2}{3}} , v')P(v, s)P(v',s)[\varphi(v + v') - \varphi(v) - \varphi(v')]\der v'
\der v \der s, \nonumber
\end{align}
for all appropriately chosen test functions $\varphi$.
A clear similarity can be noticed between the model in \eqref{diagonal equation} and the equation \eqref{fusion equation}. However, the spatial inhomogeneity of the model \eqref{diagonal equation} justifies the presence of the diagonal kernel in the limit \eqref{first diagonal}, as opposed to the model in \eqref{limit fusion}. More precisely,
assuming that the limit of $H_{\epsilon}$, with $H_{\epsilon}$ as in \eqref{diagonal equation}, behaves indeed like a Dirac measure in the space variable, this implies that as we pass to the limit as $\epsilon\rightarrow 0$ in \eqref{diagonal equation}, only particles having the same size can interact due to the space inhomogeneity of the model.

Moreover, the same methods as in \cite{cristian2023fast} cannot be
used to prove the result of this paper. One issue stems from the fact that our results in relation to the model \eqref{fusion equation} are valid for Radon measure solutions. While it is possible to define Radon measure solutions for the model \eqref{fusion equation}, we need $G$ to be continuous in \eqref{first diagonal} in order to rigorously pass to the limit as $\epsilon\rightarrow 0$ in \eqref{diagonal equation}. Nonetheless, comparing the two results offers a clear overview into the understanding of general existence theory for coagulation models. If we work with a coagulation kernel of form $K(v,v')=v^{\gamma}+v'^{\gamma}$, $\gamma>1$, in \eqref{fusion equation}, we obtain in the limit a coagulation kernel of the same homogeneity. This implies that, since solutions do not exist for the model \eqref{limit fusion} when $K(v,v')=v^{\gamma}+v'^{\gamma}$, in \eqref{fusion equation}, we cannot expect solutions of \eqref{fusion equation} to exist in that case (unless we choose a stronger fusion speed). On the other hand, when we choose a coagulation kernel of the form $K(v,v')=v^{\gamma}+v'^{\gamma}$, $1<\gamma<1+\alpha$, in \eqref{diagonal equation}, we obtain in the limit \eqref{first diagonal} a coagulation kernel of homogeneity $\gamma-\alpha<1$ as noticed in Remark \ref{remark homogeneity decreases}. We can thus expect solutions to exist and to behave like a product between a Dirac measure and a function in the volume (or a mollified version of this).

This clearly establishes the following point. The contribution given by a transport term does not suffice to assure existence of mass-conserving solutions of coagulation models (in the case where existence of solutions is not known in the standard case). Rather, both the presence of a transport term and a spatially inhomogeneous model are needed.

\section{Local in time existence}\label{section two}
\subsection{Setting and main results}
\begin{ass}[Assumptions on the initial data]\label{assumption initial condition}
    Let $\epsilon\in(0,1)$, $\alpha\in(0,1)$. Assume that $H_{\epsilon,\textup{in}}\in C^{1}(\mathbb{R}\times(0,\infty))$ and that there exists some constant $A>0$ such that
\begin{align}\label{initial condition decay}
   (1+v^{b}) H_{\epsilon,\textup{in}}(y,v)\leq \frac{A}{1+|y-v^{\alpha}|^{m}},
\end{align}
for all $y\in\mathbb{R}$, $v\in(0,\infty)$, and for sufficiently large $m,b>2$.
\end{ass}
\begin{ass}[Assumptions on the coagulation kernel]\label{assumption kernel} We work with a 
coagulation kernel
\begin{align}\label{upper bound coagulation kernel}
   0\leq  K(v,v')= K(v',v)\leq K_{0}(v^{\gamma}+v'^{\gamma}), \quad 1<\gamma<1+\alpha,
\end{align}
such that
\begin{align}\label{standard condition}
    K(v-v',v')\leq K(v,v'), \textup{ for all } 0\leq v'\leq \frac{v}{2}.
\end{align}
\end{ass}
\begin{rmk}
Condition (\ref{standard condition}) is a rather standard assumption in the study of coagulation equations, see for example \cite{diffusion}, and most of the kernels used in applications satisfy this condition. Notice that Assumption \ref{assumption kernel} implies that our results hold for sum-type kernels of homogeneity $\gamma>1$ which are relevant since there are no solutions to the standard coagulation model for these type of kernels. Moreover, Assumption \ref{assumption kernel} is also valid for sum-type kernels which vanish on the diagonal, such as \eqref{rain kernel intro}, which are used in the physical literature to describe droplet formation and the sedimentation of particles. We refer to \cite{physicspaper} for more details on how these kernels are derived.
\end{rmk}
\begin{defi}[Mild solutions]\label{mild solution definition}
    Let $\epsilon\in(0,1)$, $\alpha\in(0,1)$, $\gamma\in[0,1+\alpha)$, and $m\in\mathbb{N}$, sufficiently large. Let $T>0$ and $K$ satisfy Assumption \ref{assumption kernel}. We say that a non-negative function $H_{\epsilon}\in\textup{C}([0,T];L^{
    \infty}(\mathbb{R}\times (0,\infty)))$ 
is a mild solution of equation (\ref{diagonal equation}) if
\begin{align}\label{mild solution equation}
&H_{\epsilon}(y,v,t)-S_{\epsilon}(t)[H_{\epsilon}(y,v,0) ] D[H_{\epsilon}](y,v,0,t)=\\
&\frac{1}{2}\int_{0}^{t}\int_{(0,v)}D[H_{\epsilon}](y,v,s,t)
S_{\epsilon}(t-s)\bigg[K(v-v',v')H_{\epsilon}(y,v-v',s) H_{\epsilon}(y,v',s)\bigg]\der s,\nonumber 
\end{align}
for all $t\in[0,T]$, $v\in(0,\infty)$  and $y \in\mathbb{R}$, where
\begin{align}\label{definition s mild}
  D[H_{\epsilon}](y,v,s,t):=  \textup{e}^{-\int_{s}^{t}a[H_{\epsilon}](e^{\frac{t-\tau}{\epsilon}}(y-v^{\alpha})+v^{\alpha},v,\tau)\der\tau},
\end{align}
with
\begin{align}\label{general a f def}
a[H_{\epsilon}](y,v,t):=\int_{(0,\infty)}K(v,v')H_{\epsilon}(y,v',t)\der v',
\end{align}
and
\begin{align}
    S_{\epsilon}(s)H_{\epsilon}(y,v):=e^{\frac{s}{\epsilon}}H_{\epsilon}(e^{\frac{s}{\epsilon}}(y-v^{\alpha})+v^{\alpha},v).\label{definition semigroup}
\end{align}
\end{defi}

\begin{defi}\label{mass cons defi sol}
  Let $\epsilon\in(0,1)$. We call $H_{\epsilon}$ a mass-conserving solution of equation  (\ref{diagonal equation}) if $H_{\epsilon}$ is as in Definition \ref{mild solution definition} and satisfies in addition
   \begin{align*}
              \int_{\mathbb{R}}\int_{(0,\infty)}vH_{\epsilon}(y,v,t)\der v\der y= \int_{\mathbb{R}}\int_{(0,\infty)}vH_{\epsilon}(y,v,0)\der v\der y   \textup{ for all } t\in[0,T].
   \end{align*}
\end{defi}
\noindent\textbf{Some notation.} For two quantities $x,y$, we use the notation
$x \approx y$ to say that there exists a constant $M > 1$ such that $\frac{1}{M}y
\leq x \leq M y$.  Additionally, we denote
$x \gg y$ or $y \ll x$ to mean that $M y\leq x$, for a sufficiently large constant $M>1$.

In Subsection \ref{subsection region above}, the dependence on the parameters plays an important role. We thus introduce the following notation which is valid throughout Subsection \ref{subsection region above}. We write $C$ for a numerical constant that can change from line to line but which does not depend on the parameters $\alpha,\gamma,b,m$, where $\alpha,\gamma,b,m$ are as in Assumption \ref{assumption initial condition} and Assumption \ref{assumption kernel}. In addition, we write $C(a_{1},\cdots,a_{n})$ for a constant that can change from line to line which depends on the parameters $a_{1},\cdots,a_{n}$.

In Subsection \ref{proof of existence heuristics} and in Subsection \ref{approx formal subsection}, we offer some formal arguments to justify the proofs in this paper. For this purpose, for two quantities $x$ and $y$, we introduce the notation $x \sim y$  to indicate that $x$ can be approximated at a formal level by $y$. This notation will be used in Subsection \ref{proof of existence heuristics} and in Subsection \ref{approx formal subsection}.

Let $A\geq 1$. Let $j\in\{1,2\}$. We denote by 
\begin{align}
    M_{j} &:=M_{j}(\alpha,\gamma,b,m), \textup{ a sufficiently large but fixed constant which depends on } \alpha,\gamma,b,m;\label{defm2}\\
    C_{0} &:=C_{0}(A), \textup{ a constant  such that } C_{0}^{m-1}:=M_{1}(\alpha,\gamma,b,m)A\label{defc0}.
\end{align}
We remark that $M_{2}$ is chosen as in \eqref{def m2} and that $M_{1}$ depends on $M_{2}$ through \eqref{choice of m1}. We then define the functions $\psi, T_{1},T_{2},$ and $T_{3}$ as below.
\begin{align}
        \psi(z)&=\frac{1}{1+|z|^{m}};\label{psi definition}\\
  T_{1}(y,v,t)&:=  \frac{2A}{1+v^{b}}e^{\frac{t}{\epsilon}}\psi(e^{\frac{t}{\epsilon}}(y-v^{\alpha}))     \chi_{1}(y,v,t);\label{def t1}\\
   T_{2}(y,v,t)&:=\frac{2M_{1}A^{3} \epsilon}{(1+v^{b})|y-v^{\alpha}|} \chi_{2}(y,v,t);\label{def t2}\\
     T_{3}(y,v,t)&:=  \frac{M_{2}A}{1+v^{b}}e^{\frac{t}{\epsilon}}\psi(e^{\frac{t}{\epsilon}}(y-v^{\alpha}))  \chi_{3}(y,v,t);\label{def t3}\\
     \chi_{1}(y,v,t)&:=\mathbbm{1}_{\{y\geq v^{\alpha}-C_{0}^{-1}\epsilon^{-\frac{1}{m-1}}  e^{-\frac{t}{\epsilon}}\}};\label{chi1}\\
   \chi_{2}(y,v,t)&:= \mathbbm{1}_{\{(\frac{v}{3})^{\alpha}\leq y\leq v^{\alpha}-C_{0}^{-1}\epsilon^{-\frac{1}{m-1}}  e^{-\frac{t}{\epsilon}}\}} ;\label{chi2}\\
     \chi_{3}(y,v,t)&:=\mathbbm{1}_{\{y\leq (\frac{v}{3})^{\alpha}\}}.\label{chi3}
\end{align}
Moreover, we define
 \begin{align}\label{def overline d}
 \overline{d}:=2\bigg\lceil \frac{1}{2}\bigg(\frac{\alpha}{\alpha+1-\gamma}-1\bigg)\bigg\rceil,
 \end{align} where  $\lceil\cdot \rceil$ denotes the ceiling function. We then define  $\overline{b}(\gamma,\alpha)$ to be 
 \begin{align}\label{definition overline b}
     \overline{b}(\gamma,\alpha)=\alpha \overline{d}+2.
 \end{align} 
\begin{rmk}
Notice that
\begin{align*}
 T_{2}(y,v,t):=\frac{2C_{0}^{m-1}A^{2} \epsilon}{(1+v^{b})|y-v^{\alpha}|} \chi_{2}(y,v,t).
\end{align*}
The reason for introducing $C_{0}(A)$ is since the cut-offs $\chi_{1}, \chi_{2}$ in \eqref{chi1}-\eqref{chi2} (and regions included in the respective sets) will play an important role in the estimates. We refer to \eqref{c0 initial mention} and \eqref{most important c0 dependence}. Thus, we introduce $C_{0}$ in order to simplify the notation.
\end{rmk}
The main result of this paper is as follows.
\begin{teo}\label{proposition existence via fixed point}
   Let $H_{\epsilon,\textup{in}}$ be as in Assumption \ref{assumption initial condition} and $K$ as in Assumption \ref{assumption kernel}. Let $b\geq \max\{\overline{b}(\gamma,\alpha), $ $2\gamma+1\}$ and $m>\max\{\frac{2(\gamma+1)}{\alpha},\frac{b}{\alpha}+1\}$, where  $\overline{b}(\gamma,\alpha)$ is as in \eqref{definition overline b}. There exist two constants $M_{1}(\alpha,\gamma,b,m), M_{2}(\alpha,\gamma,b,m),$ depending only on $\alpha,\gamma,b,m,$ such that for any $A\geq 1$, there exist $\overline{\epsilon}\in(0,1)$, sufficiently small, and  $T>0$, sufficiently small, which is independent of $\epsilon\in(0,1),$ such that  for all $\epsilon\in(0,\overline{\epsilon}]$ there exists a mass-conserving mild solution $H_{\epsilon}\in\textup{C}([0,T];L^{\infty}(\mathbb{R}\times(0,\infty)))$ to (\ref{diagonal equation}) as in Definition \ref{mass cons defi sol}. Moreover, we have that $H_{\epsilon}$ satisfies the following estimates for $C_{0}$ as in \eqref{defc0}.
   
   If $e^{\frac{t}{\epsilon}}v^{\alpha}\big[1-\frac{1}{3^{\alpha}}\big]\geq  C_{0}^{-1}\epsilon^{-\frac{1}{m-1}} $, then
    \begin{align}\label{existence space fixed point}
   0\leq &(1+v^{b}) H_{\epsilon}(y,v,t)\leq  T_{1}(y,v,t)+T_{2}(y,v,t)+T_{3}(y,v,t).
    \end{align}
Otherwise, if  $e^{\frac{t}{\epsilon}}v^{\alpha}\big[1-\frac{1}{3^{\alpha}}\big]\leq  C_{0}^{-1}\epsilon^{-\frac{1}{m-1}} $, then
\begin{align}
   0\leq (1+v^{b}) H_{\epsilon}(y,v,t)\leq 2Ae^{\frac{t}{\epsilon}}\psi(e^{\frac{t}{\epsilon}}(y-v^{\alpha})).
    \end{align}
\end{teo}
The bounds in \eqref{existence space fixed point} can be explained as follows. Since our functions $H_{\epsilon}$, $\epsilon\in(0,1)$, approach a Dirac measure supported on the curve $\{y=v^{\alpha}\}$ as $\epsilon\rightarrow 0$, we expect that most of the mass of the solution is contained in the region $|y-v^{\alpha}|\ll 1$. Outside this region, we expect the function to have small values and to approach $0$ as $\epsilon\rightarrow 0$. In other words, since $H_{\epsilon}$ can be thought of as a mollifier approaching the region $\{y=v^{\alpha}\}$, most of the mass of $H_{\epsilon}$ will be contained in  the region where $e^{\frac{t}{\epsilon}}|y-v^{\alpha}|\ll 1$. In other words, our choice of space is a rescaled version of the function
\begin{align*}
\frac{1}{1+|y-v^{\alpha}|^{m}}\mathbbm{1}_{\{|y-v^{\alpha}|\leq 1\}}+\frac{\epsilon}{|y-v^{\alpha}|}\mathbbm{1}_{\{|y-v^{\alpha}|\geq 1\}}.
\end{align*}

\begin{rmk}\label{remark time change depending on the mass}
If we denote by
    \begin{align*}
        T_{\textup{min}}:=\sup_{t\geq 0}\bigg\{\int_{\mathbb{R}}\int_{(0,\infty)}vH_{\epsilon}(y,v,t)\der v \der y=\int_{\mathbb{R}}\int_{(0,\infty)}vH_{\epsilon}(y,v,0)\der v \der y \bigg\},
    \end{align*}
    then $T_{\textup{min}}>0$ and $T_{\textup{min}}$ is independent of the choice of $\epsilon\in(0,\overline{\epsilon}),$ where $\overline{\epsilon}$ is as in Theorem \ref{proposition existence via fixed point}. In addition, the choice of $T_{\textup{min}}$ and $\overline{\epsilon}$ depends on the total mass of particles at initial time through the choice of $A$. For the exact choices of $T_{\textup{min}}$  and $\overline{\epsilon}$, we refer to the Proof of Proposition \ref{main ingredient}.
\end{rmk}
\begin{rmk}\label{we need function to be even minimum recovered}
 Notice that $\overline{d}$ in \eqref{def overline d} is even and $\gamma\leq \frac{\overline{d}\alpha}{\overline{d}+1}+1$. This will be needed in Proposition \ref{der v g negative} and in Proposition \ref{main statement region below}.
\end{rmk}
\begin{rmk}
As already mentioned in Remark \ref{remark absence of uniform time}, Theorem \ref{proposition existence via fixed point}  yields existence of mass-conserving solutions for times of order $\mathcal{O}(e^{\frac{1}{\epsilon}})$ to the equation \eqref{original equation}. We refer to Appendix \ref{appendix rescaled version of the results} for more details. 
\end{rmk}

We summarize below the results in \cite{cristianinhom} and Theorem \ref{proposition existence via fixed point}. We notice that the order of the time of existence can be improved if the upper bound of $H_{\epsilon}$ contains instead the error term $T_{2}$ as defined in \eqref{def t2}. While $T_{2}(y,v,t)$ is larger than $\frac{Ae^{\frac{t}{\epsilon}}\chi_{2}(y,v,t)}{(1+v^{b})(1+e^{\frac{tm}{\epsilon}}|y-v^{\alpha}|^{m})}$, it still vanishes as $\epsilon\rightarrow 0$ when $y$ and $v$ are supported away from the line $\{y=v^{\alpha}\}$.
\begin{table}[H]
    \centering
    \begin{tabular}{|c||c|}
    \hline
  Results in \cite{cristianinhom}& Theorem \ref{proposition existence via fixed point}\\
        \hline
        \hline
       existence for $t\in[0,\epsilon\ln(\tilde{T}+1)]$, with $\tilde{T}$ small & existence for $t\in[0,T]$, with $T$ small \\
         \hline
        $H_{\epsilon,\textup{in}}(y,v)\leq \frac{A}{1+|y|^{m}+v^{p}}$ &     $H_{\epsilon,\textup{in}}(y,v)\leq \frac{A}{(1+v^{b})(1+|y-v^{\alpha}|^{m})}$\\
        \hline
              $H_{\epsilon}(y,v,t)\leq \frac{C Ae^{\frac{t}{\epsilon}}}{1+|e^{\frac{t}{\epsilon}}(y-v^{\alpha})+v^{\alpha}|^{m}+v^{p}},$ for some $C>0$ &     $H_{\epsilon}(y,v,t)\leq \frac{C Ae^{\frac{t}{\epsilon}}[\chi_{1}(y,v,t)+\chi_{3}(y,v,t)]}{(1+v^{b})(1+e^{\frac{tm}{\epsilon}}|y-v^{\alpha}|^{m})}+T_{2}$\\ \hline
    \end{tabular}
    \caption{Results in \cite{cristianinhom} in comparison to Theorem \ref{proposition existence via fixed point}}
    \label{tab:my_label two}
\end{table}

\begin{rmk} Notice that if in \eqref{initial condition decay} we have that
 $(1+v^{b}) H_{\epsilon,\textup{in}}(y,v)\leq \frac{B}{1+|y-v^{\alpha}|^{m}}$, with $B<1$, then $(1+v^{b}) H_{\epsilon,\textup{in}}(y,v)\leq \frac{1}{1+|y-v^{\alpha}|^{m}}$ and we can apply Theorem \ref{proposition existence via fixed point} with $A=1$ in order to obtain existence of mass-conserving solutions in this case. 
 \end{rmk}
 
Actually, we can argue in a similar manner as in the case $A\geq 1$ in order to obtain that if $(1+v^{b}) H_{\epsilon,\textup{in}}(y,v)\leq \frac{A}{1+|y-v^{\alpha}|^{m}}$, with $A<1$, then the following holds. For $C_{0}^{m-1} =M_{1}(\alpha,\gamma,b,m)$, if $e^{\frac{t}{\epsilon}}v^{\alpha}\big[1-\frac{1}{3^{\alpha}}\big]\geq  C_{0}^{-1}\epsilon^{-\frac{1}{m-1}} $, then
    \begin{align}\label{estimate for a smaller than one}
 (1+v^{b}) H_{\epsilon}(y,v,t)\leq  T_{1}(y,v,t)+T_{3}(y,v,t)+\frac{2C_{0}^{m-1} \epsilon}{(1+v^{b})|y-v^{\alpha}|} \chi_{2}(y,v,t),
    \end{align}
where $T_{1}, T_{3}$ are as in \eqref{def t1} and \eqref{def t3}. Otherwise, if  $e^{\frac{t}{\epsilon}}v^{\alpha}\big[1-\frac{1}{3^{\alpha}}\big]\leq  C_{0}^{-1}\epsilon^{-\frac{1}{m-1}} $, then
\begin{align}
   0\leq (1+v^{b}) H_{\epsilon}(y,v,t)\leq 2Ae^{\frac{t}{\epsilon}}\psi(e^{\frac{t}{\epsilon}}(y-v^{\alpha})).
    \end{align}

Notice that this estimate provides accurate information about the distribution of mass of the solution $H_{\epsilon}(y,v,t)$. The term in \eqref{estimate for a smaller than one} which plays a role similar to $T_{2}$ does not contain any dependence on $A$. Therefore, if $A$ is very small, it yields a contribution that does not vary depending on $A$. 

Nevertheless, as explained above,  $T_{2}$ is to be understood as an error term since it vanishes as $\epsilon\rightarrow 0$ when $y$ and $v$ are supported away from the line $\{y=v^{\alpha}\}$. Thus, despite the fact that $T_{2}$ loses the dependence on $A$ when $A<1$, it has a small contribution. In other words, we can consider at a formal level that we stay close to the mass at initial time also in the case $A<1$.

Nonetheless, as stated in Remark \ref{remark time change depending on the mass}, the time of existence varies depending on $A$ when $A\geq 1$.  Because $T_{2}$ does not depend on $A$ when $A<1$, the minimal time of existence will not vary depending on $A$ when $A<1$. More precisely, for $T_{\textup{min}}$ as in Remark \ref{remark time change depending on the mass}, we have that $T_{\textup{min}}$ when $A<1$ is equal to $T_{\textup{min}}$ when $A=1$.

\vspace{0.2cm}

As conjectured in the Introduction, we obtain that the solutions $H_{\epsilon}$ found in Theorem \ref{proposition existence via fixed point} converge to a Dirac measure.
\begin{teo}\label{teo convergence}
  Let $H_{\epsilon,\textup{in}}$ be as in Assumption \ref{assumption initial condition} and $K$ as in Assumption \ref{assumption kernel}. Let $T>0$ and $H_{\epsilon}\in\textup{C}([0,T];L^{\infty}(\mathbb{R}\times(0,\infty)))$ be as in Theorem \ref{proposition existence via fixed point}. Let $\sigma>0$ and assume in addition that $b>\gamma+3$. There exists a subsequence (which we do not relabel) and $H\in\textup{C}([\sigma,T];  \mathscr{M}_{+}(\mathbb{R}\times(0,\infty)))$, where we denoted by $  \mathscr{M}_{+}(\mathbb{R}\times(0,\infty))$ the space of non-negative Radon measures, such that
\begin{align*}
\int_{\mathbb{R}}\int_{(0,\infty)}H_{\epsilon}(y,v,t)\varphi(y,v)\der v\der y\rightarrow \int_{\mathbb{R}}\int_{(0,\infty)}H(y,v,t)\varphi(y,v)\der v\der y,
\end{align*}
as $\epsilon\rightarrow 0$, for every $t\in[\sigma,T]$, and every $\varphi\in\textup{C}_{0}(\mathbb{R}\times(0,\infty))$. Moreover, for every $t\in[\sigma,T]$, $H$ has the form 
\begin{align}\label{form of limit}
H(y,v,t)=\overline{H}(v,t)\delta(y-v^{\alpha})
\end{align}
in the sense of measures. More precisely, there exists $\overline{H}\in\textup{C}([\sigma,T];\mathscr{M}_{+}((0,\infty)))$ such that
\begin{align*}
\int_{\mathbb{R}}\int_{(0,\infty)}H(y,v,t)\varphi(y,v)\der v \der y=\int_{\mathbb{R}}\int_{(0,\infty)}\overline{H}(v,t)\delta(y-v^{\alpha})\varphi(y,v)\der v\der y,
\end{align*}
for every $t\in[\sigma,T]$ and $\varphi\in\textup{C}_{0}(\mathbb{R}\times(0,\infty))$.
\end{teo}
However, it is not clear whether the limit $\overline{H}$ satisfies a one-dimensional coagulation equation with a diagonal kernel as in \eqref{first diagonal}. From the bound in \eqref{existence space fixed point}, we are able to prove that the sequence $\{H_{\epsilon}\}_{\epsilon\in(0,1)}$ is a uniformly bounded sequence in $L^{1}$, but it is not uniformly bounded in $L^{\infty}$. This, together with the spatial inhomogeneity of the model \eqref{diagonal equation}, will create difficulties when trying to prove that the marginals $\int_{\mathbb{R}} H_{\epsilon}(y,v,t)\der y$ are uniformly equicontinuous in the $v$ variable. As such, we do not know whether $\overline{H}$ is continuous in $v$ in order to rigorously derive \eqref{first diagonal}.


\subsection{Justification of the main results}\label{proof of existence heuristics}
\subsubsection*{Existence of solutions - Theorem \ref{proposition existence via fixed point}}
We first offer some insight into the proof of existence stated in Theorem \ref{proposition existence via fixed point} and the choice of $T_{1},T_{2}, T_{3}$ in \eqref{def t1}-\eqref{def t3}. We start by approximating the contribution of the gain term in the coagulation kernel assuming that our solutions behave like Dirac measures supported on the curve $\{y=v^{\alpha}\}$ with sufficiently fast decay for large values of $v$. The decay in $v$ is so that we obtain mass-conserving solutions. More precisely, we have that
\begin{align*}
    v^{\gamma}&\int_{0}^{v}\frac{1}{1+w^{b}}\delta(y-w^{\alpha})\frac{1}{1+(v-w)^{b}}\delta(y-(v-w)^{\alpha})\der w\\
    &\approx \frac{v^{\gamma+1-\alpha}}{(1+v^{b})^{2}}\delta(y-(\frac{v}{2})^{\alpha})\int_{0}^{v}\delta(w-\frac{v}{2})\der w=\frac{v^{\gamma+1-\alpha}}{(1+v^{b})^{2}}\delta(y-(\frac{v}{2})^{\alpha}).
\end{align*}
Applying the semigroup action to the obtained term, we have that
\begin{align*}
    \int_{0}^{t}S_{\epsilon}(t-s)\bigg[\frac{v^{\gamma+1-\alpha}}{(1+v^{b})^{2}}\delta(y-(\frac{v}{2})^{\alpha})\bigg]\der s&= \frac{v^{\gamma+1-\alpha}}{(1+v^{b})^{2}}\int_{0}^{t}e^{\frac{t-s}{\epsilon}}\delta(e^{\frac{t-s}{\epsilon}}(y-v^{\alpha})+v^{\alpha}-(\frac{v}{2})^{\alpha})\der s\\
    &= \frac{v^{\gamma+1-\alpha}}{(1+v^{b})^{2}}\int_{0}^{t}e^{\frac{s}{\epsilon}}\delta(e^{\frac{s}{\epsilon}}(y-v^{\alpha})+v^{\alpha}-(\frac{v}{2})^{\alpha})\der s\\
       &= \epsilon \frac{v^{\gamma+1-\alpha}}{(1+v^{b})^{2}}\int_{0}^{\frac{t}{\epsilon}}e^{s}\delta(e^{s}(y-v^{\alpha})+v^{\alpha}-(\frac{v}{2})^{\alpha})\der s.
\end{align*}
Making the change of variables $z=e^{s}(y-v^{\alpha})+v^{\alpha}-(\frac{v}{2})^{\alpha}$, we further obtain that
\begin{align*}
    \int_{0}^{t}S_{\epsilon}(t-s)\bigg[\frac{v^{\gamma+1-\alpha}}{(1+v^{b})^{2}}\delta(y-(\frac{v}{2})^{\alpha})\bigg]\der s&= \frac{\epsilon}{v^{\alpha}-y} \frac{v^{\gamma+1-\alpha}}{(1+v^{b})^{2}}\int_{e^{\frac{t}{\epsilon}}(y-v^{\alpha})+v^{\alpha}-(\frac{v}{2})^{\alpha}}^{y-(\frac{v}{2})^{\alpha}}\delta(z)\der z\\
    &= \frac{\epsilon}{v^{\alpha}-y} \frac{v^{\gamma+1-\alpha}}{(1+v^{b})^{2}}\mathbbm{1}_{\{(\frac{v}{2})^{\alpha}\leq y \leq v^{\alpha}-e^{-\frac{t}{\epsilon}}(v^{\alpha}-(\frac{v}{2})^{\alpha})\}}.
\end{align*}
    This explains the choice of the term $T_{2}$ in \eqref{def t2}. 
\subsubsection*{Choice of the estimate in \eqref{existence space fixed point}}
 As noticed in the previous computations the gain term of the coagulation operator shifts the support from $\{y=v^{\alpha}\}$ to $\{y=(\frac{v}{2})^{\alpha}\}$ if we apply it on a Dirac mass $\delta(y-v^{\alpha})$. For convenience, we  approximate the term $\frac{\epsilon}{v^{\alpha}-y} \frac{v^{\gamma+1-\alpha}}{(1+v^{b})^{2}}\mathbbm{1}_{\{(\frac{v}{2})^{\alpha}\leq y \leq v^{\alpha}-e^{-\frac{t}{\epsilon}}(v^{\alpha}-(\frac{v}{2})^{\alpha})\}}$ by $\frac{\epsilon}{v^{\alpha}-y} \frac{v^{\gamma+1-\alpha}}{(1+v^{b})^{2}}\mathbbm{1}_{\{(\frac{v}{2})^{\alpha}\leq y \leq v^{\alpha}\}}$ in the following. However, the fact that $T_{2}$ is at a distance of order $v^{\alpha}e^{-\frac{t}{\epsilon}}$ from the curve $\{y=v^{\alpha}\}$ will be important in order to avoid unbounded terms. Because of the shift of the support, it is natural to look for a decomposition of the form $H_{\epsilon}=h_{0}+h_{1}+\ldots$, with $h_{0}(y,v,t)=g(v,t)\delta(y-v^{\alpha})$ and each $h_{i}, i\in\mathbb{N}$, supported in a region of the form $y\in[(\frac{v}{2^{i}})^{\alpha},(\frac{v}{2^{i-1}})^{\alpha})$, with a decay in $v$ of larger order, and $h_{i}\approx\mathcal{O}(\epsilon^i)$. It then holds that
\begin{align*}
    \epsilon\partial_{t}g(v,t)\delta(y-v^{\alpha})+\partial_{y}((v^{\alpha}-y)h_{1})&=\epsilon v^{\gamma}\int_{0}^{\frac{v}{2}}g(v-w,t)\delta(y-(v-w)^{\alpha})g(w,t)\delta(y-w^{\alpha})\der w\\
    &\approx\epsilon v^{\gamma+1-\alpha} \bigg(g\big(\frac{v}{2},t\big)\bigg)^{2}\delta(y-(\frac{v}{2})^{\alpha}).
\end{align*}
Thus, we have that
\begin{align*}
    \partial_{y}((v^{\alpha}-y)h_{1})
    =\epsilon v^{\gamma+1-\alpha} \bigg(g\big(\frac{v}{2},t\big)\bigg)^{2}\delta(y-(\frac{v}{2})^{\alpha}),
\end{align*}
which implies that
\begin{align*}
(v^{\alpha}-y)h_{1}
    = \epsilon v^{\gamma+1-\alpha} \bigg(g\big(\frac{v}{2},t\big)\bigg)^{2}\int_{-\infty}^{y}\delta(\xi-(\frac{v}{2})^{\alpha})\der \xi
\end{align*}
or in other words
\begin{align*}
h_{1}
    =\frac{\epsilon v^{\gamma+1-\alpha} }{(v^{\alpha}-y)}\bigg(g\big(\frac{v}{2},t\big)\bigg)^{2}\mathbbm{1}_{\{(\frac{v}{2})^{\alpha}\leq y\leq v^{\alpha}\}}.
\end{align*}
Thus it holds that   $h_{0}(y,v,t)=g(v,t)\delta(y-v^{\alpha})$ and $h_{1}(y,v,t)= \frac{\epsilon v^{\gamma+1-\alpha} }{(v^{\alpha}-y)}\bigg(g\big(\frac{v}{2},t\big)\bigg)^{2}\mathbbm{1}_{\{(\frac{v}{2})^{\alpha}\leq y\leq v^{\alpha}\}}$. We further obtain that
\begin{align*}
   v^{\gamma} & \int_{0}^{v}h_{0}(y,v-w,t)h_{1}(y,w,t)\der w\\
   &= \epsilon v^{\gamma}\int_{0}^{v}g(v-w,t)\delta(y-(v-w)^{\alpha}) w^{\gamma+1-\alpha}\frac{(g(\frac{w}{2},t))^{2}}{w^{\alpha}-y}\mathbbm{1}_{\{(\frac{w}{2})^{\alpha}\leq y\leq w^{\alpha}\}}\der w.
\end{align*}
Since if $y=(v-w)^{\alpha}$ and $(\frac{w}{2})^{\alpha}\leq y\leq w^{\alpha}$ implies that $\frac{2v}{3}\geq w\geq \frac{v}{2}$ and thus $\frac{v}{3}\leq v-w\leq \frac{v}{2}$, we have that
\begin{align*}
   v^{\gamma} \int_{0}^{v}h_{0}(y,v-w,t)h_{1}(y,w,t)\der w\approx  \frac{\epsilon v^{2(\gamma+1-\alpha)} g(\frac{v}{3},t)(g(\frac{v}{4},t))^{2}}{((v-y^{\frac{1}{\alpha}})^{\alpha}-y)}\mathbbm{1}_{\{(\frac{v}{3})^{\alpha}\leq y\leq (\frac{v}{2})^{\alpha}\}}.
\end{align*}
Thus, arguing as before, it follows that
\begin{align*}
    \partial_{y}((v^{\alpha}-y)h_{2})&=  \frac{\epsilon^{2} v^{2(\gamma+1-\alpha)} g(\frac{v}{3},t)(g(\frac{v}{4},t))^{2}}{((v-y^{\frac{1}{\alpha}})^{\alpha}-y)}\mathbbm{1}_{\{(\frac{v}{3})^{\alpha}\leq y\leq (\frac{v}{2})^{\alpha}\}}\\
    &\sim \frac{\epsilon^{2} v^{2(\gamma+1-\alpha)} g(\frac{v}{3},t)(g(\frac{v}{4},t))^{2}}{(\frac{v}{2})^{\alpha}-y}\mathbbm{1}_{\{(\frac{v}{3})^{\alpha}\leq y\leq (\frac{v}{2})^{\alpha}\}},
\end{align*}
which by integrating in $y$ implies that
\begin{align}\label{div int}
h_{2}=  \frac{\epsilon^{2} v^{2(\gamma+1-\alpha)} g(\frac{v}{3},t)(g(\frac{v}{4},t))^{2}}{(v^{\alpha}-y)}\int_{-\infty}^{y}\frac{1}{(\frac{v}{2})^{\alpha}-\xi}\mathbbm{1}_{\{(\frac{v}{3})^{\alpha}\leq \xi\leq (\frac{v}{2})^{\alpha}\}}.
\end{align}
The integral will give a logarithmic contribution which explodes at $y=(\frac{v}{2})^{\alpha}$. However, due to the cut-off in $h_{1}$ noticed above, namely $\mathbbm{1}_{\{(\frac{v}{2})^{\alpha}\leq y \leq v^{\alpha}-e^{-\frac{t}{\epsilon}}(v^{\alpha}-(\frac{v}{2})^{\alpha})\}}$, the integral turns out to be convergent. Taking into account the effect of this cut-off, we are able to estimate the integral in \eqref{div int} as  $C\big(\frac{t}{\epsilon}+\log(v^{\alpha})\big)$ and thus
\begin{align*}
h_{2}\approx  \frac{t \epsilon v^{2(\gamma+1-\alpha)} g(\frac{v}{3},t)(g(\frac{v}{4},t))^{2}}{(v^{\alpha}-y)}+\frac{ \epsilon^{2} v^{2(\gamma+1-\alpha)} g(\frac{v}{3},t)(g(\frac{v}{4},t))^{2}}{(v^{\alpha}-y)}\log(v^{\alpha}).
\end{align*}
Then, by choosing a sufficiently good decay in $v$ for large values of $v$ for the function $g(v,t)$ and for times $t$ of order one, we should be able to approximate $h_{2}$ for $v\geq 1$ by
\begin{align*}
h_{2}\sim  \frac{ \epsilon v^{2(\gamma+1)} g(\frac{v}{3},t)(g(\frac{v}{4},t))^{2}}{(v^{\alpha}-y)}.
\end{align*}
 This is the expected choice for the estimate in \eqref{existence space fixed point} and explains the choice of the truncation in $T_{1}$ and $T_{2}$ in \eqref{def t1}-\eqref{def t2}.

Another problem that will arise when passing to the rigorous computations is the control of the decay in $v$ in order to obtain mass-conserving solutions. We remark that if the function $g$ is chosen to have sufficiently fast decay for large $v$, then $h_{0}, h_{1}, h_{2}$ also have sufficiently fast decay in $v$. However, the terms of the series become difficult to study after multiple iterations. Moreover, our functions will behave like mollifiers of a Dirac measure supported on the curve $\{y=v^{\alpha}\}$ and because of this their support will become larger, creating the possibility of losing the decay in $v$ after their convolution. More precisely, for $\frac{y}{v^{\alpha}}\ll 1$, we may lose the decay in $v$. Nonetheless, we notice that for $\frac{y}{v^{\alpha}}\ll 1$, our equation can be approximated by
\begin{align}
&\partial_{t}H_{\epsilon}(y,v,t)+\frac{1}{\epsilon}\partial_{y}[(v^{\alpha}-y)H_{\epsilon}(y,v,t)]-\mathbb{K}[H_{\epsilon}](y,v,t)\label{approx eq part one}\\
&\sim \partial_{t}H_{\epsilon}(y,v,t)+\frac{1}{\epsilon}v^{\alpha}\partial_{y}H_{\epsilon}(y,v,t)-\mathbb{K}[H_{\epsilon}](y,v,t).\label{approximated equation}
\end{align}
For fixed $\epsilon$, existence of the model \eqref{approximated equation} has been studied in  \cite{cristianinhom}. More precisely, in this case the coagulation operator can be approximated by a transport term in the $v$ variable and one can prove that $G_{\epsilon}$ is a supersolution of the model \eqref{approximated equation}, where $G_{\epsilon}$ satisfies the following PDE.
\begin{align}\label{approx supersol}
    \partial_{t}G_{\epsilon}(y,v,t)+\frac{1}{\epsilon}v^{\alpha}\partial_{y}G_{\epsilon}(y,v,t)+\frac{L v^{\gamma} }{1+|y|^{d}}\partial_{v}G_{\epsilon}(y,v,t)=0,
\end{align}
for some suitably chosen $d\in\mathbb{N}$ and $L>0$. For completeness, we present the idea behind passing from \eqref{approximated equation} to \eqref{approx supersol} in Subsection \ref{approx formal subsection}. This idea was also used in \cite[Section 2.1]{cristianinhom} and a rigorous proof can be found in Subsection \ref{subsection approximation} of this paper.

Using \eqref{approx supersol}, the precise form of $G_{\epsilon}$ can then be studied since the associated characteristics
to \eqref{approx supersol} are given by
 \begin{equation}
\left\{\begin{aligned}\label{characteristics approx equation}
\partial_{t}{Y_{\epsilon}}(y,v,t)&=-\frac{1}{\epsilon}{V_{\epsilon}}^{\alpha}, & \qquad {Y_{\epsilon}}(y,v,0)&=y, \\
\partial_{t}{V_{\epsilon}}(y,v,t)&=-\frac{{L}{V_{\epsilon}}^{\gamma}}{1+|{Y_{\epsilon}}|^{d}},& \qquad {V_{\epsilon}}(y,v,0)&=v.
   \end{aligned}\right.
   \end{equation}
  It immediately follows from \eqref{characteristics approx equation} that we can use separation of variables to obtain more information about the form of the characteristics since
  \begin{align*}
      \frac{\der V_{\epsilon}}{\der Y_{\epsilon}}=\frac{\epsilon L V_{\epsilon}^{\gamma-\alpha}}{1+|Y_{\epsilon}|^{d}}\Rightarrow V_{\epsilon}^{-\gamma+\alpha}\der V_{\epsilon}=\frac{\epsilon L \der Y_{\epsilon}}{1+|Y_{\epsilon}|^{d}}.
  \end{align*}
 However, for a rigorous proof, a supersolution of \eqref{approx eq part one} has the following form (up to minor modifications)
 \begin{align}\label{supersol formal example}
    \partial_{t}G_{\epsilon}(y,v,t)+\frac{1}{\epsilon}(v^{\alpha}-y)\partial_{y}G_{\epsilon}(y,v,t)+\frac{L v^{\gamma} }{1+|y|^{d}}\partial_{v}G_{\epsilon}(y,v,t)=0. \end{align} The associated characteristics for \eqref{supersol formal example} satisfy
  \begin{align}\label{new char}
      \frac{\der V_{\epsilon}}{\der Y_{\epsilon}}=\frac{\epsilon L V_{\epsilon}^{\gamma}}{(1+|Y_{\epsilon}|^{d})(V_{\epsilon}^{\alpha}-Y_{\epsilon})}
  \end{align}
  and it is not possible to solve \eqref{new char} using the separation of variables and implicitly we are not able to reproduce the methods in \cite{cristianinhom}. Nonetheless, we will prove that for  $\frac{y}{v^{\alpha}}\ll 1$, namely for $y\leq (\frac{v}{3})^{\alpha}$, the associated characteristics behave similarly to the ones in \eqref{characteristics approx equation}. In particular, the characteristics do not go to infinity in finite time. This explains the choice of $T_{3}$ in \eqref{def t3}.
  \subsubsection*{Convergence to a Dirac measure - Theorem \ref{teo convergence}}
  We now present some justification for the proof of Theorem \ref{teo convergence}. The solution found in Theorem 
\ref{proposition existence via fixed point} satisfies the following upper bound,
  \begin{align}\label{bound h epsilon equicontinuity formal}
      H_{\epsilon}(y,v,t)\leq T_{1}(y,v,t)+T_{2}(y,v,t)+T_{3}(y,v,t), \textup{ with } T_{1},T_{2},T_{3} \textup{ as in } \eqref{def t1}-\eqref{def t3}.
  \end{align}
With this upper bound, it is straightforward to prove that $H_{\epsilon}$ vanishes in the region $\{|y-v^{\alpha}|\geq \delta\}$, for some $\delta>0$. More precisely, we prove that if $t>0$, $|y-v^{\alpha}|\geq \delta$, it holds that
\begin{align}\label{solution vanishes outside dirac}
 H_{\epsilon}(y,v,t)\leq C(\delta)[\epsilon+e^{-\frac{t(m-1)}{\epsilon}}]\rightarrow 0 \textup{ as } \epsilon\rightarrow 0.
\end{align}
If in addition to \eqref{solution vanishes outside dirac}, $H_{\epsilon}$ is also equicontinuous in time, then Theorem \ref{teo convergence} follows. At a first glance, the transport term in \eqref{diagonal equation} may bring a contribution that is too large because of the presence of $\frac{1}{\epsilon}$ and a more detailed analysis of the equation is then required in order to prove equicontinuity in time.

Assume we are at a time $t\geq\sigma>0$. At a formal level, $H_{\epsilon}$ behaves like a mollifier of a Dirac measure at $\{y=v^{\alpha}\}$ with sufficiently fast decay in the $v$ variable. In other words, the part of $H_{\epsilon}$ supported in the same region as $T_{1}$  in \eqref{def t1} contains most of the mass. We test with $\varphi$ in \eqref{diagonal equation}  and integrate in $y$ and $v$. It follows that
\begin{align*}
    \bigg|\int_{(0,\infty)}\int_{\mathbb{R}}\frac{e^{\frac{t}{\epsilon}}(y-v^{\alpha})}{(1+v^{b})(1+e^{\frac{tm}{\epsilon}}|y-v^{\alpha}|^{m})}\partial_{y}\varphi(y,v) \der y\der v\bigg|&\leq Ce^{-\frac{t}{\epsilon}} \int_{(0,\infty)}\frac{\der v}{1+v^{b}}\int_{\mathbb{R}}\frac{|\xi|\der \xi}{1+|\xi|^{m}}\\
    &\leq Ce^{-\frac{\sigma}{\epsilon}}.
    \end{align*}
In other words, the transport term in \eqref{diagonal equation} behaves like
\begin{align*}
    \frac{1}{\epsilon}\int_{s}^{t}\int_{\mathbb{R}}\int_{(0,\infty)}(v^{\alpha}-y)H_{\epsilon}(y,v,z)\partial_{y}\varphi(y,v)\der v \der y \der z\sim \frac{Ce^{-\frac{\sigma}{\epsilon}}|t-s|}{\epsilon}\leq C|t-s|.
\end{align*}
for all $t,s\geq\sigma>0$. This offers a sufficiently good bound in order to obtain the desired equicontinuity in time.  Moreover, the contribution due to the coagulation operator can be easily controlled using \eqref{bound h epsilon equicontinuity formal}. 
\subsection{Approximation of the model by a transport term in the region $y\leq (\frac{v}{3})^{\alpha}$}\label{approx formal subsection}
Finally, we present formally the approximation of \eqref{approximated equation} by \eqref{approx supersol} and refer to Subsection \ref{subsection approximation} for a rigorous proof. We can approximate \eqref{approximated equation}  via 
\begin{align}
&\partial_{t}H_{\epsilon}(y,v,t)+\frac{1}{\epsilon}v^{\alpha}\partial_{y}H_{\epsilon}(y,v,t) \nonumber\\
    &\sim\int_{(0,\frac{v}{2})}[K(v-w,w)H_{\epsilon}(y,v-w,t)-K(v,w)H_{\epsilon}(y,v,t)]H_{\epsilon}(y,w,t)\der w\label{first part approximation formal}.
    \end{align}

Since our strategy relies on finding a suitable supersolution, it suffices to find a lower bound for \eqref{first part approximation formal}. We use (\ref{standard condition}), i.e., $K(v-w,w)\leq K(v,w)$ when $w\in[0,\frac{v}{2}]$, and obtain that
\begin{align}
\partial_{t}& H_{\epsilon}(y,v,t)+\frac{1}{\epsilon}v^{\alpha}\partial_{y}H_{\epsilon}(y,v,t) \nonumber\\
    &-\int_{(0,\frac{v}{2})}[K(v-w,w)H_{\epsilon}(y,v-w,t)-K(v,w)H_{\epsilon}(y,v,t)]H_{\epsilon}(y,w,t)\der w\label{this is where the assumption is needed}\\
    &\geq \partial_{t}H_{\epsilon}(y,v,t)+\frac{1}{\epsilon}v^{\alpha}\partial_{y}H_{\epsilon}(y,v,t) -\int_{(0,\frac{v}{2})}K(v,w)[H_{\epsilon}(y,v-w,t)-H_{\epsilon}(y,v,t)]H_{\epsilon}(y,w,t)\der w.\nonumber
\end{align}
    Since $K(v,w)\leq K_{0}(w^{\gamma}+v^{\gamma})$ and since $w\in[0,\frac{v}{2}]$, we further deduce that
    \begin{align}\label{approximation derivation}
\partial_{t}H_{\epsilon}(y,v,t)+\frac{1}{\epsilon}v^{\alpha}\partial_{y}H_{\epsilon}(y,v,t)   &\sim\int_{(0,\frac{v}{2})}v^{\gamma}[H_{\epsilon}(y,v-w,t)-H_{\epsilon}(y,v,t)]H_{\epsilon}(y,w,t)\der w\nonumber\\
&= - v^{\gamma}\int_{(0,\frac{v}{2})}\int_{v-w}^{v}\partial_{v}H_{\epsilon}(y,z,t)\der z H_{\epsilon}(y,w,t)\der w.
\end{align}
Assuming that $\partial_{v}H_{\epsilon}(y,z,t)$ behaves similarly for $z\in[\frac{v}{2},v]$, it follows that
\begin{align}\label{approximate model with derivative}
\partial_{t} H_{\epsilon}(y,v,t)+\frac{1}{\epsilon}v^{\alpha}\partial_{y}H_{\epsilon}(y,v,t)  \sim -v^{\gamma}\partial_{v} H_{\epsilon}(y,v,t)\int_{(0,\frac{v}{2})}wH_{\epsilon}(y,w,t)\der w.
\end{align}
We denote by
$M_1(y,t):=\int_{(0,\infty)}wH_{\epsilon}(y,w,t)\der w$ the first moment in $v$ of $H_{\epsilon}$. We consider only large values of $v$ so that we can safely assume that $\int_{(0,\frac{v}{2})}wH_{\epsilon}(y,w,t)\der w$ contains most of the mass. Suppose now that $M_1(y,t)$ decays sufficiently fast for large values of $y$, that is 
assume that 
\begin{align}\label{moment x t decays sufficiently fast}
    M_1(y,t)\leq \frac{L}{1+|y|^{d}},
\end{align}
for some sufficiently large $d$ and some $L>0$. Combining (\ref{moment x t decays sufficiently fast}) with \eqref{approximate model with derivative}, we obtain that $H_{\epsilon}$ can be estimated formally from above by a solution of the equation
\begin{align}\label{approximate model}
\partial_{t} G_{\epsilon}(y,v,t)+\frac{1}{\epsilon}v^{\alpha}\partial_{y}G_{\epsilon}(y,v,t)   +\frac{
    Lv^{\gamma}\partial_v G_{\epsilon}(y,v,t)}{1+|y|^{d}}=0.
\end{align}

 \subsection{Mathematical toolbox}
The proof of Theorem \ref{proposition existence via fixed point} relies on the fact that for small values of $\epsilon$, $H_{\epsilon}$ should behave like a mollified Dirac measure in the space variable of the form $\delta(y-v^{\alpha})$. Following the computations in \eqref{the gain term dirac}, the gain term in the coagulation operator will behave like a Dirac measure 
of the form $\delta(y-\big(\frac{v}{2}\big)^{\alpha})$ multiplied by some function depending on the $v$ variable. In other words, when we are at a positive distance from the line $\{y=\big(\frac{v}{2}\big)^{\alpha}\}$, we should obtain sufficiently fast decay for our solutions in a straightforward manner.  The terms $T_{1}$ and $T_{3}$ in \eqref{def t1}, \eqref{def t3} represent the action of the semigroup generated by the sedimentation term on the initial condition in \eqref{initial condition decay}.  The term $T_{2}$ in \eqref{def t2} is an additional term to control the problematic region approaching $\{y=\big(\frac{v}{2}\big)^{\alpha}\}$.

We gather here some technical lemmas that will be useful later in the proof of Theorem \ref{proposition existence via fixed point} and that provide some insight into the choice of space for existence. Firstly, we need estimates for the gain term in the coagulation operator. As noticed in Subsection \ref{proof of existence heuristics}, we often use the fact that
\begin{align}
\int_{0}^{v}\delta(y-w^{\alpha})\delta(y-(v-w)^{\alpha})\der w&=\int_{0}^{v}\delta(y-w^{\alpha})\delta(w^{\alpha}-(v-w)^{\alpha})\der w\label{main heuristic}\\
&=\frac{1}{2\alpha} \bigg(\frac{v}{2}\bigg)^{1-\alpha}\delta\bigg(y-\Big(\frac{v}{2}\Big)^{\alpha}\bigg)\int_{0}^{v}\delta\bigg(w-\frac{v}{2}\bigg)\der w,\label{main heuristic 2}
\end{align}or alternatively, in \eqref{main heuristic} we use the fact that for $\xi,B_{1},B_{2}\in\mathbb{R}$
\begin{align}\label{diracs}
    \delta(\xi-B_{1}) \delta(\xi-B_{2})= \delta(B_{1}-B_{2}) \delta(\xi-B_{1})= \delta(B_{1}-B_{2}) \delta(\xi-B_{2}).
\end{align}
Our solutions behave like a mollified version of a Dirac measure at $y=v^{\alpha}$. While we cannot expect \eqref{diracs} to be valid, we need a similar estimate to the one in \eqref{diracs} to hold for our solutions. We prove an estimate of the form
\begin{align*}
    H_{\epsilon}(\xi-B_{1}) H_{\epsilon}(\xi-B_{2})\sim H_{\epsilon}(B_{1}-B_{2}) [H_{\epsilon}(\xi-B_{1})+H_{\epsilon}(\xi-B_{2})].
\end{align*}
More precisely, we prove that
\begin{lem}\label{lemma dirac convolution}
Let $\xi,B_{1},B_{2}\in\mathbb{R}$. It holds that 
\begin{align}\label{ineq 1}
       (1+|\xi-B_{1}|)(1+|\xi-B_{2}|)\geq \frac{(1+\min\{|\xi-B_{1}|,|\xi-B_{2}|\})(1+|B_{1}-B_{2}|)}{2}.
\end{align}
Let $m\in\mathbb{N}$. As an immediate consequence of \eqref{ineq 1}, we have that
    \begin{align}\label{ineq 2}
        \frac{1}{(1+|\xi-B_{1}|)^{m}} \frac{1}{(1+|\xi-B_{2}|)^{m}}\leq \frac{2^{2m-1}}{(1+|B_{1}-B_{2}|)^{m}}\bigg[\frac{1}{(1+|\xi-B_{1}|)^{m}} +\frac{1}{(1+|\xi-B_{2}|)^{m}}\bigg],
    \end{align}
    for any $\xi,B_{1},B_{2}\in\mathbb{R}$.
\end{lem}
This will in turn imply that for $\psi$ as in \eqref{psi definition}, it holds that
\begin{align*}
e^{\frac{2t}{\epsilon}}\psi(e^{\frac{t}{\epsilon}}(\xi-B_{1}))\psi(e^{\frac{t}{\epsilon}}(\xi-B_{2}))\leq C e^{\frac{t}{\epsilon}}\psi(e^{\frac{t}{\epsilon}}(B_{1}-B_{2}))\bigg[e^{\frac{t}{\epsilon}}\psi(e^{\frac{t}{\epsilon}}(\xi-B_{1}))+e^{\frac{t}{\epsilon}}\psi(e^{\frac{t}{\epsilon}}(\xi-B_{2}))\bigg].
\end{align*}
\begin{proof}[Proof of Lemma \ref{lemma dirac convolution}]
    Let $\xi, B_{1},B_{2}  \in\mathbb{R}$ and assume without loss of generality that $B_{1}\leq B_{2}$. 
 
 If $\xi\leq B_{1},$ it holds that
    \begin{align*}
        (1+|\xi-B_{1}|)(1+|\xi-B_{2}|)&\geq (1+|\xi-B_{1}|)(1+|B_{1}-B_{2}|)\\
        &= (1+\min\{|\xi-B_{1}|,|\xi-B_{2}|\})(1+|B_{1}-B_{2}|).
    \end{align*}
    If $\xi\geq B_{2},$ it holds that
    \begin{align*}
        (1+|\xi-B_{1}|)(1+|\xi-B_{2}|)&\geq (1+|B_{1}-B_{2}|)(1+|\xi-B_{2}|)\\
        &= (1+\min\{|\xi-B_{1}|,|\xi-B_{2}|\})(1+|B_{1}-B_{2}|).
    \end{align*}
        If $\xi\in[B_{1},B_{2}],$ it holds that
        \begin{align*}
               (1+|\xi-B_{1}|)(1+|\xi-B_{2}|)\geq \frac{1}{2}(1+\min\{|\xi-B_{1}|,|\xi-B_{2}|\})(1+|B_{1}-B_{2}|),
        \end{align*}
        since
        \begin{align*}
            |B_{1}-B_{2}|\leq |\xi-B_{1}|+|\xi-B_{2}|\leq 2\max\{|\xi-B_{1}|,|\xi-B_{2}|\}.
        \end{align*}

        Thus, \eqref{ineq 1} holds. Moreover, we have that
        \begin{align*}
               [(1+|\xi-B_{1}|)&+(1+|\xi-B_{2}|)](1+\min\{|\xi-B_{1}|,|\xi-B_{2}|\})\\
               &\geq(1+\max\{|\xi-B_{1}|,|\xi-B_{2}|\})(1+\min\{|\xi-B_{1}|,|\xi-B_{2}|\})\\
               &=(1+|\xi-B_{1}|)(1+|\xi-B_{2}|)
        \end{align*}
     which implies that
        \begin{align*}
           \frac{1}{1+\min\{|\xi-B_{1}|,|\xi-B_{2}|\}}\leq \frac{1}{1+|\xi-B_{1}|}+\frac{1}{1+|\xi-B_{2}|}.
        \end{align*}
By \eqref{ineq 1}, it follows that
\begin{align*}
      \frac{1}{(1+|\xi-B_{1}|)^{m}}  \frac{1}{(1+|\xi-B_{2}|)^{m}}&\leq \frac{2^{m}}{(1+\min\{|\xi-B_{1}|,|\xi-B_{2}|\})^{m}(1+|B_{1}-B_{2}|)^{m}}\\
      &\leq \frac{2^{m}}{(1+|B_{1}-B_{2}|)^{m}}\bigg[\frac{1}{1+|\xi-B_{1}|}+\frac{1}{1+|\xi-B_{2}|}\bigg]^{m}\\
        &\leq \frac{2^{2m-1}}{(1+|B_{1}-B_{2}|)^{m}}\bigg[\frac{1}{(1+|\xi-B_{1}|)^{m}}+\frac{1}{(1+|\xi-B_{2}|)^{m}}\bigg]
\end{align*}
and therefore \eqref{ineq 2} holds.
\end{proof}
We also use in \eqref{main heuristic 2} the fact that
\begin{align*}
\delta((v-w)^{\alpha}-w^{\alpha})=\frac{1}{2\alpha }\Big(\frac{v}{2}\Big)^{1-\alpha}\delta\Big(w-\frac{v}{2}\Big).
\end{align*}
Similarly, the following inequality holds for the function $\psi$ in \eqref{psi definition}.
\begin{lem}[Change of variables]\label{ref lemma dirac v-w v}
Let $\alpha\in(0,1), \epsilon\in(0,1)$, $\beta>1$, $s\geq 0$, $y\in\mathbb{R},$ and $v>0$. Then there exists a constant $C(\alpha)>0$, which depends on $\alpha$, such that 
    \begin{align}
      \int_{0}^{\frac{v}{2}} \frac{e^{\frac{s}{\epsilon}}\der w}{1+e^{\frac{sn}{\epsilon}}|(v-w)^{\alpha}-w^{\alpha}|^{\beta}}&\leq C(\alpha) v^{1-\alpha};\label{ineq jacobian 1}\\
       \int_{0}^{\frac{v}{2}} \frac{e^{\frac{s}{\epsilon}}\der w}{1+e^{\frac{sn}{\epsilon}}|y-(v-w)^{\alpha}|^{\beta}}&\leq C(\alpha) v^{1-\alpha}.\label{ineq jacobian 2}
    \end{align} 
\end{lem}
\begin{proof}
We only prove \eqref{ineq jacobian 1} since \eqref{ineq jacobian 2} can be proved similarly. We make the change of variables $z=e^{\frac{s}{\epsilon}}[(v-w)^{\alpha}-w^{\alpha}]$. Then 
\begin{align*}
    \frac{\der w}{\der z}=e^{-\frac{s}{\epsilon}}\big(-\alpha (v-w)^{\alpha-1}-\alpha w^{\alpha-1}\big)^{-1}.
\end{align*} 
Since $\alpha\in(0,1)$ and $w\in\big[0,\frac{v}{2}\big]$, it holds that
\begin{align*}
    \bigg|\frac{\der w}{\der z}\bigg|\leq C e^{-\frac{s}{\epsilon}} v^{1-\alpha}.
\end{align*}
Thus
    \begin{align*}
      \int_{0}^{\frac{v}{2}} \frac{e^{\frac{s}{\epsilon}}\der w}{1+e^{\frac{sn}{\epsilon}}|(v-w)^{\alpha}-w^{\alpha}|^{\beta}}\leq Cv^{\alpha-1} \int_{0}^{e^{\frac{s}{\epsilon}} v^{\alpha}}\frac{\der z}{1+|z|^{\beta}}\leq C v^{1-\alpha}.  
    \end{align*} 
\end{proof}
Moreover, we need to derive some estimates for the action of the semigroup generated by the transport term on the function $\psi$ in \eqref{psi definition}. This will further justify the term $T_{2}$ in the estimate \eqref{existence space fixed point}.
\begin{prop}\label{prop semigroup decay}
Let $\epsilon\in(0,1)$, $s\geq 0$, $y\in\mathbb{R}$,  $v\in(0,\infty)$, and $S_{\epsilon}$ be as in \eqref{definition semigroup}. It holds that
\begin{align}\label{estimate for psi}
    \int_{0}^{t}S_{\epsilon}(t-s) \psi(y-v^{\alpha})\der s\leq \frac{C \epsilon}{|y-v^{\alpha}|}\frac{\min\{1,e^{\frac{t}{\epsilon}}|y-v^{\alpha}|\}}{(1+|y-v^{\alpha}|^{m-1})},
\end{align}
for $\psi$ as in \eqref{psi definition}.
\end{prop}
\begin{proof}
 We make the change of variables $z=t-s$ and thus \eqref{estimate for psi} is equivalent to proving that
\begin{align*}
    \int_{0}^{t}S_{\epsilon}(z) \psi(y-v^{\alpha})\der z\leq \frac{C\epsilon}{|y-v^{\alpha}|}\frac{\min\{1,e^{\frac{t}{\epsilon}}|y-v^{\alpha}|\}}{(1+|y-v^{\alpha}|^{m-1})}.
\end{align*}
We make the change of variables $s=e^{\frac{z}{\epsilon}}|y-v^{\alpha}|$ and we obtain that
\begin{align*}
    \int_{0}^{t}\frac{e^{\frac{z}{\epsilon}}}{1+e^{\frac{mz}{\epsilon}}|y-v^{\alpha}|^{m}} \der z=&\frac{\epsilon}{|y-v^{\alpha}|}\int_{|y-v^{\alpha}|}^{e^{\frac{t}{\epsilon}}|y-v^{\alpha}|}\frac{1}{1+|s|^{m}}\der s.
\end{align*}
It holds that
\begin{align}
           \int_{|y-v^{\alpha}|}^{e^{\frac{t}{\epsilon}}|y-v^{\alpha}|}\frac{1}{1+|s|^{m}}\der s&\leq  \int_{|y-v^{\alpha}|}^{\infty}\frac{1}{1+|s|^{m}}\der s\leq \frac{C}{1+|y-v^{\alpha}|^{m-1}};\label{part1}\\
                     \int_{|y-v^{\alpha}|}^{e^{\frac{t}{\epsilon}}|y-v^{\alpha}|}\frac{1}{1+|s|^{m}}\der s&\leq \frac{1}{1+|y-v^{\alpha}|^{m}}  \int_{|y-v^{\alpha}|}^{e^{\frac{t}{\epsilon}}|y-v^{\alpha}|}\der s\leq \frac{Ce^{\frac{t}{\epsilon}}|y-v^{\alpha}|}{1+|y-v^{\alpha}|^{m}} \leq \frac{Ce^{\frac{t}{\epsilon}}|y-v^{\alpha}|}{1+|y-v^{\alpha}|^{m-1}}.\label{part2}
\end{align}
Combining the above estimates, we obtain the desired inequality.
\end{proof}

\begin{lem}[Monotonicity properties of the semigroup] \label{properties of semigroup lemma}Let $y\in\mathbb{R}$ and $v\in(0,\infty)$. The following properties hold.
\begin{enumerate}
\item Let $S_{\epsilon}$ as in \eqref{definition semigroup}. Assume that $f(y,v)\geq  g(y,v)$ for all $y\geq y_{1}$, where $y_{1}\geq v^{\alpha}$. In addition, assume that $f(y,v)\leq  g(y,v)$ for all $y\leq y_{2}$, where $y_{2}\leq v^{\alpha}$. It holds that
\begin{align}
    S_{\epsilon}(s)[f(y,v)]&\geq S_{\epsilon}(s)[g(y,v)] \textup{ if } y\geq y_{1}, \textup{ for all } s\geq 0;\label{semigroup one}\\
      S_{\epsilon}(s)[f(y,v)]&\leq S_{\epsilon}(s)[g(y,v)] \textup{ if } y\leq y_{2}, \textup{ for all } s\geq 0.\label{semigroup two}
\end{align} 
\item Let $y\in\mathbb{R}, v\in(0,\infty),$ and $w\in[0,\frac{v}{2}]$. It holds that either $|y-w^{\alpha}|\geq |y-(\frac{v}{2})^{\alpha}|$ or  $|y-(v-w)^{\alpha}|\geq |y-(\frac{v}{2})^{\alpha}|$.
\end{enumerate}
\end{lem}
\begin{proof}
$1.$ We only prove that  \eqref{semigroup two} holds. We have to prove that $e^{\frac{s}{\epsilon}}(y-v^{\alpha})+v^{\alpha}\leq y_{2}$ and \eqref{semigroup two} follows. This is true since $y_{2}\leq v^{\alpha}$ implies that  $e^{\frac{s}{\epsilon}}(y-v^{\alpha})\leq y-v^{\alpha}\leq y_{2}-v^{\alpha}$, for all $y\leq y_{2}\leq v^{\alpha}$ and $s\geq 0$.

$2.$  Since $w\in[0,\frac{v}{2}]$, if $y\geq \big(\frac{v}{2}\big)^{\alpha}$, then $y-w^{\alpha}\geq y-(\frac{v}{2})^{\alpha}\geq 0$. Otherwise, since $v-w\in[\frac{v}{2},v]$, then $y-(v-w)^{\alpha}\leq y-(\frac{v}{2})^{\alpha}\leq 0$. This proves our claim.
\end{proof}

\subsection{Plan of the proof of Theorem \ref{proposition existence via fixed point}}

We define inductively a sequence of solutions that solve a linear version of our coagulation model. More precisely, let $H_{n+1}$ be defined in the following manner.
\begin{align}
    \partial_{t}H_{n+1}(y,v,t)+\frac{1}{\epsilon}\partial_{y}[(v^{\alpha}-y)H_{n+1}(y,v,t)]=&\mathbb{K}[H_{n+1},H_{n}](y,v,t)\nonumber\\
    :=&\int_{0}^{\frac{v}{2}}K(v-w,w)H_{n+1}(y,v-w,t)H_{n}(y,w,t)\der w\nonumber \\
   &-\int_{0}^{\infty}K(v,w)H_{n+1}(y,v,t)H_{n}(y,w,t)\der w ,\label{iteration supersol}
\end{align}
with $H_{n}(y,v,0)\in  C^{1}(\mathbb{R}\times (0,\infty))$, for all $n\in\mathbb{N}$ and such that
\begin{align}\label{initial condition linear}
(1+v^{b})H_{n+1}(y,v,0)=(1+v^{b})H_{n}(y,v,0)\leq \frac{A}{1+|y-v^{\alpha}|^{m}}, \textup{ for all } n\in \mathbb{N}.
\end{align}
At $n=0$, we take $H_{0}$ to be
\begin{align}\label{h0 form}
(1+v^{b})H_{0}(y,v,t)=\frac{A e^{\frac{t}{\epsilon}}}{1+e^{\frac{mt}{\epsilon}}|y-v^{\alpha}|^{m}}
\end{align}
or in other words 
\begin{align}
\partial_{t}H_{0}+\partial_{y}[(v^{\alpha}-y)H_{0}]&=0;\label{h0 pde}\\
(1+v^{b})H_{0}(y,v,0)&=\frac{A}{1+|y-v^{\alpha}|^{m}}.\nonumber
\end{align}
It holds that there exists $T\leq 1$ sufficiently small and independent of $\epsilon$ such that there exists a solution $H_{n+1}\in C^{1}([0,T]\times \mathbb{R}\times (0,\infty))$ which solves \eqref{iteration supersol} with $H_{n+1}(y,v,0)$ as in \eqref{initial condition linear}.
\begin{prop}\label{prop compact support functions}
Let $N>1$, $\epsilon>0$, $n\in\mathbb{N}$. Then there exists $T\leq 1$, sufficiently small but independent of $\epsilon,n,$ or $N$, such that if $H_{\epsilon,\textup{in}}\in C^{1}(\mathbb{R}\times(0,\infty))$ such that  $\supp  H_{\epsilon,\textup{in}}\subseteq \{[-N^{\alpha},N^{\alpha}]\times[0,N]\}$ and if $H_{n}\in C^{1}([0,T]\times\mathbb{R}\times (0,\infty))$ is such that  $\supp  H_{n}(t)\subseteq \{[-N^{\alpha},N^{\alpha}]\times[0,N]\}$, for all $t\in[0,T]$, then there exists $H_{n+1}\in C^{1}([0,T]\times\mathbb{R}\times (0,\infty))$ that solves
    \begin{align}
 \partial_{t}H_{n+1}(y,v,t)+\frac{1}{\epsilon}\partial_{y}[(v^{\alpha}-y)H_{n+1}(y,v,t)]=&\mathbb{K}_{N}[H_{n+1},H_{n}](y,v,t)\label{equation appendix}\\ 
 :=&\int_{0}^{\frac{v}{2}}K_{N}(v-w,w)H_{n+1}(y,v-w,t)H_{n}(y,w,t)\der w\nonumber \\
   &-\int_{0}^{\infty}K_{N}(v,w)H_{n+1}(y,v,t)H_{n}(y,w,t)\der w ,\nonumber
\end{align}
for all $t\in[0,T]$, where
\begin{align}
K_{N}(v,w)&:=K(v,w)\chi_{N}(v+w), \textup{ with }\label{truncation kernel in m}\\
\chi_{N}\in C([0,\infty)),&  \quad \chi_{N}\colon [0,\infty)\rightarrow [0,1] \; \mbox{ such that } \chi_{N}(v)=1
\mbox{ if } v\leq \frac{N}{2} \mbox{ and } \chi_{N}(v)=0 \mbox{ if } v\geq N.
\nonumber
\end{align}
Moreover, it holds that 
\begin{align}
    \supp H_{n+1}(t)\subseteq \{[-N^{\alpha},N^{\alpha}]\times[0,N]\}, \textup{ for all }t\in[0,T].
\end{align}
\end{prop}
 We prove Proposition \ref{prop compact support functions} in Appendix \ref{appendix a}.
\begin{rmk}\label{first remark compact support}
In principle, since the functions in Proposition \ref{prop compact support functions} depend on the iteration $n$ and on the truncation $N$,  we have to keep in mind that we work with a sequence $H_{n,N}$. The compact support of functions will be used in Proposition \ref{continuity argument prop}. We will then obtain a uniform  fast decaying bound for the sequence of solutions $H_{n,N}$, which is in particular independent of the truncation $N$ in \eqref{truncation kernel in m} and the iteration in $n\in\mathbb{N}$. Then one can pass to the limit $N \to \infty$. Since this procedure is standard once one has the bounds on the solution, we omit the dependence on $N$ and work directly with $K$ and equation \eqref{iteration supersol} in order to simplify the notation.
\end{rmk}

We now focus on proving Theorem \ref{proposition existence via fixed point}.
\begin{proof}[Plan of the proof of Theorem \ref{proposition existence via fixed point}]
Let $\epsilon\in(0,1)$, fixed. We will prove that there exists $T\leq 1$ which is independent of $\epsilon$ such that for all $t\in[0,T]$, it holds that if 
\begin{align*}
H_{n}(y,v,t)\leq T_{1}(y,v,t)+T_{2}(y,v,t)+T_{3}(y,v,t), \textup{ for all } y\in\mathbb{R}, v\in(0,\infty),
\end{align*}
where $T_{1},T_{2},T_{3}$ are as in \eqref{def t1}-\eqref{def t3}, then
\begin{align}\label{main bound hn}
H_{n+1}(y,v,t)\leq T_{1}(y,v,t)+T_{2}(y,v,t)+T_{3}(y,v,t), \textup{ for all } y\in\mathbb{R}, v\in(0,\infty).
\end{align}
\begin{itemize}
    \item \textbf{Step $1.$ Region $ y\leq (\frac{v}{3})^{\alpha}$.} 
\end{itemize}
We will prove that in this region $G_{\epsilon}$ is a supersolution of \eqref{iteration supersol}, where $G_{\epsilon}$ is defined as follows. 
\begin{align*}
  \partial_{t}G_{\epsilon}(y,v,t)+\frac{1}{\epsilon}\partial_{y}[(v^{\alpha}-y)G_{\epsilon}(y,v,t)]+\frac{L\xi(v)v^{\gamma}}{1+|y|^{d}}\partial_{v}G_{\epsilon}(y,v,t)=0,
  \end{align*}
  for some suitably chosen $d\in\mathbb{N}$, some sufficiently large $L>1$, and where $\xi$ is such that
  \begin{equation}
\xi\in C([0,\infty))\,, \quad \xi\colon [0,\infty)\rightarrow [0,1] \; \mbox{ such that } \xi(v)=1
\mbox{ if } v\geq 1 \mbox{ and } \xi(v)=0 \mbox{ if } v\leq \frac{1}{2}.
\end{equation}
We then prove that when $y\leq \big(\frac{v}{3}\big)^{\alpha}$, it holds that $G_{\epsilon}$ satisfies the following upper bound
\begin{align}
 G_{\epsilon}(y,v,t)\leq T_{3}(y,v,t),
\end{align}
where $T_{3}$ is as in \eqref{def t3}.

This is the content of Section \ref{section four}.

\begin{itemize}
    \item \textbf{Step $2.$ Region $ y\geq (\frac{v}{3})^{\alpha}$.} 
\end{itemize}

\textit{\underline{Continuity argument:}} We first prove by continuity that if at time $s\geq 0$, we have that $H_{n+1}(y,v,s)\leq T_{1}(y,v,s)+T_{2}(y,v,s)+T_{3}(y,v,s)$, then for a sufficiently small $\tilde{\delta}$ (which depends on $\epsilon$) it holds that $H_{n+1}(y,v,z)$ $\leq 4[T_{1}(y,v,z)+T_{2}(y,v,z)+T_{3}(y,v,z)]$, for all $z\in[s,s+\tilde{\delta}]$. This is the content of Proposition \ref{continuity argument prop}.

\textit{\underline{Iteration argument:}} We then work with the mild formulation of the linear version of our coagulation model \eqref{iteration supersol}. More precisely, we have that
\begin{align}\label{mild form iteration in n}
&H_{n+1}(y,v,t)=S_{\epsilon}(t)[H_{n+1}(y,v,0) ] D[H_{n}](y,v,0,t)\\
&+\int_{0}^{t}\int_{(0,\frac{v}{2})}D[H_{n}](y,v,s,t)
S_{\epsilon}(t-s)\bigg[K(v-v',v')H_{n+1}(y,v-v',s) H_{n}(y,v',s)\bigg]\der s,\nonumber 
\end{align}
with $D[H_{n}]$ defined as in \eqref{definition s mild}. In other words, we have that
\begin{align*}
H_{n+1}(y,v,t)\leq S_{\epsilon}(t)[H_{n+1}(y,v,0) ] +2K_{0}v^{\gamma}\int_{0}^{t}\int_{(0,\frac{v}{2})}
S_{\epsilon}(t-s)\bigg[H_{n+1}(y,v-v',s) H_{n}(y,v',s)\bigg]\der s.\nonumber 
\end{align*}
For $t\leq s+\tilde{\delta}$, the following holds.
\begin{itemize}
    \item By induction, we have that $H_{n}(y,v,t)\leq T_{1}+T_{2}+T_{3}$.
    \item By the continuity argument, we have that $H_{n+1}(y,v,t)\leq 4(T_{1}+T_{2}+T_{3})$.
\end{itemize}We then prove that as long as $t\leq 1$ and $y\geq (\frac{v}{3})^{\alpha}$, it holds that
\begin{align*}
Cv^{\gamma}\int_{0}^{t}\int_{(0,\frac{v}{2})}
S_{\epsilon}(t-s)\bigg[H_{n+1}(y,v-v',s) H_{n}(y,v',s)\bigg]\der s\leq \frac{T_{1}+T_{2}}{2}.
\end{align*}
Since if  $y\geq (\frac{v}{3})^{\alpha}$ we also have that $S_{\epsilon}(t)[H_{n+1}(y,v,0) ] \leq \frac{T_{1}+T_{2}}{2}$, we deduce that
\begin{align}\label{bound h n+1}
    H_{n+1}(y,v,t)\leq T_{1}+T_{2}, \textup{ for all } t\leq \min\{1,s+\tilde{\delta}\},  y\geq (\frac{v}{3})^{\alpha}.
\end{align}
From Step $1$, we have 
\begin{align*}
    H_{n+1}(y,v,t)&\leq T_{3}, \textup{ for all } t\leq 1,  y\leq (\frac{v}{3})^{\alpha}.
\end{align*}
We can thus conclude that $H_{n+1}(y,v,t)\leq T_{1}+T_{2}+T_{3}, \textup{ for all } t\leq \min\{1,s+\tilde{\delta}\}$ and then iterate the argument in order to obtain that the upper bound \eqref{bound h n+1} holds true for all $t\leq 1$. 

This is the content of Section \ref{section three} of this paper.
\begin{itemize}
    \item \textbf{Step $3.$ Limit as $n\rightarrow\infty$} 
\end{itemize}
Let $t\geq 0$, $y\in\mathbb{R}$ and $v> 0$. For $n\in\mathbb{N}$, we denote by 
\begin{align}
R_n(y,v,t):=H_{n+1}(y,v,t)-H_n(y,v,t).
\end{align}

From \eqref{iteration supersol}, we find that $R_{n}$ satisfies
\begin{equation}
\left\{\begin{aligned}
\partial_t R_n(y,v,t)+\frac{1}{\epsilon}\partial_y[(v^{\alpha}-y) R_n(y,v,t)]&=\mathbb{K}[R_n,H_{n}]+\mathbb{K}[H_{n},R_{n-1}]; \\
R_{n}(y,v,0)&=0,
   \end{aligned}\right.
   \end{equation}
where $\mathbb{K}[f,g]$, for two functions $f,g$, is as in \eqref{iteration supersol}. Since $\mathbb{K}[R_n,H_{n}]$ is a linear operator in $R_{n}$ and $\mathbb{K}[H_{n},R_{n-1}]$ does not depend on $H_{n}$, we use Duhamel's principle in order to deduce that in order to find estimates for $R_{n}$  it is sufficient to find estimates for the system
    \begin{equation}
\left\{\begin{aligned}
\partial_t R^{s}_n(x,v,t)+\frac{1}{\epsilon}\partial_y[(v^{\alpha}-y) R^{s}_n(y,v,t)]&=\mathbb{K}[R^{s}_n,H_{n}], \textup{ for } t>s; \\
R^{s}_{n}(y,v,s)&=\mathbb{K}[H_{n},R^{s}_{n-1}].
   \end{aligned}\right.
   \end{equation}
We then use induction in order to prove that
\begin{align}\label{first bound rn}
    |R_{n}(t)|\leq (Cte^{\frac{1}{\epsilon}})^{n}, \textup{ for all } n\in\mathbb{N}.
\end{align}
We can thus conclude that 
\begin{align}\label{rn converges}
    \textup{ for all } t\approx\mathcal{O}(e^{-\frac{1}{\epsilon}}),  R_{n}(t)\rightarrow 0 \textup{ as } n\rightarrow\infty.
\end{align} This was also the strategy in \cite{cristianinhom}. Nonetheless, the difference is that functions in \cite{cristianinhom} were uniformly bounded in $\epsilon$ in time, while in our case $H_{n}(t)\leq Ce^{\frac{t}{\epsilon}}$, which is the reason why the exponential term appears in \eqref{first bound rn}. We then wish to extend \eqref{rn converges} to hold for times of order $\mathcal{O}(1)$. By \eqref{first bound rn}, we have that $R_{n}\big(\frac{e^{-\frac{1}{\epsilon}}}{2}\big)\leq 2^{-n}$ and we can use the invariance under translations in time in order to iterate the argument. Let
\begin{align*}
\tilde{R}_{n}(t)=R_{n}\bigg(t+\frac{e^{-\frac{1}{\epsilon}}}{2}\bigg).
\end{align*}
$\tilde{R}_{n}$ then satisfies
\begin{equation}\label{formal argument duhamel}
\left\{\begin{aligned}
\partial_t \tilde{R}_n(y,v,t)+\frac{1}{\epsilon}\partial_y[(v^{\alpha}-y) \tilde{R}_n(y,v,t)]&=\mathbb{K}\bigg[\tilde{R}_n,H_{n}\bigg(t+\frac{e^{-\frac{1}{\epsilon}}}{2}\bigg)\bigg]+\mathbb{K}\bigg[H_{n}\bigg(t+\frac{e^{-\frac{1}{\epsilon}}}{2}\bigg),\tilde{R}_{n-1}\bigg]; \\
\tilde{R}_{n}(y,v,0)&=R_{n}\bigg(\frac{e^{-\frac{1}{\epsilon}}}{2}\bigg).
   \end{aligned}\right.
   \end{equation}
Since $R_{n}\bigg(t+\frac{e^{-\frac{1}{\epsilon}}}{2}\bigg)\leq 2^{-n}\rightarrow 0$ as $n\rightarrow \infty$, then \eqref{formal argument duhamel} is to be understood at a formal level as
 \begin{equation}
\left\{\begin{aligned}
\partial_t \tilde{R}_n(y,v,t)+\frac{1}{\epsilon}\partial_y[(v^{\alpha}-y) \tilde{R}_n(y,v,t)]&=\mathbb{K}\bigg[\tilde{R}_n,H_{n}\bigg(t+\frac{e^{-\frac{1}{\epsilon}}}{2}\bigg)\bigg]+\mathbb{K}\bigg[H_{n}\bigg(t+\frac{e^{-\frac{1}{\epsilon}}}{2}\bigg),\tilde{R}_{n-1}\bigg]; \\
\tilde{R}_{n}(y,v,0)&=0.
   \end{aligned}\right.
   \end{equation}
We are able to use again Duhamel's principle and find suitable estimates that allow us to extend the argument to hold for longer times. We refer to Section \ref{section five} for a rigorous proof.

The choice of presentation, namely we present the proof of Step $2$ before Step $1$, is since the region $y\geq (\frac{v}{3})^{\alpha}$ contains the part of the solution where most of the mass will concentrate and is thus more difficult to prove.
\end{proof} 
\section{Region $y\geq (\frac{v}{3})^{\alpha}$}\label{section three}
\subsection{Continuity argument}
We first prove the continuity argument.

  \begin{prop}[Continuity argument]\label{continuity argument prop} Assume that at time $s\geq 0$, we have that
\begin{align}\label{initial condition continuity argument}
   H_{n+1}(y,v,s)\leq T_{1}(y,v,s)+T_{2}(y,v,s)+T_{3}(y,v,s),
\end{align}
where $T_{1},T_{2},T_{3}$ are as in \eqref{def t1}-\eqref{def t3}. Then there exists $\tilde{\delta}$, which may depend on $n$, $\epsilon$,  but is independent of $s,$ $y,v$, such that for all $z\in[s,\min\{s+\frac{\epsilon}{m},1,s+\tilde{\delta}\}]$, where $m$ is as in \eqref{initial condition decay}, it holds that
\begin{align}\label{continuity argument statement}
     H_{n+1}(y,v,z)\leq 4\big[ T_{1}(y,v,z)+ T_{2}(y,v,z)+ T_{3}(y,v,z)\big].
\end{align}
\end{prop}
\begin{rmk}\label{rmk compact support for delta uniform}  We will need in the proof of Proposition \ref{supersol prop above region} that $\tilde{\delta}$ is independent of $y, v$. By Proposition \ref{prop compact support functions}, it holds that $H_{n+1}$ has compact support and this will imply that $\tilde{\delta}$ does not depend on $y$ or $v$. We refer to Remark \ref{first remark compact support}. We remember we work with a sequence $H_{n,N}$, which is compactly supported and the support depends on $N$. We will then obtain a uniform  fast decaying bound for the sequence of solutions $H_{n,N}$, which is in particular independent of the truncation $N$ and then one can pass to the limit as $N \to \infty$. The rest of the estimates in this paper do not depend on the truncation $N$ and thus do not depend on the compact support. As mentioned in Remark \ref{first remark compact support}, we omit the dependence on $N$ in order to avoid complicated notation.
\end{rmk}
\begin{proof}[Proof of Proposition \ref{continuity argument prop}]
We have that \eqref{initial condition continuity argument} holds. Let 
\begin{align}\label{definition epsilon tilde}
\tilde{\epsilon}:=\min\{\frac{A\psi(e^{\frac{1}{\epsilon}}(y-v^{\alpha}))}{1+v^{b}},\frac{A\epsilon}{|y-v^{\alpha}|(1+v^{b})}\}.
\end{align} Since $H_{n+1}$ has compact support, it holds that there exists $\epsilon_{N}$, where $N$ is as in \eqref{truncation kernel in m}, which depends on $N$ but is otherwise independent on $y,v$, such that $\tilde{\epsilon}\geq \epsilon_{N}$. By continuity in time of $H_{n+1}$, we have that there exists  $\tilde{\delta}$ such that for all $z\in[s,s+\tilde{\delta}]$ it holds that $|H_{n+1}(y,v,z)- H_{n+1}(y,v,s)|\leq \epsilon_{N}$. It thus holds that
\begin{align*}
H_{n+1}(y,v,z)\leq  T_{1}(y,v,s)+T_{2}(y,v,s)+T_{3}(y,v,s)+\epsilon_{N}.
\end{align*}
Take $z\in[s,\min\{1,s+\frac{\epsilon}{m},s+\tilde{\delta}\}]$. We analyze each term separately.
\begin{align*}
  T_{1}(y,v,s)&=   \frac{2A}{1+v^{b}}e^{\frac{s}{\epsilon}}\psi(e^{\frac{s}{\epsilon}}(y-v^{\alpha}))\chi_{1}(y,v,s)\leq  \frac{2A}{1+v^{b}}e^{\frac{z}{\epsilon}}\psi(e^{\frac{s}{\epsilon}}(y-v^{\alpha}))\chi_{1}(y,v,s),
\end{align*}
with $\chi_{1}$ as in \eqref{chi1}. Since $z\leq s+\frac{\epsilon}{m}$, it holds that 
\begin{align}\label{bound t1 and t3}
1+e^{\frac{zm}{\epsilon}}|y-v^{\alpha}|^{m}=1+e^{\frac{sm}{\epsilon}}|y-v^{\alpha}|^{m}e^{\frac{(z-s)m}{\epsilon}}\leq 3(1+e^{\frac{sm}{\epsilon}}|y-v^{\alpha}|^{m})
\end{align}and thus we further have that
\begin{align*}
  T_{1}(y,v,s)\leq  \frac{6A }{1+v^{b}}e^{\frac{z}{\epsilon}}\psi(e^{\frac{z}{\epsilon}}(y-v^{\alpha}))\chi_{1}(y,v,s).
\end{align*}
We define the set
\begin{align*}
    U(s):=\{y\leq  v^{\alpha}-C_{0}^{-1}\epsilon^{-\frac{1}{m-1}}  e^{-\frac{z}{\epsilon}}\}.
\end{align*}Since $z\geq s$, we have $v^{\alpha}-C_{0}^{-1}\epsilon^{-\frac{1}{m-1}}  e^{-\frac{s}{\epsilon}}\leq v^{\alpha}-C_{0}^{-1}\epsilon^{-\frac{1}{m-1}}  e^{-\frac{z}{\epsilon}}$. Thus
\begin{align}\label{t1 bound continuity first}
  T_{1}(y,v,s)\leq  3T_{1}(y,v,z)+\frac{6A }{1+v^{b}}e^{\frac{z}{\epsilon}}\psi(e^{\frac{z}{\epsilon}}(y-v^{\alpha}))\mathbbm{1}_{U(z)\setminus U(s)}, \textup{ for } s\leq z.
\end{align}
However, in the region where $  \{y\leq v^{\alpha}-C_{0}^{-1}\epsilon^{-\frac{1}{m-1}}  e^{-\frac{z}{\epsilon}}\}$, it holds that $e^{\frac{z}{\epsilon}}|y-v^{\alpha}|\geq C_{0}^{-1}\epsilon^{-\frac{1}{m-1}}$. Thus
\begin{align*}
    e^{\frac{z}{\epsilon}}\psi(e^{\frac{z}{\epsilon}}(y-v^{\alpha}))\mathbbm{1}_{U(z)\setminus U(s)}&\leq \frac{1}{|y-v^{\alpha}|}\frac{1}{e^{\frac{z(m-1)}{\epsilon}}|y-v^{\alpha}|^{m-1}}\mathbbm{1}_{U(z)\setminus U(s)}\leq  \frac{C_{0}^{m-1}\epsilon}{|y-v^{\alpha}|}\mathbbm{1}_{U(z)\setminus U(s)}
\end{align*}
from which we can deduce that 
\begin{align}\label{t1 bound continuity second}
\frac{6A }{1+v^{b}}e^{\frac{z}{\epsilon}}\psi(e^{\frac{z}{\epsilon}}(y-v^{\alpha}))\mathbbm{1}_{U(z)\setminus U(s)}\leq  \frac{6A C_{0}^{m-1}\epsilon}{(1+v^{b})|y-v^{\alpha}|}\mathbbm{1}_{U(z)\setminus U(s)}\leq 3T_{2}(y,v,z).
\end{align}
From \eqref{t1 bound continuity first} and \eqref{t1 bound continuity second}, we deduce that
\begin{align*}
  T_{1}(y,v,s)\leq  3[T_{1}(y,v,z)+T_{2}(y,v,z)],\textup{ } s\leq z.
\end{align*}
Moreover, since $z\geq s$, we also have that
\begin{align*}
T_{2}(y,v,s)=\frac{2 C_{0}^{m-1}A^{2} \epsilon}{(1+v^{b})|y-v^{\alpha}|}\chi_{2}(y,v,s)\leq T_{2}(y,v,z), \textup{ } s\leq z,
\end{align*}
where $\chi_{2}$ is as in \eqref{chi2}. Using \eqref{bound t1 and t3}, we also have that $T_{3}(y,v,s)\leq 3 T_{3}(y,v,z)$. We now find an upper bound for $\epsilon_{N}$. Let $\tilde{\epsilon}$ be as in \eqref{definition epsilon tilde}. We have that
\begin{align*}
\epsilon_{N}\leq \tilde{\epsilon}\leq \frac{1}{1+v^{b}}\psi(e^{\frac{1}{\epsilon}}(y-v^{\alpha}))\mathbbm{1}_{ \{y\geq v^{\alpha}-C_{0}^{-1}\epsilon^{-\frac{1}{m-1}}  e^{-\frac{z}{\epsilon}}\}\cup\{y\leq (\frac{v}{3})^{\alpha}\}}+T_{2}(y,v,z)
\end{align*}
and moreover since $z\leq 1$ it follows that
\begin{align*}
\frac{A}{1+v^{b}}\psi(e^{\frac{1}{\epsilon}}(y-v^{\alpha}))\leq \frac{e^{\frac{z}{\epsilon}}}{1+v^{b}}\psi(e^{\frac{z}{\epsilon}}(y-v^{\alpha})).
\end{align*}
In other words,
\begin{align*}
    \tilde{\epsilon}\leq T_{1}(y,v,z)+T_{2}(y,v,z)+T_{3}(y,v,z).
\end{align*}
Thus, we have that \eqref{continuity argument statement} holds for all $z\in[s,\min\{s+\epsilon,s+\tilde{\delta},1\}],$  . This concludes our proof.
\end{proof}

  \subsection{Iteration argument}\label{subsection region above}
  
We are now able to find an upper bound for $H_{n+1}$ in the region $y\geq (\frac{v}{3})^{\alpha}$. 
  \begin{prop}[Region $y\geq (\frac{v}{3})^{\alpha}$]\label{supersol prop above region}
There exist $\epsilon_{1}\in(0,1)$ and $T(A)\in[0,1]$, sufficiently small and depending on $A$ in \eqref{def t1}, such that for all $\epsilon\leq \epsilon_{1}$ and for all $n\in\mathbb{N}$, the following statement holds. Let  $H_{n}$ be as in \eqref{iteration supersol}. Let $T_{1}$, $T_{2}, T_{3}$ as in \eqref{def t1}-\eqref{def t3} with  $b
\geq \max\{\overline{b}(\gamma,\alpha),\gamma+1\}$ and $m>\max\{\frac{2(\gamma+1)}{\alpha},\frac{b}{\alpha}+1\}$, where  $\overline{b}(\gamma,\alpha)$ is as in  \eqref{definition overline b}. Assume that $H_{n}\leq \sum_{i=1}^{3}T_{i}$. Then it holds that $H_{n+1}(y,v,t)\leq T_{1}+T_{2},$ for all $t\in[0,T(A)]$ and $y\geq (\frac{v}{3})^{\alpha}$.
\end{prop}
\begin{proof}
Let $y\geq (\frac{v}{3})^{\alpha}$. Using the mild formulation of \eqref{iteration supersol}, it holds that
    \begin{align*}
&H_{n+1}(y,v,t)=S_{\epsilon}(t)[H_{n+1}(y,v,0) ] D[H_{n}](y,v,0,t)\\
&+\int_{0}^{t}\int_{(0,\frac{v}{2})}D[H_{n}](y,v,s,t)
S_{\epsilon}(t-s)\bigg[K(v-v',v')H_{n+1}(y,v-v',s) H_{n}(y,v',s)\bigg]\der s,\nonumber 
\end{align*}
with $D[H_{n}]$ defined as in \eqref{definition s mild}. Since $D[H_{n}]\leq 1$ and $K(v-w,w)\leq 2K_{0}v^{\gamma}$ when $w\in[0,\frac{v}{2}]$, we further have that
\begin{align*}
H_{n+1}(y,v,t)\leq S_{\epsilon}(t)[H_{n+1}(y,v,0) ] +2K_{0}v^{\gamma}\int_{0}^{t}\int_{(0,\frac{v}{2})}
S_{\epsilon}(t-s)\bigg[H_{n+1}(y,v-v',s) H_{n}(y,v',s)\bigg]\der s.\nonumber 
\end{align*}
We prove that if $y\geq (\frac{v}{3})^{\alpha}$, then
\begin{align}\label{step i}
S_{\epsilon}(t)H_{\epsilon,\textup{in}}(y,v)\leq \frac{T_{1}+T_{2}}{2}
\end{align}
and 
\begin{align}\label{step ii}
2K_{0}v^{\gamma}\int_{0}^{t}\int_{(0,\frac{v}{2})}
S_{\epsilon}(t-s)\bigg[H_{n+1}(y,v-v',s) H_{n}(y,v',s)\bigg]\der s\leq \frac{T_{1}+T_{2}}{2}.
\end{align}
We then conclude that $H_{n+1}\leq T_{1}+T_{2}$ when $y\geq (\frac{v}{3})^{\alpha}$ as desired. We first prove \eqref{step i}. We thus bound $ S_{\epsilon}(t)H_{\epsilon,\textup{in}}(y,v)$ when $y\geq (\frac{v}{3})^{\alpha}$. From (\ref{initial condition decay}), it holds that
\begin{align*}
     S_{\epsilon}(t)H_{\epsilon,\textup{in}}(y,v)=e^{\frac{t}{\epsilon}}H_{\epsilon,\textup{in}}(e^{\frac{t}{\epsilon}}(y-v^{\alpha})+v^{\alpha},v)\leq \frac{Ae^{\frac{t}{\epsilon}}}{(1+v^{b})(1+e^{\frac{mt}{\epsilon}}|y-v^{\alpha}|^{m})}=\frac{Ae^{\frac{t}{\epsilon}}\psi(e^{\frac{t}{\epsilon}}(y-v^{\alpha}))}{1+v^{b}}.
\end{align*}
On the one hand,
\begin{align*}
\frac{Ae^{\frac{t}{\epsilon}}\psi(e^{\frac{t}{\epsilon}}(y-v^{\alpha}))}{1+v^{b}}\chi_{1}(y,v,t)\leq \frac{T_{1}}{2},
\end{align*}
where $\chi_{1}$ is as in \eqref{chi1}. On the other hand, if  $y\leq v^{\alpha}-C_{0}^{-1}\epsilon^{-\frac{1}{m-1}}  e^{-\frac{t}{\epsilon}}$ then it holds that $e^{\frac{t}{\epsilon}}|y-v^{\alpha}|\geq C_{0}^{-1}\epsilon^{-\frac{1}{m-1}} $ and thus
\begin{align*}
    Ae^{\frac{t}{\epsilon}}\psi(e^{\frac{t}{\epsilon}}(y-v^{\alpha}))\chi_{2}(y,v,t)&\leq \frac{A}{e^{\frac{t(m-1)}{\epsilon}}|y-v^{\alpha}|^{m}}\chi_{2}(y,v,t)\leq \frac{AC_{0}^{m-1}\epsilon}{|y-v^{\alpha}|}\chi_{2}(y,v,t)\leq \frac{T_{2}}{2},
\end{align*}
where $\chi_{2}$ is as in \eqref{chi2}. We now focus on \eqref{step ii}. By assumption it holds that $H_{n}\leq T_{1}+T_{2}+T_{3}$, where  $T_{1},T_{2},T_{3}$ are as in \eqref{def t1}-\eqref{def t3}. We also have that   $H_{\epsilon,\textup{in}}(y,v)\leq T_{1}(y,v,0)+T_{2}(y,v,0)+T_{3}(y,v,0)$. By Proposition \ref{continuity argument prop} and Remark \ref{rmk compact support for delta uniform}, this implies that $H_{n+1}(y,v,z)\leq 4[T_{1}(y,v,z)+T_{2}(y,v,z)+T_{3}(y,v,z)]$, for all $z\in[0,\min\{\frac{\epsilon}{m},1,\tilde{\delta}\}]$, with $\tilde{\delta}$ as in Proposition \ref{continuity argument prop}. Thus, for all $t\in[0,\min\{\frac{\epsilon}{m},1,\tilde{\delta}\}]$, it holds that
\begin{align*}
H_{n+1}(y,v,t)&\leq S_{\epsilon}(t)[H_{n+1}(y,v,0) ] +2K_{0}v^{\gamma}\int_{0}^{t}\int_{(0,\frac{v}{2})}
S_{\epsilon}(t-s)\bigg[H_{n+1}(y,v-v',s) H_{n}(y,v',s)\bigg]\der s
\\
&\leq \frac{T_{1}+T_{2}}{2}+8K_{0}v^{\gamma}\sum_{i,j=1}^{3}\int_{0}^{t}\int_{(0,\frac{v}{2})}
S_{\epsilon}(t-s)\bigg[T_{i}(y,v-v',s) T_{j}(y,v',s)\bigg]\der s.
\end{align*}
We will prove in Proposition \ref{main ingredient} that for $m>\frac{2(\gamma+1)}{\alpha}$ and $b\geq \gamma+1$, there exist a time $T(A)$, which depends on $A$,  and $\epsilon_{1}\in(0,1)$, which depends on $A$, such that 
\begin{align}\label{main ingredient ineq}
8K_{0}v^{\gamma} & \sum_{i,j=1}^{3}\int_{0}^{t}\int_{(0,\frac{v}{2})}
S_{\epsilon}(t-s)\bigg[T_{i}(y,v-v',s) T_{j}(y,v',s)\bigg]\der s\leq \frac{T_{1}+T_{2}}{2}, \\
&\textup{ for all }y\geq (\frac{v}{3})^{\alpha}, t\leq \frac{1}{T(A)}, \epsilon\leq \epsilon_{1}\nonumber
    \end{align}
and thus 
\begin{align*}
H_{n+1}(y,v,t)\leq T_{1}+T_{2},\textup{ for all } t\in[0,\min\{\epsilon,1,\tilde{\delta}\}], \quad y\geq (\frac{v}{3})^{\alpha}.\end{align*}
 From Section \ref{section four}, we also have that if $b\geq \overline{b}(\gamma,\alpha)$, $m\geq \frac{b}{\alpha}+1$ it holds that
 \begin{align*}
H_{n+1}(y,v,t)\leq T_{3}, \quad y\leq (\frac{v}{3})^{\alpha}.
\end{align*}
Thus $H_{n+1}(y,v,t)\leq T_{1}+T_{2}+T_{3}$ at $t:=\min\{\epsilon,1,\tilde{\delta}\}$ and we can then use Proposition \ref{continuity argument prop} and iterate the argument in order to extend the result for all times $t\leq \frac{1}{T(A)}$, where $T(A)$ is as in \eqref{main ingredient ineq}.
\end{proof}

\begin{rmk}
Here and in all the following, we only treat the case when $(\frac{v}{3})^{\alpha} \leq v^{\alpha}-C_{0}^{-1}\epsilon^{-\frac{1}{m-1}}  e^{-\frac{t}{\epsilon}}$ since the case $(\frac{v}{3})^{\alpha} \geq v^{\alpha}-C_{0}^{-1}\epsilon^{-\frac{1}{m-1}}  e^{-\frac{t}{\epsilon}}$ follows similarly and with fewer technicalities. On the one hand, if $(\frac{v}{3})^{\alpha} \geq v^{\alpha}-C_{0}^{-1}\epsilon^{-\frac{1}{m-1}}  e^{-\frac{t}{\epsilon}}$, then $\{y\leq  v^{\alpha}-C_{0}^{-1}\epsilon^{-\frac{1}{m-1}}  e^{-\frac{t}{\epsilon}}\}\subseteq \{y\leq(\frac{v}{3})^{\alpha}\}$ and thus we can bound the region $\{y\leq  v^{\alpha}-C_{0}^{-1}\epsilon^{-\frac{1}{m-1}}  e^{-\frac{t}{\epsilon}}\}$ as in Section \ref{section four}. On the other hand, we can bound the region $\{y\geq   v^{\alpha}-C_{0}^{-1}\epsilon^{-\frac{1}{m-1}}  e^{-\frac{t}{\epsilon}}\}$ using the estimates from Proposition \ref{prop estimates t1 t1} and then we replace $\mathbbm{1}_{\{y\geq (\frac{v}{3})^{\alpha}\}}$ in the proof of Proposition \ref{main ingredient} by  $\mathbbm{1}_{\{y\geq v^{\alpha}-C_{0}^{-1}\epsilon^{-\frac{1}{m-1}}  e^{-\frac{t}{\epsilon}}\}}$. Notice we do not use Proposition \ref{Estimates for the combinations} in this case.
\end{rmk}
\begin{rmk}\label{remark region}
Notice that in order to prove \eqref{main ingredient ineq}, since we are estimating the function on the set $\{y\geq (\frac{v}{3})^{\alpha}\}$, we actually only need to bound
\begin{align*}
     S_{\epsilon}(t)H_{\epsilon,\textup{in}}(y,v)+ \mathbbm{1}_{\{y\geq (\frac{v}{3})^{\alpha}\}} v^{\gamma}\int_{0}^{t}S_{\epsilon}(t-s)\int_{0}^{\frac{v}{2}} T_{i}(y,v-w,s)T_{j}(y,w,s)\der w\der s.
\end{align*}
\end{rmk}

We now focus on proving \eqref{main ingredient ineq}. More precisely, we prove that

\begin{prop}\label{main ingredient}
Let $\epsilon\in(0,1)$ and $S_{\epsilon}$ as in \eqref{definition semigroup}. Let $m>\frac{2(\gamma+1)}{\alpha}$ and $b\geq \gamma+1$. There exist a time $T(A)\in[0,1]$, $\epsilon_{1}\in(0,1)$, and $C_{0}$, which are dependent on $A$ as in \eqref{initial condition decay}  and on the parameters $\alpha,\gamma,b,m$, such that  if $t\leq T(A)$, $\epsilon\leq \epsilon_{1}$, then the following estimate holds for all $y\in\mathbb{R}$ and $v\in(0,\infty)$. 
    \begin{align}\label{prop main ingredient ineq}
\mathbbm{1}_{\{y\geq (\frac{v}{3})^{\alpha}\}}   8K_{0} \sum_{i,j=1}^{3} v^{\gamma}\int_{0}^{t}S_{\epsilon}(t-s)\int_{0}^{\frac{v}{2}} T_{i}(y,v-w,s)T_{j}(y,w,s)\der w\der s\leq \frac{T_{1}(y,v,t)+T_{2}(y,v,t)}{2}.
\end{align} 
\end{prop}
\begin{rmk}
We remember that $T_{2}$ as in \eqref{def t2} depends on $C_{0}$ and this is the reason why we require $C_{0}$ to depend on $A$ in the statement of Proposition \ref{main ingredient} in order for \eqref{prop main ingredient ineq} to hold. We refer to \eqref{defc0} for the exact choice of $C_{0}$ and to Proposition \ref{Estimates for the combinations} and to \eqref{e3} in order to see why this definition was required.
\end{rmk}

We first prove upper bounds for $\int_{0}^{\frac{v}{2}} T_{i}(y,v-w,s)T_{j}(y,w,s)\der w$ and then make use of Lemma \ref{properties of semigroup lemma}. However, despite the fact that we only need to analyze the region ${\{y\geq (\frac{v}{3})^{\alpha}\}}$ as noticed in Remark \ref{remark region}, when proving the inequalities before applying the semigroup, we need to analyze a wider range of values of $y$. More precisely, we have the following.
\begin{lem}
Let $\epsilon\in(0,1)$, $t\in[0,1],$ and $s\in[0,\frac{t}{2}]$. Let $S_{\epsilon}$ as in \eqref{definition semigroup}, $y\in\mathbb{R},$ and $v\in(0,\infty)$.  There exists $\overline{\epsilon}\in(0,1)$ sufficiently small such that for all $\epsilon\leq \overline{\epsilon}$, there exists $X(\epsilon)\leq 0$ such that $S_{\epsilon}(t-s)\big[\mathbbm{1}_{\{y\geq X(\epsilon)\}}\big] =\mathbbm{1}_{\{y\geq (\frac{v}{3})^{\alpha}\}}$.
\end{lem}
\begin{proof}
It holds that
\begin{align*}
S_{\epsilon}(t-s)\big[\mathbbm{1}_{\{y\geq X(\epsilon)\}}\big] =\mathbbm{1}_{\{y\geq (\frac{v}{3})^{\alpha}\}}&\textup{ if and only if }\mathbbm{1}_{\{e^{\frac{t-s}{\epsilon}}(y-v^{\alpha})+v^{\alpha}\geq X(\epsilon)\}} =\mathbbm{1}_{\{y\geq (\frac{v}{3})^{\alpha}\}}.
\end{align*}
We then have that $e^{-\frac{t-s}{\epsilon}}(X(\epsilon)-v^{\alpha})+v^{\alpha}= \Big(\frac{v}{3}\Big)^{\alpha}$ if and only if $ X(\epsilon)=e^{\frac{t-s}{\epsilon}}\Big(\Big(\frac{v}{3}\Big)^{\alpha}-v^{\alpha}\Big)+v^{\alpha}=v^{\alpha}\bigg(1-e^{\frac{t-s}{\epsilon}}\bigg)+e^{\frac{t-s}{\epsilon}}\Big(\frac{v}{3}\Big)^{\alpha}.$ In other words, for sufficiently small $\epsilon$ we can have that $X(\epsilon)\leq 0$.
\end{proof}

We have the following estimates. We first prove estimates for the convolution of $T_{1}, T_{3}$ with themselves considering different regions of $y$.
\begin{prop}[Estimates for $T_{1}$ and $T_{3}$]\label{prop estimates t1 t1} Let $\epsilon\in(0,1)$, $t\in[0,1],$ and $s\in[0,t]$. Let $S_{\epsilon}$ as in \eqref{definition semigroup}. Let $i,j\in\{1,3\}$.  Let $b\geq \gamma+1-\alpha$ and $m>\frac{2(\gamma+1)}{\alpha}$. Let $T_{1}$ and $T_{3}$ be as in \eqref{def t1} and \eqref{def t3}. There exists a constant $C_{1}(\alpha,\gamma,b,m)>0$, which depends on $\alpha,\gamma,b,m$, such that for all $y\in\mathbb{R}$ and $v\in(0,\infty)$ the following estimates hold. 
    \begin{align}
      v^{\gamma}\int_{0}^{t}S_{\epsilon}(t-s)\bigg[\int_{0}^{\frac{v}{2}} T_{i}(y,v-w,s)T_{j}(y,w,s) & \mathbbm{1}_{\{y\geq v^{\alpha}\}}\der w\bigg]\der s\leq  \frac{C_{1}(\alpha,\gamma,b,m)}{1+v^{b}} \frac{A^{2} t e^{\frac{t}{\epsilon}}\mathbbm{1}_{\{y\geq v^{\alpha}\}}}{(1+e^{\frac{t}{\epsilon}}|y-v^{\alpha}|)^{m}}  ;\label{estimate for t1 t1 y greater v alpha}\\
       v^{\gamma}\int_{0}^{t}S_{\epsilon}(t-s) \bigg[\int_{0}^{\frac{v}{2}} T_{i}(y,v-w,s)T_{j}(y,w,s) &\mathbbm{1}_{\{y\leq v^{\alpha}\}}\der w\bigg]\der s\nonumber\\
       & \leq \frac{  C_{1}(\alpha,\gamma,b,m)M_{2}^{2}}{1+v^{b}}\frac{A^{2}\epsilon \min\{1,e^{\frac{t}{\epsilon}}|y-v^{\alpha}|\}}{|y-v^{\alpha}|}\label{estimate for t1 t1 y negative}.
     \end{align}
\end{prop}
We then focus on proving estimates for the remaining possible combinations of the convolution of $T_{i}$ with $T_{j}$.

\begin{prop}[Estimates for the rest of combinations]\label{Estimates for the combinations}Let $\epsilon\in(0,1)$, $t\in[0,1],$ and $s\in[0,t]$. Let $S_{\epsilon}$ as in \eqref{definition semigroup}. Let $i\in\{2,3\}$. Let $b\geq\gamma+1$ and $m>\frac{\gamma+1}{\alpha}+2$. Let $T_{1},\ldots,T_{3}$ as in \eqref{def t1}-\eqref{def t3}. There exists a constant $C_{2}(\alpha,\gamma,b,m)>0$, which depends on $\alpha,\gamma,b,m$, such that  for all $y\in\mathbb{R}$ and $v\in(0,\infty)$ the following estimates hold.
\begin{align}
      v^{\gamma} &\int_{0}^{t}S_{\epsilon}(t-s)\int_{0}^{\frac{v}{2}} T_{i-1}(y,v-w,s)T_{i}(y,w,s)\der w\der s   =0; \label{estimate for t1 t2}\\
   v^{\gamma} &\int_{0}^{t}S_{\epsilon}(t-s)\int_{0}^{\frac{v}{2}} T_{i}(y,v-w,s)T_{i-1}(y,w,s)\der w\der s  \nonumber\\
   &   \leq \frac{C_{2}(\alpha,\gamma,b,m)M_{2}}{1+v^{b}}\frac{A^{3}C_{0}^{m-1} \epsilon[t+\epsilon+C_{0}^{-(m-1)}] }{v^{\alpha}-y}\mathbbm{1}_{\{y\leq v^{\alpha}-e^{-\frac{t}{\epsilon}}C^{-1}_{0}\epsilon^{-\frac{1}{m-1}}\}};\label{estimate for t2 t1}\\
   v^{\gamma} &\int_{0}^{t}S_{\epsilon}(t-s)\int_{0}^{\frac{v}{2}}T_{2}(y,v-w,s)T_{2}(y,w,s)\der w\der s \nonumber\\
   &\leq \frac{C_{2}(\alpha,\gamma,b,m)}{1+v^{b}}\frac{A^{4}C_{0}^{2(m-1)} \epsilon (t+\epsilon)^{2} }{v^{\alpha}-y}\mathbbm{1}_{\{y\leq v^{\alpha}-e^{-\frac{t}{\epsilon}}C^{-1}_{0}\epsilon^{-\frac{1}{m-1}}\}}.\label{estimate for t2 t2}
\end{align}
\end{prop}
Assuming the estimates in Proposition \ref{prop estimates t1 t1} and Proposition \ref{Estimates for the combinations} hold true, we are now able to prove Proposition \ref{main ingredient}. 

\begin{proof}[Proof of Proposition \ref{main ingredient}]
We wish to prove \eqref{prop main ingredient ineq}. We remember that $C_{0}^{m-1}=M_{1}A$, see \eqref{defc0}. More precisely, $C_{0}$ is of order $A^{\frac{1}{m-1}}$. We choose $M_{1}$, depending on $M_{2}$, such that 
\begin{align}\label{choice of m1}
    8K_{0}\max\{C_{1}(\alpha,\gamma,b,m),C_{2}(\alpha,\gamma,b,m)\}M_{2}^{2}\leq \frac{M_{1}}{24},
\end{align} 
where $C_{1}(\alpha,\gamma,b,m)$ is as in Proposition \ref{prop estimates t1 t1} and $C_{2}(\alpha,\gamma,b,m)$ is as in Proposition \ref{Estimates for the combinations}.

From \eqref{estimate for t1 t1 y greater v alpha}, we can choose a sufficiently small $T$, which depends on $A$, such that for all $t\leq T$, it holds that
\begin{align*}
 8K_{0}\sum_{i,j\in\{1,3\}}v^{\gamma}\int_{0}^{t}S_{\epsilon}(t-s)\bigg[\int_{0}^{\frac{v}{2}}&  T_{i}(y,v-w,s)T_{j}(y,w,s)\mathbbm{1}_{\{y\geq v^{\alpha}\}}\der w\bigg]\der s\\
 &\leq  8K_{0}C_{1}(\alpha,\gamma,b,m)At T_{1}(y,v,t)\leq \frac{T_{1}(y,v,t)}{12}.
 \end{align*}
We now  estimate \eqref{estimate for t1 t1 y negative}.  We prove in this case that 
\begin{align}\label{use also later}
     \mathbbm{1}_{\{y\geq (\frac{v}{3})^{\alpha}\}} \frac{  8K_{0}C_{1}(\alpha,\gamma,b,m)M_{2}^{2}}{1+v^{b}}\frac{A^{2}\epsilon \min\{1,e^{\frac{t}{\epsilon}}|y-v^{\alpha}|\}}{|y-v^{\alpha}|}& \leq \frac{T_{1}(y,v,t)}{12}+\frac{T_{2}(y,v,t)}{12}.
\end{align} 
\begin{itemize}
    \item Let us assume first that $e^{\frac{t}{\epsilon}}|y-v^{\alpha}|\leq 1$. Then 
\end{itemize}\begin{align}
    \frac{ \epsilon}{|y-v^{\alpha}|(1+v^{b})}\min\{1,e^{\frac{t}{\epsilon}}|y-v^{\alpha}|\}=\frac{\epsilon e^{\frac{t}{\epsilon}}}{(1+v^{b})}\leq \frac{2\epsilon e^{\frac{t}{\epsilon}}}{(1+e^{\frac{m t}{\epsilon}}|y-v^{\alpha}|^{m})(1+v^{b})}.
\end{align}
Thus, if $e^{\frac{t}{\epsilon}}|y-v^{\alpha}|\leq 1$, then for sufficiently small $\epsilon$, which depends on $A$, it holds that
\begin{align}\label{e2}
&  \frac{  8K_{0}C_{1}(\alpha,\gamma,b,m)M_{2}^{2}}{1+v^{b}} \frac{A^{2}\epsilon \min\{1,e^{\frac{t}{\epsilon}}|y-v^{\alpha}|\}}{|y-v^{\alpha}|}\nonumber \\
  &\leq 8K_{0} C_{1}(\alpha,\gamma,b,m)M_{2}^{2} A\epsilon T_{1}(y,v,t)\leq \frac{T_{1}(y,v,t)}{24}.
\end{align}
\begin{itemize}
    \item If $C_{0}^{-1}\epsilon^{-\frac{1}{m-1}}\geq e^{\frac{t}{\epsilon}}|y-v^{\alpha}|\geq 1$, then 
\end{itemize}
\begin{align}
&\frac{\epsilon}{|y-v^{\alpha}|}\min\{1,e^{\frac{t}{\epsilon}}|y-v^{\alpha}|\}\mathbbm{1}_{\{1\leq e^{\frac{t}{\epsilon}}|y-v^{\alpha}|\leq C_{0}^{-1}\epsilon^{-\frac{1}{m-1}}\}} = \frac{\epsilon}{|y-v^{\alpha}|}\mathbbm{1}_{\{1\leq e^{\frac{t}{\epsilon}}|y-v^{\alpha}|\leq C_{0}^{-1}\epsilon^{-\frac{1}{m-1}}\}}\nonumber\\
&= \frac{\epsilon e^{\frac{t}{\epsilon}}}{e^{\frac{t}{\epsilon}}|y-v^{\alpha}|}\mathbbm{1}_{\{1\leq e^{\frac{t}{\epsilon}}|y-v^{\alpha}|\leq C_{0}^{-1}\epsilon^{-\frac{1}{m-1}}\}}\leq \frac{ C_{0}^{-(m-1)}e^{\frac{t}{\epsilon}}}{(e^{\frac{t}{\epsilon}}|y-v^{\alpha}|)^{m}}\mathbbm{1}_{\{1\leq e^{\frac{t}{\epsilon}}|y-v^{\alpha}|\leq C_{0}^{-1}\epsilon^{-\frac{1}{m-1}}\}}\nonumber\\
&\leq \frac{ 2 C_{0}^{-(m-1)}e^{\frac{t}{\epsilon}}}{1+(e^{\frac{t}{\epsilon}}|y-v^{\alpha}|)^{m}}\mathbbm{1}_{\{1\leq e^{\frac{t}{\epsilon}}|y-v^{\alpha}|\leq C_{0}^{-1}\epsilon^{-\frac{1}{m-1}}\}}.\label{c0 initial mention}
\end{align}
It thus follows by \eqref{choice of m1} that 
\begin{align}
 & \frac{  8K_{0}C_{1}(\alpha,\gamma,b,m)M_{2}^{2}}{1+v^{b}} \frac{A^{2}\epsilon \min\{1,e^{\frac{t}{\epsilon}}|y-v^{\alpha}|\}}{|y-v^{\alpha}|}\nonumber\\
  &\leq 8K_{0}C_{1}(\alpha,\gamma,b,m)M_{2}^{2}A C_{0}^{-(m-1)}T_{1}(y,v,t)\leq \frac{T_{1}(y,v,t)}{24}.\label{e3}
\end{align}
\begin{itemize}
    \item If $e^{\frac{t}{\epsilon}}|y-v^{\alpha}|\geq C_{0}^{-1}\epsilon^{-\frac{1}{m-1}}$, then by the choice of $T_{2}$ in \eqref{def t2}
\end{itemize}
\begin{align}\label{region above dirac constant bound}
&\mathbbm{1}_{\{v^{\alpha}\geq y\geq (\frac{v}{3})^{\alpha}\}} \frac{  8K_{0}C_{1}(\alpha,\gamma,b,m) M_{2}^{2}}{1+v^{b}} \frac{A^{2}\epsilon}{|y-v^{\alpha}|}\mathbbm{1}_{\{ e^{\frac{t}{\epsilon}}|y-v^{\alpha}|\geq C_{0}^{-1}\epsilon^{-\frac{1}{m-1}}\}}\nonumber\\
&\leq  8K_{0}C_{1}(\alpha,\gamma,b,m)M_{2}^{2} C_{0}^{-(m-1)} T_{2}(y,v,t)\leq  \frac{T_{2}(y,v,t)}{12}.
\end{align}

Combining estimates \eqref{e1}-\eqref{region above dirac constant bound} together with the fact that $S_{\epsilon}(t-s)(\mathbbm{1}_{\{y\leq v^{\alpha}\}})=\mathbbm{1}_{\{y\leq v^{\alpha}\}}$ which follows by Item $1$ of Lemma 
\ref{properties of semigroup lemma}, we obtain \eqref{use also later}.

We now proceed with \eqref{estimate for t2 t1}. If $t,\epsilon$ are sufficiently small and for sufficiently large $C_{0}$, then
\begin{align*}
  & \mathbbm{1}_{\{y\geq (\frac{v}{3})^{\alpha}\}} 8K_{0}     v^{\gamma} \sum_{i\in \{2,3\}}\int_{0}^{t}S_{\epsilon}(t-s)\int_{0}^{\frac{v}{2}} T_{i}(y,v-w,s)T_{i-1}(y,w,s)\der w\der s\\
      &  \leq \frac{8K_{0} C_{2}(\alpha,\gamma,b,m)M_{2}}{1+v^{b}}\frac{A^{3}C_{0}^{m-1} \epsilon[t+\epsilon+C_{0}^{-(m-1)}] }{v^{\alpha}-y}\chi_{2}(y,v,t) \leq \frac{T_{2}(y,v,t)}{24}.
\end{align*}
This is since by \eqref{choice of m1}, it holds that
\begin{align}\label{most important c0 dependence}
    \frac{8K_{0} C_{2}(\alpha,\gamma,b,m) M_{2} C_{0}^{-(m-1)} A}{1+v^{b}}\frac{A^{2}C_{0}^{m-1} \epsilon }{v^{\alpha}-y}\chi_{2}(y,v,t)  &= 4K_{0} C_{2}(\alpha,\gamma,b,m) M_{2} C_{0}^{-(m-1)} AT_{2}(y,v,t)\nonumber\\
    &\leq \frac{T_{2}(y,v,t)}{48}.
\end{align}

Similarly, for \eqref{estimate for t2 t2}, there exist $T$ and $\epsilon_{1}$, which depend on $A$, such that for all $t\leq T$ and $\epsilon\leq \epsilon_{1}$, then
\begin{align*}
  & \mathbbm{1}_{\{y\geq (\frac{v}{3})^{\alpha}\}}  8K_{0} v^{\gamma} \sum_{i\in \{2,3\}}\int_{0}^{t}S_{\epsilon}(t-s)\int_{0}^{\frac{v}{2}} T_{2}(y,v-w,s)T_{2}(y,w,s)\der w\der s\\
  &\leq \frac{8K_{0}C_{2}(\alpha,\gamma,b,m)}{1+v^{b}}\frac{A^{4}C_{0}^{2(m-1)} \epsilon(t+\epsilon)^{2} }{v^{\alpha}-y}\chi_{2}(y,v,t)\\
      &  \leq 8K_{0}C_{2}(\alpha,\gamma,b,m)A^{2}C_{0}^{m-1} (t+\epsilon)^{2} T_{2}(y,v,t)\leq \frac{T_{2}(y,v,t)}{24}.
\end{align*}
This concludes our proof.
\end{proof}

The rest of this section is dedicated to proving Proposition \ref{prop estimates t1 t1} and Proposition \ref{Estimates for the combinations}.
\subsection{Proof of Proposition \ref{prop estimates t1 t1}}
\begin{proof}[Proof of Proposition 
\ref{prop estimates t1 t1}] Let $i,j\in\{1,3\}$. We focus on finding an upper bound for 

\begin{align}\label{semigroup t1 t1}
    v^{\gamma}\int_{0}^{t}S_{\epsilon}(t-s)\int_{0}^{\frac{v}{2}} T_{i}(y,v-w,s)T_{j}(y,w,s)\der w\der s.
\end{align}
We notice that since $w\in[0,\frac{v}{2}]$, it holds that
\begin{align*}
   v^{\gamma}&\int_{0}^{t}S_{\epsilon}(t-s)\int_{0}^{\frac{v}{2}} T_{i}(y,v-w,s)T_{j}(y,w,s)\der w\der s\\
    &\leq \int_{0}^{\frac{v}{2}}\frac{C(\alpha,\gamma,b,m) A^{2}v^{\gamma}}{(1+w^{b})(1+(v-w)^{b})}\psi(e^{\frac{s}{\epsilon}}(y-w^{\alpha}))\psi(e^{\frac{s}{\epsilon}}(y-(v-w)^{\alpha})\der w\\
    &\leq \frac{C(\alpha,\gamma,b,m)  A^{2} v^{\gamma}}{1+v^{b}}\int_{0}^{\frac{v}{2}}\frac{1}{1+w^{b}}\psi(e^{\frac{s}{\epsilon}}(y-w^{\alpha}))\psi(e^{\frac{s}{\epsilon}}(y-(v-w)^{\alpha})\der w.
\end{align*}
Since $w\in[0,\frac{v}{2}]$, it holds that $w^{\alpha}\leq (v-w)^{\alpha}$, for all $v>0, w\in[0,\frac{v}{2}]$. By \eqref{ineq 2}, we have that
\begin{align}
&\int_{0}^{\frac{v}{2}}\frac{v^{\gamma}}{1+w^{b}}e^{\frac{2s}{\epsilon}}\psi(e^{\frac{s}{\epsilon}}(y-w^{\alpha}))\psi(e^{\frac{s}{\epsilon}}(y-(v-w)^{\alpha})\der w\nonumber\\
&= \int_{0}^{\frac{v}{2}}\frac{v^{\gamma}}{1+w^{b}}\frac{e^{\frac{2s}{\epsilon}}}{(1+e^{\frac{s}{\epsilon}}|y-w^{\alpha}|)^{m}}\frac{\der w}{(1+e^{\frac{s}{\epsilon}}|y-(v-w)^{\alpha}|)^{m}}\nonumber\\
&\leq \int_{0}^{\frac{v}{2}}\frac{C(m)v^{\gamma}}{1+w^{b}}\frac{e^{\frac{2s}{\epsilon}}}{(1+e^{\frac{s}{\epsilon}}|y-w^{\alpha}|)^{m}}\frac{\der w}{(1+e^{\frac{s}{\epsilon}}|(v-w)^{\alpha}-w^{\alpha}|)^{m}}\nonumber\\
&+\int_{0}^{\frac{v}{2}}\frac{C(m)v^{\gamma}}{1+w^{b}}\frac{e^{\frac{2s}{\epsilon}}}{(1+e^{\frac{s}{\epsilon}}|y-(v-w)^{\alpha}|)^{m}}\frac{\der w}{(1+e^{\frac{s}{\epsilon}}|(v-w)^{\alpha}-w^{\alpha}|)^{m}}=:I_{1}+I_{2}.\label{first step y greater v alpha}
\end{align}

\textbf{Proof of \eqref{estimate for t1 t1 y greater v alpha}.} We will prove that there exists a constant $C(\alpha,\gamma,b,m)$, which depends on $\alpha$, $\gamma$, $b$, $m$, such that
\begin{align*}
 v^{\gamma}&\int_{0}^{t}S_{\epsilon}(t-s)\bigg[\int_{0}^{\frac{v}{2}} T_{1}(y,v-w,s)T_{1}(y,w,s)\der w\mathbbm{1}_{\{y\geq v^{\alpha}\}}\bigg]\der s\\
 \leq &
4A^{2}\int_{0}^{t}S_{\epsilon}(t-s)\bigg[\int_{0}^{\frac{v}{2}}\frac{e^{\frac{2s}{\epsilon}}v^{\gamma}}{(1+w^{b})(1+(v-w)^{b})} \psi(e^{\frac{s}{\epsilon}}(y-w^{\alpha}))\psi(e^{\frac{s}{\epsilon}}(y-(v-w)^{\alpha})\der w \mathbbm{1}_{\{y\geq v^{\alpha}\}}\bigg]\der s\\
\leq & \frac{C(\alpha,\gamma,b,m) }{1+v^{b}}\frac{ A^{2}t e^{\frac{t}{\epsilon}}}{1+e^{\frac{tm}{\epsilon}}|y-v^{\alpha}|^{m}} \mathbbm{1}_{\{y\geq v^{\alpha}\}}.
\end{align*}
In order to do this, we will first prove some estimates for
\begin{align}\label{estimate t1 t1}
e^{\frac{2s}{\epsilon}}\int_{0}^{\frac{v}{2}}\frac{v^{\gamma}}{(1+w^{b})(1+(v-w)^{b})}\psi(e^{\frac{s}{\epsilon}}(y-w^{\alpha}))\psi(e^{\frac{s}{\epsilon}}(y-(v-w)^{\alpha})\der w
\end{align}
and then use Item $1$ of Lemma 
\ref{properties of semigroup lemma} to obtain a bound for \eqref{semigroup t1 t1}.  Let $\delta\in(0,\frac{1}{2})$ be sufficiently small. We start by estimating $I_{1}$ in \eqref{first step y greater v alpha}. We analyze separately the regions $w\in[0,\delta v]$ and $w\in[\delta v, \frac{v}{2}]$.  Moreover, since we are in the case when $y\geq v^{\alpha}$, we have that $y-w^{\alpha}\geq y-(\frac{v}{2})^{\alpha}\geq C(\alpha)v^{\alpha}$ and $y-w^{\alpha}\geq y-v^{\alpha}\geq 0$. We analyze first $w\in[\delta v, \frac{v}{2}]$. In this case
\begin{align}
   \mathbbm{1}_{\{y\geq v^{\alpha}\}} \int_{\delta v}^{\frac{v}{2}} & \frac{v^{\gamma}}{1+w^{b}}\frac{e^{\frac{2s}{\epsilon}}}{(1+e^{\frac{s}{\epsilon}}|y-w^{\alpha}|)^{m}}\frac{\der w}{(1+e^{\frac{s}{\epsilon}}|(v-w)^{\alpha}-w^{\alpha}|)^{m}}\nonumber\\
&\leq  \mathbbm{1}_{\{y\geq v^{\alpha}\}} \frac{C(\delta,b)e^{\frac{s}{\epsilon}}}{(1+e^{\frac{s}{\epsilon}}|y-v^{\alpha}|)^{m}}  \int_{\delta v}^{\frac{v}{2}} \frac{v^{\gamma}}{1+v^{b}}\frac{e^{\frac{s}{\epsilon}}\der w}{(1+e^{\frac{s}{\epsilon}}|(v-w)^{\alpha}-w^{\alpha}|)^{m}}\nonumber\\
&\leq  \mathbbm{1}_{\{y\geq v^{\alpha}\}} \frac{C(\delta,\alpha,b) e^{\frac{s}{\epsilon}}}{(1+e^{\frac{s}{\epsilon}}|y-v^{\alpha}|)^{m}}   \frac{ v^{\gamma+1-\alpha}}{1+v^{b}}\leq \mathbbm{1}_{\{y\geq v^{\alpha}\}}\frac{C(\delta,\alpha,b) e^{\frac{s}{\epsilon}}}{(1+e^{\frac{s}{\epsilon}}|y-v^{\alpha}|)^{m}},\label{part one}
\end{align}
where in the last line we used Lemma \ref{ref lemma dirac v-w v} for $\beta=m$ and then the fact that $b\geq \gamma+1-\alpha$. Notice also that since $\gamma>1$ and $\alpha\in(0,1)$, we have that $\gamma+1-\alpha>0$.

We now analyze the region where $w\leq \delta v$. Then $(v-w)^{\alpha}-w^{\alpha}>(1-\delta)^{\alpha}v^{\alpha}-\delta^{\alpha} v^{\alpha}\geq C(\delta,\alpha) v^{\alpha}$ since $\delta<\frac{1}{2}$. Then, since $m>\frac{\gamma+1}{\alpha}$, it holds that
  \begin{align}\label{part two}
    \mathbbm{1}_{\{y\geq v^{\alpha}\}}\int_{0}^{\delta v} & \frac{v^{\gamma}}{1+w^{b}}\frac{e^{\frac{2s}{\epsilon}}}{(1+e^{\frac{s}{\epsilon}}|y-w^{\alpha}|)^{m}}\frac{\der w}{(1+e^{\frac{s}{\epsilon}}|(v-w)^{\alpha}-w^{\alpha}|)^{m}}\nonumber\\
&\leq   \mathbbm{1}_{\{y\geq v^{\alpha}\}} \frac{e^{\frac{s}{\epsilon}} v^{\gamma}}{(1+e^{\frac{s}{\epsilon}}|y-v^{\alpha}|)^{m}}   \int_{0}^{\delta v} \frac{e^{\frac{s}{\epsilon}}\der w}{(1+e^{\frac{s}{\epsilon}}|(v-w)^{\alpha}-w^{\alpha}|)^{m}}\nonumber\\
&\leq \mathbbm{1}_{\{y\geq v^{\alpha}\}}\frac{C(\delta,\alpha,\gamma) e^{\frac{s}{\epsilon}}}{(1+e^{\frac{s}{\epsilon}}|y-v^{\alpha}|)^{m}}  \int_{0}^{\delta v} \frac{ v^{\gamma}}{1+v^{\gamma+1-\alpha}}\frac{e^{\frac{s}{\epsilon}}\der w}{(1+e^{\frac{s}{\epsilon}}|(v-w)^{\alpha}-w^{\alpha}|)^{m+1-\frac{\gamma+1}{\alpha}}}\nonumber\\
&\leq    \mathbbm{1}_{\{y\geq v^{\alpha}\}}\frac{C(\delta,\alpha,\gamma) e^{\frac{s}{\epsilon}}}{(1+e^{\frac{s}{\epsilon}}|y-v^{\alpha}|)^{m}}   \frac{v^{\gamma+1-\alpha}}{1+v^{\gamma+1-\alpha}},
\end{align}
where in the last line we used Lemma \ref{ref lemma dirac v-w v} for $\beta=m-\frac{\gamma+1}{\alpha}+1$.

We now estimate $I_{2}$ in \eqref{first step y greater v alpha}. We start by analyzing the region $w\in[\delta v, \frac{v}{2}]$. Following the estimates for $I_{1}$, when $y\geq v^{\alpha}$ it holds that $y-(v-w)^{\alpha}\geq y-v^{\alpha}\geq 0$. By Lemma \ref{ref lemma dirac v-w v}, we thus have as before that
\begin{align*}
\mathbbm{1}_{\{y\geq v^{\alpha}\}}\int_{\delta v}^{\frac{v}{2}}\frac{v^{\gamma}}{1+w^{b}}\frac{e^{\frac{2s}{\epsilon}}}{(1+e^{\frac{s}{\epsilon}}|y-(v-w)^{\alpha}|)^{m}}\frac{\der w}{(1+e^{\frac{s}{\epsilon}}|(v-w)^{\alpha}-w^{\alpha}|)^{m}}\leq   \mathbbm{1}_{\{y\geq v^{\alpha}\}} \frac{C(\delta,\alpha,b)e^{\frac{s}{\epsilon}}}{(1+e^{\frac{s}{\epsilon}}|y-v^{\alpha}|)^{m}}  .
\end{align*}

In the region where $w\in[0,\delta v]$, we have that $(v-w)^{\alpha}-w^{\alpha}\geq C(\delta)v^{\alpha}$ and thus we obtain as before that 
  \begin{align*}
  \mathbbm{1}_{\{y\geq v^{\alpha}\}}  \int_{0}^{\delta v} & \frac{v^{\gamma}}{1+w^{b}}\frac{e^{\frac{2s}{\epsilon}}}{(1+e^{\frac{s}{\epsilon}}|y-(v-w)^{\alpha}|)^{m}}\frac{\der w}{(1+e^{\frac{s}{\epsilon}}|(v-w)^{\alpha}-w^{\alpha}|)^{m}}\\
       &\leq   \mathbbm{1}_{\{y\geq v^{\alpha}\}} \frac{C(\delta,\alpha,\gamma)e^{\frac{s}{\epsilon}}}{(1+e^{\frac{s}{\epsilon}}|y-v^{\alpha}|)^{m}}  .
\end{align*}
Combining all the above estimates and since $b>2$, we obtain that
\begin{align*}
   \mathbbm{1}_{\{y\geq v^{\alpha}\}} e^{\frac{2s}{\epsilon}}&\int_{0}^{\frac{v}{2}}\frac{v^{\gamma}}{(1+w^{b})(1+(v-w)^{b})}\psi(e^{\frac{s}{\epsilon}}(y-w^{\alpha}))\psi(e^{\frac{s}{\epsilon}}(y-(v-w)^{\alpha})\der w\\
    &\leq  \mathbbm{1}_{\{y\geq v^{\alpha}\}} \frac{e^{\frac{s}{\epsilon}}}{(1+e^{\frac{s}{\epsilon}}|y-v^{\alpha}|)^{m}}    \frac{C(\delta,\alpha,\gamma,b,m)}{1+v^{b}}
\end{align*}
and then using Item $1$ of Lemma 
\ref{properties of semigroup lemma} it follows that when $y\geq v^{\alpha},$ it holds that
\begin{align*}
   v^{\gamma}\int_{0}^{t}S_{\epsilon}(t-s)\bigg[\int_{0}^{\frac{v}{2}} T_{1}(y,v-w,s)T_{1}(y,w,s)\der w\mathbbm{1}_{\{y\geq v^{\alpha}\}}\bigg]\der s\leq  \mathbbm{1}_{\{y\geq v^{\alpha}\}}\frac{C(\delta,\alpha,\gamma,b,m)t e^{\frac{t}{\epsilon}}}{(1+e^{\frac{t}{\epsilon}}|y-v^{\alpha}|)^{m}}    \frac{A^{2}}{1+v^{b}}.
   \end{align*}
\textbf{Proof of \eqref{estimate for t1 t1 y negative}.}
\begin{itemize}
    \item \textbf{Case $y\leq 0$:}  We will prove estimates for \eqref{estimate t1 t1} in this case too. 
\end{itemize}
We have that $0\geq 2^{-\alpha}(y-v^{\alpha})\geq y-(\frac{v}{2})^{\alpha}$ and thus $| y-(\frac{v}{2})^{\alpha}|\geq  2^{-\alpha}|y-v^{\alpha}|$. Moreover, for $w\in[0,\frac{v}{2}]$, it holds that 
\begin{align}
y-(v-w)^{\alpha}\leq y-(\frac{v}{2})^{\alpha}\leq -(\frac{v}{2})^{\alpha}\leq 0.
\end{align}
We thus have that 
\begin{align}\label{ineq negative y}
    |y-(v-w)^{\alpha}|\geq |y-(\frac{v}{2})^{\alpha}|\geq  2^{-\alpha}|y-v^{\alpha}| \textup{ and }  |y-(v-w)^{\alpha}|\geq C(\alpha) v^{-\alpha}.
\end{align} 

As in \eqref{first step y greater v alpha}, by \eqref{ineq 2}, it holds that
\begin{align*}
e^{\frac{2s}{\epsilon}}\int_{0}^{\frac{v}{2}}\frac{v^{\gamma}}{1+w^{b}}\psi(e^{\frac{s}{\epsilon}}(y-w^{\alpha}))\psi(e^{\frac{s}{\epsilon}}(y-(v-w)^{\alpha})\der w \leq I_{1}+I_{2},
\end{align*}
with $I_{1}$ and $I_{2}$ as in \eqref{first step y greater v alpha}. We start by estimating $I_{2}$ in \eqref{first step y greater v alpha}. We first analyze the region where $w\in[\delta v, \frac{v}{2}]$. By \eqref{ineq negative y}, we have that

\begin{align}
\mathbbm{1}_{\{y\leq 0\}}\int_{\delta v}^{\frac{v}{2}}&\frac{v^{\gamma}}{1+w^{b}}\frac{e^{\frac{2s}{\epsilon}}}{(1+e^{\frac{s}{\epsilon}}|y-(v-w)^{\alpha}|)^{m}}\frac{\der w}{(1+e^{\frac{s}{\epsilon}}|(v-w)^{\alpha}-w^{\alpha}|)^{m}}\nonumber\\
&\leq \mathbbm{1}_{\{y\leq 0\}}\frac{C(\delta,\alpha,b)}{(1+e^{\frac{s}{\epsilon}}|y-v^{\alpha}|)^{m-1}} \int_{\delta v}^{\frac{v}{2}}\frac{v^{\gamma-\alpha}}{1+v^{b}}\frac{e^{\frac{s}{\epsilon}}\der w}{(1+e^{\frac{s}{\epsilon}}|(v-w)^{\alpha}-w^{\alpha}|)^{m}}\label{y neg w close to v term v-w}
\end{align}
and we use Lemma \ref{ref lemma dirac v-w v} in order to conclude that
\begin{align}
\mathbbm{1}_{\{y\leq 0\}}\int_{\delta v}^{\frac{v}{2}}&\frac{v^{\gamma}}{1+w^{b}}\frac{e^{\frac{2s}{\epsilon}}}{(1+e^{\frac{s}{\epsilon}}|y-(v-w)^{\alpha}|)^{m}}\frac{\der w}{(1+e^{\frac{s}{\epsilon}}|(v-w)^{\alpha}-w^{\alpha}|)^{m}}\nonumber\\
&\leq   \mathbbm{1}_{\{y\leq 0\}}  \frac{C(\delta,\alpha,b)}{(1+|y-v^{\alpha}|)^{m-1}}  \frac{ v^{\gamma+1-2\alpha}}{1+v^{b}}\label{y neg w close to v term v-w part two}
\end{align}
in this case.

If  $w\in[0, \delta v]$, by \eqref{ineq negative y} and \eqref{part two}, it follows that
\begin{align*}
\mathbbm{1}_{\{y\leq 0\}}\int_{0}^{\delta v}&\frac{v^{\gamma}}{1+w^{b}}\frac{e^{\frac{2s}{\epsilon}}}{(1+e^{\frac{s}{\epsilon}}|y-(v-w)^{\alpha}|)^{m}}\frac{\der w}{(1+e^{\frac{s}{\epsilon}}|(v-w)^{\alpha}-w^{\alpha}|)^{m}}\leq    \mathbbm{1}_{\{y\leq 0\}}\frac{C(\delta,\alpha,\gamma)}{(1+|y-v^{\alpha}|)^{m-1}}.
\end{align*}

We now estimate $I_{1}$ in \eqref{first step y greater v alpha}.  We analyze the region $w\in[\delta v, \frac{v}{2}]$. In this region, it holds that $y-w^{\alpha}\leq y-\delta^{\alpha} v^{\alpha}\leq -\delta^{\alpha} v^{\alpha}\leq 0$ and $y-w^{\alpha}\leq y-\delta^{\alpha} v^{\alpha}\leq \delta^{\alpha}(y-v^{\alpha})\leq 0$. As such, we can proceed as in \eqref{y neg w close to v term v-w} and \eqref{y neg w close to v term v-w part two} in order to prove that
\begin{align*}
   \mathbbm{1}_{\{y\leq 0\}} \int_{\delta v}^{\frac{v}{2}} & \frac{v^{\gamma}}{1+w^{b}}\frac{e^{\frac{2s}{\epsilon}}}{(1+e^{\frac{s}{\epsilon}}|y-w^{\alpha}|)^{m}}\frac{\der w}{(1+e^{\frac{s}{\epsilon}}|(v-w)^{\alpha}-w^{\alpha}|)^{m}}  \leq  \mathbbm{1}_{\{y\leq 0\}}\frac{C(\delta,\alpha,b)}{(1+|y-v^{\alpha}|)^{m-1}} .
\end{align*}
We now estimate $I_{1}$ in \eqref{first step y greater v alpha} in the region $w\in[0, \delta v]$. We have that
\begin{align*}
    \mathbbm{1}_{\{y\leq 0\}}  \int_{0}^{\delta v} & \frac{v^{\gamma}}{1+w^{b}}\frac{e^{\frac{2s}{\epsilon}}}{(1+e^{\frac{s}{\epsilon}}|y-w^{\alpha}|)^{m}}\frac{\der w}{(1+e^{\frac{s}{\epsilon}}|(v-w)^{\alpha}-w^{\alpha}|)^{m}}\\
    &  =    \mathbbm{1}_{\{y\in[-v^{\alpha},0]\}}  \int_{0}^{\delta v} \frac{v^{\gamma}}{1+w^{b}}\frac{e^{\frac{2s}{\epsilon}}}{(1+e^{\frac{s}{\epsilon}}|y-w^{\alpha}|)^{m}}\frac{\der w}{(1+e^{\frac{s}{\epsilon}}|(v-w)^{\alpha}-w^{\alpha}|)^{m}}\\
    & +  \mathbbm{1}_{\{y\leq -v^{\alpha}\}} \int_{0}^{\delta v}\frac{v^{\gamma}}{1+w^{b}}\frac{e^{\frac{2s}{\epsilon}}}{(1+e^{\frac{s}{\epsilon}}|y-w^{\alpha}|)^{m}}\frac{\der w}{(1+e^{\frac{s}{\epsilon}}|(v-w)^{\alpha}-w^{\alpha}|)^{m}}.
\end{align*}
For the region $y\leq -v^{\alpha}$ we have that 
\begin{align*}
    y-w^{\alpha}\leq y = \frac{y}{2}+\frac{y}{2}\leq \frac{y-v^{\alpha}}{2}\leq 0 \textup{ and } y-w^{\alpha}\leq -v^{\alpha}\leq 0.
\end{align*}
Thus, when $y\leq -v^{\alpha}$, we have that
\begin{align*}
   \mathbbm{1}_{\{y\leq -v^{\alpha}\}}   \int_{0}^{\delta v} & \frac{v^{\gamma}}{1+w^{b}}\frac{e^{\frac{2s}{\epsilon}}}{(1+e^{\frac{s}{\epsilon}}|y-w^{\alpha}|)^{m}}\frac{\der w}{(1+e^{\frac{s}{\epsilon}}|(v-w)^{\alpha}-w^{\alpha}|)^{m}}\\
      &\leq\mathbbm{1}_{\{y\leq -v^{\alpha}\}} \frac{C(m)v^{\gamma-\alpha}}{(1+e^{\frac{s}{\epsilon}}|y-v^{\alpha}|)^{m-1}}  \int_{0}^{\delta v} \frac{e^{\frac{s}{\epsilon}}\der w}{(1+e^{\frac{s}{\epsilon}}|(v-w)^{\alpha}-w^{\alpha}|)^{m}}
\end{align*}
and we can proceed as in \eqref{part two} in order to show that
\begin{align*}
      \mathbbm{1}_{\{y\leq -v^{\alpha}\}} \int_{0}^{\delta v}\frac{v^{\gamma}}{1+w^{b}}\frac{e^{\frac{2s}{\epsilon}}}{(1+e^{\frac{s}{\epsilon}}|y-w^{\alpha}|)^{m}}\frac{\der w}{(1+e^{\frac{s}{\epsilon}}|(v-w)^{\alpha}-w^{\alpha}|)^{m}}\leq   \frac{C(\delta,\alpha,\gamma,m)}{(1+|y-v^{\alpha}|)^{m-1}}.
\end{align*}
On the other hand, if $y\in[-v^{\alpha},0]$ and $w\in[0,\delta v]$, we have that $v^{\alpha}\leq |y-v^{\alpha}|\leq  2v^{\alpha}$ and that $|(v-w)^{\alpha}-w^{\alpha}|\geq C(\delta,\alpha) v^{\alpha}$. Thus
\begin{align*}
    \mathbbm{1}_{\{y\in[-v^{\alpha},0]\}}&   \int_{0}^{\delta v}\frac{v^{\gamma}}{1+w^{b}}\frac{e^{\frac{2s}{\epsilon}}}{(1+e^{\frac{s}{\epsilon}}|y-w^{\alpha}|)^{m}}\frac{\der w}{(1+e^{\frac{s}{\epsilon}}|(v-w)^{\alpha}-w^{\alpha}|)^{m}}\\
       &\leq  \mathbbm{1}_{\{y\in[-v^{\alpha},0]\}} v^{\gamma}\int_{0}^{\delta v}\frac{e^{\frac{2s}{\epsilon}}\der w}{(1+e^{\frac{s}{\epsilon}}|(v-w)^{\alpha}-w^{\alpha}|)^{m}}\\
       &\leq   \mathbbm{1}_{\{y\in[-v^{\alpha},0]\}}\frac{C(\delta,\alpha,m)v^{\gamma}}{1+|y-v^{\alpha}|^{\frac{m}{2}}}\int_{0}^{\delta v}\frac{e^{\frac{2s}{\epsilon}}\der w}{(1+e^{\frac{s}{\epsilon}}|(v-w)^{\alpha}-w^{\alpha}|)^{\frac{m}{2}}} \\
       &   \leq   \mathbbm{1}_{\{y\in[-v^{\alpha},0]\}}\frac{C(\delta,\alpha,m)v^{\gamma-\alpha}}{1+|y-v^{\alpha}|^{\frac{m}{2}}}\int_{0}^{\delta v}\frac{e^{\frac{s}{\epsilon}}\der w}{(1+e^{\frac{s}{\epsilon}}|(v-w)^{\alpha}-w^{\alpha}|)^{\frac{m}{2}-1}}  \\
       &\leq   \mathbbm{1}_{\{y\in[-v^{\alpha},0]\}}\frac{C(\delta,\alpha,\gamma,m)}{1+|y-v^{\alpha}|^{\frac{m}{2}}}\frac{v^{\gamma-\alpha}}{1+v^{\gamma+1-2\alpha}}\int_{0}^{\delta v}\frac{e^{\frac{s}{\epsilon}}\der w}{(1+e^{\frac{s}{\epsilon}}|(v-w)^{\alpha}-w^{\alpha}|)^{\frac{m}{2}+1-\frac{\gamma+1}{\alpha}}} \\
       &\leq  \mathbbm{1}_{\{y\in[-v^{\alpha},0]\}} \frac{C(\delta,\alpha,\gamma,m)}{1+|y-v^{\alpha}|^{\frac{m}{2}}},  
\end{align*}
where in the last line we used Lemma \ref{ref lemma dirac v-w v} for $\beta=\frac{m}{2}+1-\frac{\gamma+1}{\alpha}$.

Combining all the above estimates and performing similar computations as the ones in Proposition \ref{prop semigroup decay} and since $m>\frac{2(\gamma+1)}{\alpha}>2$, we obtain that
\begin{align}\label{e1}
  v^{\gamma}\int_{0}^{t}S_{\epsilon}(t-s)\bigg[ \int_{0}^{\frac{v}{2}} T_{1}(y,w,s)T_{1}(y,v-w,s)\der w  \mathbbm{1}_{\{y\leq 0\}}  \bigg]\der s\leq \frac{  C(\alpha,\gamma,b,m)A^{2}\epsilon\min\{1,e^{\frac{t}{\epsilon}}|y-v^{\alpha}|\}}{(1+v^{b})|y-v^{\alpha}|}.
\end{align}
\begin{itemize}
    \item \textbf{Case $y\in[0,v^{\alpha}]$:} We have in this case that $|y-v^{\alpha}|\leq v^{\alpha}$.
    \end{itemize}
    \begin{itemize}
    \item \textbf{Subcase $v\geq 1$:} We will prove that 
\end{itemize}
\begin{align}\label{main estimate y close to v alpha}
\mathbbm{1}_{\{0\leq y\leq v^{\alpha}\}}&\int_{0}^{\frac{v}{2}}\frac{e^{\frac{2s}{\epsilon}}v^{\gamma}}{(1+w^{b})(1+(v-w)^{b})}\psi(e^{\frac{s}{\epsilon}}(y-w^{\alpha}))\psi(e^{\frac{s}{\epsilon}}(y-(v-w)^{\alpha})\der w\nonumber\\
&\leq \frac{C(\alpha,\gamma,b,m) e^{\frac{s}{\epsilon}}}{(1+v^{b})(1+e^{\frac{s}{\epsilon}}|y-(\frac{v}{2})^{\alpha}|)^{l}},
  \end{align}
  for  $l=\frac{m}{2}$. Notice that $l>2$ since $m>\frac{2(\gamma+1)}{\alpha}>4.$ If \eqref{main estimate y close to v alpha} holds true, then we prove that
  \begin{align}\label{needed later for t2 t1}
      \int_{0}^{t}S_{\epsilon}(t  -s)\bigg[\frac{ e^{\frac{s}{\epsilon}}}{(1+e^{\frac{sd}{\epsilon}}|y-(\frac{v}{2})^{\alpha}|)^{l}}\bigg]\der s\leq\frac{ C(\alpha)\epsilon\min\{1,e^{\frac{t}{\epsilon}}|y-v^{\alpha}|\}}{|y-v^{\alpha}|}.
  \end{align}

  We now prove \eqref{needed later for t2 t1}. Assume first that $|y-v^{\alpha}|\geq e^{-\frac{t}{\epsilon}}$. It follows that
  \begin{align*}
      \int_{0}^{t}S_{\epsilon}(t-s)\bigg[\frac{e^{\frac{s}{\epsilon}}}{(1+e^{\frac{sl}{\epsilon}}|y-(\frac{v}{2})^{\alpha}|)^{l}}\bigg]\der s&=\int_{0}^{t}\frac{e^{\frac{t}{\epsilon}}\der s}{1+e^{\frac{sl}{\epsilon}}|(y-v^{\alpha})e^{\frac{t-s}{\epsilon}}+v^{\alpha}-(\frac{v}{2})^{\alpha}|^{l}}\\
      &=\int_{0}^{t}\frac{e^{\frac{t}{\epsilon}}\der s}{1+e^{\frac{tl}{\epsilon}}|(y-v^{\alpha})+e^{-\frac{t-s}{\epsilon}}\big(v^{\alpha}-(\frac{v}{2})^{\alpha}\big)|^{l}}\\
      &=\int_{0}^{t}\frac{e^{\frac{t}{\epsilon}}\der s}{1+e^{\frac{tl}{\epsilon}}|(y-v^{\alpha})+e^{-\frac{s}{\epsilon}}\big(v^{\alpha}-(\frac{v}{2})^{\alpha}\big)|^{l}},
  \end{align*}
where in the last line we made the change of variables $z=t-s$. Making now the change of variable $z=\frac{s}{\epsilon}$, we further obtain that
 \begin{align*}
      \int_{0}^{t}S_{\epsilon}(t-s)\bigg[\frac{ e^{\frac{s}{\epsilon}}}{(1+e^{\frac{sl}{\epsilon}}|y-(\frac{v}{2})^{\alpha}|)^{l}}\bigg]\der s&=\epsilon \int_{0}^{\frac{t}{\epsilon}}\frac{e^{\frac{t}{\epsilon}}\der s}{1+e^{\frac{tl}{\epsilon}}|(y-v^{\alpha})+e^{-s}\big(v^{\alpha}-(\frac{v}{2})^{\alpha}\big)|^{l}}.
  \end{align*}
We further make the change of variables $z=y-v^{\alpha}+e^{-s}\big(v^{\alpha}-(\frac{v}{2})^{\alpha}\big)$. Notice that
\begin{align*}
    \der z &= -e^{-s}(v^{\alpha}-(\frac{v}{2})^{\alpha})\der s;\\
    e^{-s} &=\frac{z+(v^{\alpha}-y)}{v^{\alpha}-(\frac{v}{2})^{\alpha}};\\
    \bigg|\frac{\der s}{\der z}\bigg|&=\frac{1}{z+(v^{\alpha}-y)}.
\end{align*}
It then follows that
\begin{align*}
      \int_{0}^{t}S_{\epsilon}(t-s)\bigg[\frac{ e^{\frac{s}{\epsilon}}}{(1+e^{\frac{sl}{\epsilon}}|y-(\frac{v}{2})^{\alpha}|)^{l}}\bigg]\der s&=\epsilon \int_{y-v^{\alpha}+e^{-\frac{t}{\epsilon}}(v^{\alpha}-(\frac{v}{2})^{\alpha})}^{y-(\frac{v}{2})^{\alpha}}\frac{e^{\frac{t}{\epsilon}}\der z}{1+e^{\frac{tl}{\epsilon}}|z|^{l}}\bigg|\frac{\der s}{\der z}\bigg|\\
      &=\epsilon \int_{y-v^{\alpha}+e^{-\frac{t}{\epsilon}}(v^{\alpha}-(\frac{v}{2})^{\alpha})}^{y-(\frac{v}{2})^{\alpha}}\frac{e^{\frac{t}{\epsilon}}\der z}{1+e^{\frac{tl}{\epsilon}}|z|^{l}}\frac{1}{z+(v^{\alpha}-y)}.
  \end{align*}
  It follows that
  \begin{align*}
      \int_{0}^{t}S_{\epsilon}(t-s)\bigg[\frac{ e^{\frac{s}{\epsilon}}}{(1+e^{\frac{sl}{\epsilon}}|y-(\frac{v}{2})^{\alpha}|)^{l}}\bigg]\der s
      &=\epsilon e^{\frac{t}{\epsilon}} \int_{e^{\frac{t}{\epsilon}}(y-v^{\alpha})+(v^{\alpha}-(\frac{v}{2})^{\alpha})}^{e^{\frac{t}{\epsilon}}(y-(\frac{v}{2})^{\alpha})}\frac{\der z}{1+|z|^{l}}\frac{1}{z+e^{\frac{t}{\epsilon}}(v^{\alpha}-y)}\\
      &=\epsilon e^{\frac{t}{\epsilon}} \int_{(v^{\alpha}-(\frac{v}{2})^{\alpha})}^{e^{\frac{t}{\epsilon}}(v^{\alpha}-(\frac{v}{2})^{\alpha})}\frac{\der z}{1+|z-e^{\frac{t}{\epsilon}}(v^{\alpha}-y)|^{l}}\frac{1}{z}.
  \end{align*}
We thus have that
\begin{align*}
 \epsilon e^{\frac{t}{\epsilon}} \int_{(v^{\alpha}-(\frac{v}{2})^{\alpha})}^{e^{\frac{t}{\epsilon}}(v^{\alpha}-(\frac{v}{2})^{\alpha})}\frac{\der z}{1+|z-e^{\frac{t}{\epsilon}}(v^{\alpha}-y)|^{l}}\frac{1}{z}&=\epsilon e^{\frac{t}{\epsilon}} \int_{(v^{\alpha}-(\frac{v}{2})^{\alpha}), |z|\leq \frac{e^{\frac{t}{\epsilon}}}{2}(v^{\alpha}-y)}^{e^{\frac{t}{\epsilon}}(v^{\alpha}-(\frac{v}{2})^{\alpha})}\frac{\der z}{1+|z-e^{\frac{t}{\epsilon}}(v^{\alpha}-y)|^{l}}\frac{1}{z}\\
    &+\epsilon e^{\frac{t}{\epsilon}} \int_{(v^{\alpha}-(\frac{v}{2})^{\alpha}), |z|\geq \frac{e^{\frac{t}{\epsilon}}}{2}(v^{\alpha}-y)}^{e^{\frac{t}{\epsilon}}(v^{\alpha}-(\frac{v}{2})^{\alpha})}\frac{\der z}{1+|z-e^{\frac{t}{\epsilon}}(v^{\alpha}-y)|^{l}}\frac{1}{z}.
\end{align*}
For the first term, since $v\geq 1$, we have that
\begin{align*}
 & \epsilon e^{\frac{t}{\epsilon}} \int_{(v^{\alpha}-(\frac{v}{2})^{\alpha}), |z|\leq \frac{e^{\frac{t}{\epsilon}}}{2}(v^{\alpha}-y)}^{e^{\frac{t}{\epsilon}}(v^{\alpha}-(\frac{v}{2})^{\alpha})}\frac{\der z}{1+|z-e^{\frac{t}{\epsilon}}(v^{\alpha}-y)|^{l}}\frac{1}{z}\leq \frac{C(\alpha)\epsilon e^{\frac{t}{\epsilon}} }{v^{\alpha}}\int_{(v^{\alpha}-(\frac{v}{2})^{\alpha}), |z|\leq \frac{e^{\frac{t}{\epsilon}}}{2}(v^{\alpha}-y)}^{e^{\frac{t}{\epsilon}}(v^{\alpha}-(\frac{v}{2})^{\alpha})}\frac{\der z}{1+|z-e^{\frac{t}{\epsilon}}(v^{\alpha}-y)|^{l}}\\
    &\leq \frac{C(\alpha)\epsilon e^{\frac{t}{\epsilon}} }{ e^{\frac{t}{\epsilon}}(v^{\alpha}-y)}\int_{(v^{\alpha}-(\frac{v}{2})^{\alpha}), |z|\leq \frac{e^{\frac{t}{\epsilon}}}{2}(v^{\alpha}-y)}^{e^{\frac{t}{\epsilon}}(v^{\alpha}-(\frac{v}{2})^{\alpha})}\frac{\der z}{1+|z-e^{\frac{t}{\epsilon}}(v^{\alpha}-y)|^{l-1}}\leq  \frac{C(\alpha) \epsilon }{ (v^{\alpha}-y)},
\end{align*}
since $l>2$, which is the correct estimate for the region $|y-v^{\alpha}|\geq e^{-\frac{t}{\epsilon}}$. 

For the second term, we have that
\begin{align*}
 \epsilon e^{\frac{t}{\epsilon}} \int_{(v^{\alpha}-(\frac{v}{2})^{\alpha}), |z|\geq \frac{e^{\frac{t}{\epsilon}}}{2}(v^{\alpha}-y)}^{e^{\frac{t}{\epsilon}}(v^{\alpha}-(\frac{v}{2})^{\alpha})}\frac{\der z}{1+|z-e^{\frac{t}{\epsilon}}(v^{\alpha}-y)|^{l}}\frac{1}{z} &\leq \frac{2\epsilon  }{v^{\alpha}-y}\int_{(v^{\alpha}-(\frac{v}{2})^{\alpha}), |z|\geq  \frac{e^{\frac{t}{\epsilon}}}{2}(v^{\alpha}-y)}^{e^{\frac{t}{\epsilon}}(v^{\alpha}-(\frac{v}{2})^{\alpha})}\frac{\der z}{1+|z-e^{\frac{t}{\epsilon}}(v^{\alpha}-y)|^{l}}\\
  &\leq \frac{C\epsilon  }{v^{\alpha}-y},
 \end{align*}
 which concludes our proof in the case when $v\geq 1, |y-v^{\alpha}|\geq e^{-\frac{t}{\epsilon}}$. 
 
We now focus on the case $v\geq 1, |y-v^{\alpha}|\leq  e^{-\frac{t}{\epsilon}}$. Since $v\geq 1$, we have that
\begin{align*}
 &    \int_{0}^{t}S_{\epsilon}(t-s)\bigg[\frac{ e^{\frac{s}{\epsilon}}}{(1+e^{\frac{sd}{\epsilon}}|y-(\frac{v}{2})^{\alpha}|)^{l}}\bigg]\der s\leq \frac{C(\alpha)\epsilon e^{\frac{t}{\epsilon}} }{v^{\alpha}}\int_{(v^{\alpha}-(\frac{v}{2})^{\alpha})}^{e^{\frac{t}{\epsilon}}(v^{\alpha}-(\frac{v}{2})^{\alpha})}\frac{\der z}{1+|z-e^{\frac{t}{\epsilon}}(v^{\alpha}-y)|^{l}}\\
    &\leq C(\alpha)\epsilon e^{\frac{t}{\epsilon}}\int_{(v^{\alpha}-(\frac{v}{2})^{\alpha})}^{e^{\frac{t}{\epsilon}}(v^{\alpha}-(\frac{v}{2})^{\alpha})}\frac{\der z}{1+|z-e^{\frac{t}{\epsilon}}(v^{\alpha}-y)|^{l}}\leq  C(\alpha)\epsilon e^{\frac{t}{\epsilon}}.
\end{align*}
We now focus on proving \eqref{main estimate y close to v alpha} when $v\geq 1$. We assume without loss of generality that $|y-w^{\alpha}|\geq |y-(\frac{v}{2})^{\alpha}|$ since the region $|y-(v-w)^{\alpha}|\geq |y-(\frac{v}{2})^{\alpha}|$ can be treated similarly. As before, we analyze separately the regions $w\in[\delta v, \frac{v}{2}]$ and $w\in(0,\delta v),$ for some $\delta<\frac{1}{2}$. First, we notice that
\begin{align*}
\int_{\delta v}^{\frac{v}{2}}&\frac{e^{\frac{2s}{\epsilon}}v^{\gamma}}{(1+w^{b})(1+(v-w)^{b})}\psi(e^{\frac{s}{\epsilon}}(y-w^{\alpha}))\psi(e^{\frac{s}{\epsilon}}(y-(v-w)^{\alpha})\der w\\
&\leq  \frac{C(\delta,b)e^{\frac{s}{\epsilon}}}{(1+e^{\frac{s}{\epsilon}}|y-(\frac{v}{2})^{\alpha}|)^{m}}\int_{\delta v}^{\frac{v}{2}}\frac{v^{\gamma}}{(1+v^{b})^{2}}\frac{e^{\frac{s}{\epsilon}}\der w}{(1+e^{\frac{s}{\epsilon}}|y-(v-w)^{\alpha}|)^{m}}\\
&\leq \frac{e^{\frac{s}{\epsilon}}}{(1+e^{\frac{s}{\epsilon}}|y-(\frac{v}{2})^{\alpha}|)^{m}}\frac{C(\delta,\alpha,b) v^{\gamma+1-\alpha}}{(1+v^{b})^{2}}\leq  \frac{C(\delta,\alpha,b) e^{\frac{s}{\epsilon}}}{(1+v^{b})(1+e^{\frac{s}{\epsilon}}|y-(\frac{v}{2})^{\alpha}|)^{m}},
\end{align*}
where we used Lemma \ref{ref lemma dirac v-w v}.

For the region where $w\in[0,\delta v]$, we will prove that
\begin{align*}
    &\int_{0}^{\delta v}\frac{v^{\gamma}}{1+w^{b}}\frac{e^{\frac{2s}{\epsilon}}}{(1+e^{\frac{s}{\epsilon}}|y-w^{\alpha}|)^{m}}\frac{\der w}{(1+e^{\frac{s}{\epsilon}}|(v-w)^{\alpha}-w^{\alpha}|)^{m}}\nonumber\\
&+\int_{0}^{\delta v}\frac{v^{\gamma}}{1+w^{b}}\frac{e^{\frac{2s}{\epsilon}}}{(1+e^{\frac{s}{\epsilon}}|y-(v-w)^{\alpha}|)^{m}}\frac{\der w}{(1+e^{\frac{s}{\epsilon}}|(v-w)^{\alpha}-w^{\alpha}|)^{m}}\\
&\leq \frac{C(\delta,\alpha,\gamma,m) e^{\frac{s}{\epsilon}}}{(1+e^{\frac{s}{\epsilon}}|y-(\frac{v}{2})^{\alpha}|)^{l}},
\end{align*}
for some $l=\frac{m}{2}>2$ as before and then \eqref{main estimate y close to v alpha} follows by \eqref{first step y greater v alpha}. We can assume without loss of generality that $|y-w^{\alpha}|\geq |y-(\frac{v}{2})^{\alpha}|$. 
  Moreover, since when $w\in[0,\delta v]$ it holds that $(v-w)^{\alpha}-w^{\alpha}\geq C(\delta,\alpha) v^{\alpha}$, we have that
    \begin{align*}
&    \int_{0}^{\delta v}\frac{v^{\gamma}}{1+w^{b}}\frac{e^{\frac{2s}{\epsilon}}}{(1+e^{\frac{s}{\epsilon}}|y-w^{\alpha}|)^{m}}\frac{\der w}{(1+e^{\frac{s}{\epsilon}}|(v-w)^{\alpha}-w^{\alpha}|)^{m}}\\
         &\leq \frac{C(\delta,\alpha,\gamma) e^{\frac{s}{\epsilon}}v^{\gamma}}{(1+e^{\frac{s}{\epsilon}}|y-(\frac{v}{2})^{\alpha}|)^{m}}  \int_{0}^{\delta v}\frac{e^{\frac{s}{\epsilon}}\der w}{(1+e^{\frac{s}{\epsilon}}|(v-w)^{\alpha}-w^{\alpha}|)^{m-\frac{\gamma+1}{\alpha}+1}}\frac{1}{1+v^{\gamma+1-\alpha}}\\
         &\leq \frac{C(\delta,\alpha,\gamma) e^{\frac{s}{\epsilon}}v^{\gamma+1-\alpha}}{(1+v^{\gamma+1-\alpha})(1+e^{\frac{s}{\epsilon}}|y-(\frac{v}{2})^{\alpha}|)^{m}}  \leq \frac{C(\delta,\alpha,\gamma) e^{\frac{s}{\epsilon}}}{(1+e^{\frac{s}{\epsilon}}|y-(\frac{v}{2})^{\alpha}|)^{m}},  
  \end{align*}
  where in the last line we used Lemma \ref{ref lemma dirac v-w v}.

  We now estimate $I_{2}$. 
   When $w\in[0,\delta v]$, then $(v-w)^{\alpha}-w^{\alpha}\geq C(\delta,\alpha) v^{\alpha}\geq C(\delta,\alpha)|y-(\frac{v}{2})^{\alpha}|$ and we  have that
  \begin{align*}
    &  \mathbbm{1}_{\{0\leq y\leq v^{\alpha}\}}\int_{0}^{\delta v}\frac{v^{\gamma}}{1+w^{b}}\frac{e^{\frac{2s}{\epsilon}}}{(1+e^{\frac{s}{\epsilon}}|y-(v-w)^{\alpha}|)^{m}}\frac{\der w}{(1+e^{\frac{s}{\epsilon}}|(v-w)^{\alpha}-w^{\alpha}|)^{m}}\\
    &  \leq  \mathbbm{1}_{\{0\leq y\leq v^{\alpha}\}}C(\delta,\alpha,\gamma,m) e^{\frac{s}{\epsilon}} v^{\gamma}\int_{0}^{\delta v}\frac{e^{\frac{s}{\epsilon}}\der w}{(1+e^{\frac{s}{\epsilon}}|(v-w)^{\alpha}-w^{\alpha}|)^{\frac{m}{2}-\frac{\gamma+1}{\alpha}+1}}\frac{1}{1+v^{\gamma+1-\alpha}}\frac{1}{1+e^{\frac{sm}{2\epsilon}}v^{\frac{\alpha m}{2}}}\\
    & \leq  \mathbbm{1}_{\{0\leq y\leq v^{\alpha}\}}\frac{C(\delta,\alpha,\gamma,m) e^{\frac{s}{\epsilon}} }{1+e^{\frac{sm}{2\epsilon}}v^{\frac{\alpha m}{2}}} \leq  \mathbbm{1}_{\{0\leq y\leq v^{\alpha}\}}\frac{C(\delta,\alpha,\gamma,m) e^{\frac{s}{\epsilon}} }{1+e^{\frac{sm}{2\epsilon}}|y-(\frac{v}{2})^{\alpha}|^{\frac{m}{2}}},
  \end{align*}
  which concludes our proof.

     \begin{itemize}
    \item \textbf{Subcase $v\leq 1$:} We will prove that 
\end{itemize}
\begin{align}\label{main estimate y close to v alpha part two}
&\mathbbm{1}_{\{0\leq y \leq v^{\alpha}\}}\int_{0}^{\frac{v}{2}}\frac{e^{\frac{2s}{\epsilon}}v^{\gamma}}{(1+w^{b})(1+(v-w)^{b})}\psi(e^{\frac{s}{\epsilon}}(y-w^{\alpha}))\psi(e^{\frac{s}{\epsilon}}(y-(v-w)^{\alpha})\der w\nonumber\\
&\leq \frac{C(\alpha,b) v^{\gamma}e^{\frac{s}{\epsilon}}}{(1+v^{b})(1+e^{\frac{s}{\epsilon}}|y-(\frac{v}{2})^{\alpha}|)^{m}}.
  \end{align}
  If \eqref{main estimate y close to v alpha part two} holds true, then as in the case $v\geq 1$, we have that
  \begin{align*}
      \int_{0}^{t}S_{\epsilon}(t-s)\bigg[\frac{ e^{\frac{s}{\epsilon}}v^{\gamma}}{(1+e^{\frac{sm}{\epsilon}}|y-(\frac{v}{2})^{\alpha}|)^{m}}\bigg]\der s= \epsilon e^{\frac{t}{\epsilon}}v^{\gamma} \int_{(v^{\alpha}-(\frac{v}{2})^{\alpha})}^{e^{\frac{t}{\epsilon}}(v^{\alpha}-(\frac{v}{2})^{\alpha})}\frac{\der z}{1+|z-e^{\frac{t}{\epsilon}}(v^{\alpha}-y)|^{m}}\frac{1}{z}.
      \end{align*}
      For the region $v\leq 1, |y-v^{\alpha}|\leq  e^{-\frac{t}{\epsilon}}$, we have that
\begin{align*}
 &   v^{\gamma} \int_{0}^{t}S_{\epsilon}(t-s)\bigg[\frac{ e^{\frac{s}{\epsilon}}}{(1+e^{\frac{sm}{\epsilon}}|y-(\frac{v}{2})^{\alpha}|)^{m}}\bigg]\der s\leq C(\alpha)\epsilon e^{\frac{t}{\epsilon}} v^{\gamma-\alpha}\int_{(v^{\alpha}-(\frac{v}{2})^{\alpha})}^{e^{\frac{t}{\epsilon}}(v^{\alpha}-(\frac{v}{2})^{\alpha})}\frac{\der z}{1+|z-e^{\frac{t}{\epsilon}}(v^{\alpha}-y)|^{m}}\\
    &\leq C(\alpha)\epsilon e^{\frac{t}{\epsilon}}v^{\gamma-\alpha} \leq  C(\alpha)\epsilon e^{\frac{t}{\epsilon}},
\end{align*}
since $\gamma>1>\alpha$ and $v\leq 1$.

For the region  $|y-v^{\alpha}|\geq e^{-\frac{t}{\epsilon}}$, it follows as before that
\begin{align*}
 & \epsilon e^{\frac{t}{\epsilon}} \int_{(v^{\alpha}-(\frac{v}{2})^{\alpha}), |z|\leq \frac{e^{\frac{t}{\epsilon}}}{2}(v^{\alpha}-y)}^{e^{\frac{t}{\epsilon}}(v^{\alpha}-(\frac{v}{2})^{\alpha})}\frac{v^{\gamma}\der z}{1+|z-e^{\frac{t}{\epsilon}}(v^{\alpha}-y)|^{m}}\frac{1}{z}\leq C(\alpha)\epsilon e^{\frac{t}{\epsilon}}\int_{(v^{\alpha}-(\frac{v}{2})^{\alpha}), |z|\leq \frac{e^{\frac{t}{\epsilon}}}{2}(v^{\alpha}-y)}^{e^{\frac{t}{\epsilon}}(v^{\alpha}-(\frac{v}{2})^{\alpha})}\frac{v^{\gamma-\alpha}\der z}{1+|z-e^{\frac{t}{\epsilon}}(v^{\alpha}-y)|^{m}}\\
    &\leq \frac{C(\alpha)\epsilon e^{\frac{t}{\epsilon}} v^{\gamma-\alpha} }{ e^{\frac{t}{\epsilon}}(v^{\alpha}-y)}\int_{(v^{\alpha}-(\frac{v}{2})^{\alpha}), |z|\leq \frac{e^{\frac{t}{\epsilon}}}{2}(v^{\alpha}-y)}^{e^{\frac{t}{\epsilon}}(v^{\alpha}-(\frac{v}{2})^{\alpha})}\frac{\der z}{1+|z-e^{\frac{t}{\epsilon}}(v^{\alpha}-y)|^{m-1}}\leq  \frac{C(\alpha) \epsilon }{ (v^{\alpha}-y)},
\end{align*}
since $v\leq 1$. In the same manner, we have that
\begin{align*}
 v^{\gamma}\epsilon e^{\frac{t}{\epsilon}} \int_{(v^{\alpha}-(\frac{v}{2})^{\alpha}), |z|\geq \frac{e^{\frac{t}{\epsilon}}}{2}(v^{\alpha}-y)}^{e^{\frac{t}{\epsilon}}(v^{\alpha}-(\frac{v}{2})^{\alpha})}\frac{\der z}{1+|z-e^{\frac{t}{\epsilon}}(v^{\alpha}-y)|^{m}}\frac{1}{z} &\leq \frac{2\epsilon v^{\gamma}  }{v^{\alpha}-y}\int_{(v^{\alpha}-(\frac{v}{2})^{\alpha}), |z|\geq  \frac{e^{\frac{t}{\epsilon}}}{2}(v^{\alpha}-y)}^{e^{\frac{t}{\epsilon}}(v^{\alpha}-(\frac{v}{2})^{\alpha})}\frac{\der z}{1+|z-e^{\frac{t}{\epsilon}}(v^{\alpha}-y)|^{m}}\\
  &\leq \frac{C\epsilon  }{v^{\alpha}-y},
 \end{align*}
 since $\gamma>1$ and $v\leq 1$.
We now focus on proving \eqref{main estimate y close to v alpha part two} when $v\leq 1$. We assume without loss of generality that $|y-w^{\alpha}|\geq |y-(\frac{v}{2})^{\alpha}|$. It holds that
\begin{align*}
&\int_{0}^{\frac{v}{2}}\frac{e^{\frac{2s}{\epsilon}}v^{\gamma}}{(1+w^{b})(1+(v-w)^{b})}\psi(e^{\frac{s}{\epsilon}}(y-w^{\alpha}))\psi(e^{\frac{s}{\epsilon}}(y-(v-w)^{\alpha})\der w\\
&\leq  \frac{e^{\frac{s}{\epsilon}}v^{\gamma}}{(1+e^{\frac{s}{\epsilon}}|y-(\frac{v}{2})^{\alpha}|)^{m}}\int_{0}^{\frac{v}{2}}\frac{C(b)}{1+v^{b}}\frac{e^{\frac{s}{\epsilon}}\der w}{(1+e^{\frac{s}{\epsilon}}|y-(v-w)^{\alpha}|)^{m}}\\
&\leq \frac{e^{\frac{s}{\epsilon}}v^{\gamma}}{(1+e^{\frac{s}{\epsilon}}|y-(\frac{v}{2})^{\alpha}|)^{m}}\frac{C(\alpha,b) v^{1-\alpha}}{1+v^{b}}\leq  \frac{C(\alpha,b) e^{\frac{s}{\epsilon}}v^{\gamma}}{(1+e^{\frac{s}{\epsilon}}|y-(\frac{v}{2})^{\alpha}|)^{m}}\frac{1}{1+v^{b}},
\end{align*}
where in the last line we used Lemma \ref{ref lemma dirac v-w v} and the fact that $\alpha<1$.

  \end{proof}
\subsection{Proof of Proposition \ref{Estimates for the combinations}}
\begin{proof}[Proof of Proposition \ref{Estimates for the combinations}]
\textbf{Proof of \eqref{estimate for t1 t2}.} Let $i\in\{2,3\}$. We want to bound from above the term 
\begin{align}\label{semigroup t1 t2}
    v^{\gamma}\int_{0}^{t}S_{\epsilon}(t-s)\int_{0}^{\frac{v}{2}} T_{i-1}(y,v-w,s)T_{i}(y,w,s)\der w\der s.
\end{align}
We start by estimating
\begin{align*}
    v^{\gamma}&\int_{0}^{\frac{v}{2}} T_{1}(y,v-w,s)T_{2}(y,w,s)\der w\\
    &  \leq v^{\gamma}\int_{0}^{\frac{v}{2}} \frac{4 A^{3}C_{0}^{m-1}}{(1+w^{b})(1+(v-w)^{b})}\frac{e^{\frac{s}{\epsilon}}\mathbbm{1}_{\{y\geq (v-w)^{\alpha}-C_{0}^{-1}\epsilon^{-\frac{1}{m-1}}e^{-\frac{s}{\epsilon}}\}}}{1+e^{\frac{ms}{\epsilon}}|y-(v-w)^{\alpha}|^{m}}\frac{\epsilon\mathbbm{1}_{\{y\leq w^{\alpha}-C_{0}^{-1}\epsilon^{-\frac{1}{m-1}}e^{-\frac{s}{\epsilon}}\}}}{|y-w^{\alpha}|}\der w=0.
\end{align*}
Thus, this term does not play any role. Similarly
\begin{align*}
    v^{\gamma}&\int_{0}^{\frac{v}{2}} T_{2}(y,v-w,s)T_{3}(y,w,s)\der w\\
    & \leq  v^{\gamma}\int_{0}^{\frac{v}{2}} \frac{2M_{2}A^{3}C_{0}^{m-1}}{(1+w^{b})(1+(v-w)^{b})}\frac{\epsilon\mathbbm{1}_{\{y\geq (\frac{v-w}{3})^{\alpha}\}}}{|y-(v-w)^{\alpha}|}\frac{e^{\frac{s}{\epsilon}}\mathbbm{1}_{\{y\leq (\frac{w}{3})^{\alpha}\}}}{1+e^{\frac{ms}{\epsilon}}|y-w^{\alpha}|^{m}}\der w=0.
\end{align*}

\textbf{Proof of \eqref{estimate for t2 t1}.} 
\begin{enumerate}
    \item[i)] We continue by estimating from above the term 
\end{enumerate}
\begin{align}\label{semigroup t2 t1}
    v^{\gamma}\int_{0}^{t}S_{\epsilon}(t-s)\int_{0}^{\frac{v}{2}} T_{2}(y,v-w,s)T_{1}(y,w,s)\der w\der s.
\end{align}
In the following, we bound first
  \begin{align}\label{obs on the region}
     v^{\gamma}\int_{0}^{t}S_{\epsilon}(t-s)\int_{\delta v}^{\frac{v}{2}} T_{2}(y,v-w,s)T_{1}(y,w,s)\der w\der s,
  \end{align}
  for some sufficiently small but fixed $\delta>0$. We focus then on bounding
 \begin{align*}
     v^{\gamma}\int_{0}^{t}S_{\epsilon}(t-s)\int_{0}^{\delta v}  T_{2}(y,v-w,s)T_{1}(y,w,s)\der w\der s.
  \end{align*}
We start by estimating
\begin{align*}
    v^{\gamma}\int_{\delta v}^{\frac{v}{2}} & T_{2}(y,v-w,s)T_{1}(y,w,s)\der w\\
       \leq &\frac{C(\delta, b)A^{3} C_{0}^{m-1}v^{\gamma}}{(1+v^{b})^{2}}\int_{\delta v}^{\frac{v}{2}}\frac{e^{\frac{s}{\epsilon}}\mathbbm{1}_{\{y\geq w^{\alpha}-C_{0}^{-1}\epsilon^{-\frac{1}{m-1}}e^{-\frac{s}{\epsilon}}\}}}{1+e^{\frac{ms}{\epsilon}}|y-w^{\alpha}|^{m}}\frac{\epsilon\mathbbm{1}_{\{y\leq (v-w)^{\alpha}-C_{0}^{-1}\epsilon^{-\frac{1}{m-1}}e^{-\frac{s}{\epsilon}}\}}}{|y-(v-w)^{\alpha}|}\der w\\
    \leq & \frac{C(\delta, b)A^{3}C_{0}^{m-1} v^{\gamma}}{(1+v^{b})^{2}}\int_{\delta v}^{\frac{v}{2}}\frac{e^{\frac{s}{\epsilon}}}{1+e^{\frac{ms}{\epsilon}}|y-w^{\alpha}|^{m}}\frac{\epsilon}{|y-(v-w)^{\alpha}|}\der w\\
   & \times \mathbbm{1}_{\{-C_{0}^{-1}\epsilon^{-\frac{1}{m-1}}e^{-\frac{s}{\epsilon}}\leq y\leq v^{\alpha}-C_{0}^{-1}\epsilon^{-\frac{1}{m-1}}e^{-\frac{s}{\epsilon}}\}}.
\end{align*}
\begin{rmk}\label{remark nice region}
Notice that the region $y\leq v^{\alpha}-C_{0}^{-1}\epsilon^{-\frac{1}{m-1}}e^{-\frac{s}{\epsilon}}$ implies that after the action of the semigroup we have that $e^{\frac{t-s}{\epsilon}}(y-v^{\alpha})+v^{\alpha}\leq v^{\alpha}-C_{0}^{-1}\epsilon^{-\frac{1}{m-1}}e^{-\frac{s}{\epsilon}}$ or in other words after the action of the semigroup we are in the region where $ e^{\frac{t}{\epsilon}}|y-v^{\alpha}|\geq C_{0}^{-1}\epsilon^{-\frac{1}{m-1}}$. 
\end{rmk}
 We can write $\mathbbm{1}_{\{-C_{0}^{-1}e^{-\frac{s}{\epsilon}}\epsilon^{-\frac{1}{m-1}}\leq y\leq v^{\alpha}-C_{0}^{-1}\epsilon^{-\frac{1}{m-1}}e^{-\frac{s}{\epsilon}}\}}$ as a sum of the form 
 \begin{align*}
     &\mathbbm{1}_{\{-C_{0}^{-1}e^{-\frac{s}{\epsilon}}\epsilon^{-\frac{1}{m-1}}\leq y\leq v^{\alpha}-C_{0}^{-1}\epsilon^{-\frac{1}{m-1}}e^{-\frac{s}{\epsilon}}\}}=\mathbbm{1}_{\{-C_{0}^{-1}e^{-\frac{s}{\epsilon}}\epsilon^{-\frac{1}{m-1}}\leq y\leq (\frac{v}{2})^{\alpha}-C_{0}^{-1}\epsilon^{-\frac{1}{m-1}}e^{-\frac{s}{\epsilon}}\}}\\
     &+\mathbbm{1}_{\{(\frac{v}{2})^{\alpha}+C_{0}^{-1}e^{-\frac{s}{\epsilon}}\epsilon^{-\frac{1}{m-1}}\leq y\leq v^{\alpha}-C_{0}^{-1}\epsilon^{-\frac{1}{m-1}}e^{-\frac{s}{\epsilon}}\}}+\mathbbm{1}_{\{(\frac{v}{2})^{\alpha}-C_{0}^{-1}e^{-\frac{s}{\epsilon}}\epsilon^{-\frac{1}{m-1}}\leq y\leq (\frac{v}{2})^{\alpha}+C_{0}^{-1}e^{-\frac{s}{\epsilon}}\epsilon^{-\frac{1}{m-1}}\}}.
 \end{align*}

 We analyze first the region where $\{-C_{0}^{-1}e^{-\frac{s}{\epsilon}}\epsilon^{-\frac{1}{m-1}}\leq y\leq (\frac{v}{2})^{\alpha}-C_{0}^{-1}\epsilon^{-\frac{1}{m-1}}e^{-\frac{s}{\epsilon}}\}$. In this region, since $|y-(v-w)^{\alpha}|\geq |y-(\frac{v}{2})^{\alpha}|$, then 
\begin{align*}
   \frac{e^{\frac{s}{\epsilon}}}{1+e^{\frac{ms}{\epsilon}}|y-w^{\alpha}|^{m}}\frac{\epsilon\mathbbm{1}_{\{y\leq (v-w)^{\alpha}-C_{0}^{-1}\epsilon^{-\frac{1}{m-1}}e^{-\frac{s}{\epsilon}}\}}}{|y-(v-w)^{\alpha}|}\leq  \frac{\epsilon}{|y-(\frac{v}{2})^{\alpha}|}  \frac{e^{\frac{s}{\epsilon}}}{1+e^{\frac{ms}{\epsilon}}|y-w^{\alpha}|^{m}}\mathbbm{1}_{\{y\leq (v-w)^{\alpha}-C_{0}^{-1}\epsilon^{-\frac{1}{m-1}}e^{-\frac{s}{\epsilon}}\}}.
\end{align*}Thus, it holds that
\begin{align*}
   I_{1}(y,v,s):= v^{\gamma} &\int_{\delta v}^{\frac{v}{2}} T_{2}(y,v-w,s)T_{1}(y,w,s)\der w\mathbbm{1}_{\{-C_{0}^{-1}e^{-\frac{s}{\epsilon}}\epsilon^{-\frac{1}{m-1}}\leq y\leq (\frac{v}{2})^{\alpha}-C_{0}^{-1}\epsilon^{-\frac{1}{m-1}}e^{-\frac{s}{\epsilon}}\}}\\
       \leq& \frac{C(\delta, b) A^{3}C_{0}^{m-1}v^{\gamma}}{(1+v^{b})^{2}}\int_{\delta v}^{\frac{v}{2}}\frac{\epsilon}{|y-(\frac{v}{2})^{\alpha}|}  \frac{e^{\frac{s}{\epsilon}}}{1+e^{\frac{ms}{\epsilon}}|y-w^{\alpha}|^{m}}\der w\\
  &  \times \mathbbm{1}_{\{-C_{0}^{-1}e^{-\frac{s}{\epsilon}}\epsilon^{-\frac{1}{m-1}}\leq y\leq (\frac{v}{2})^{\alpha}-C_{0}^{-1}\epsilon^{-\frac{1}{m-1}}e^{-\frac{s}{\epsilon}}\}}\\
 \leq &\frac{C(\delta, \alpha,b)\epsilon}{|y-(\frac{v}{2})^{\alpha}|} \frac{A^{3}C_{0}^{m-1} v^{\gamma+1-\alpha}}{(1+v^{b})^{2}}\mathbbm{1}_{\{-C_{0}^{-1}e^{-\frac{s}{\epsilon}}\epsilon^{-\frac{1}{m-1}}\leq y\leq (\frac{v}{2})^{\alpha}-C_{0}^{-1}\epsilon^{-\frac{1}{m-1}}e^{-\frac{s}{\epsilon}}\}},
\end{align*}
where in the last line we used Lemma \ref{ref lemma dirac v-w v}.
 
Then
\begin{align}
  v^{\gamma}\int_{0}^{t}S_{\epsilon}(t-s)\bigg[I_{1}(y,v,s)\bigg]\der s\leq &\frac{C(\delta, \alpha,b)A^{3}C_{0}^{m-1} v^{\gamma+1-\alpha}\epsilon }{(1+v^{b})^{2}}\int_{0}^{t}\frac{e^{\frac{t-s}{\epsilon}}}{|e^{\frac{t-s}{\epsilon}}(y-v^{\alpha})+v^{\alpha}-(\frac{v}{2})^{\alpha}|} \nonumber\\
      &\times\mathbbm{1}_{\{-C_{0}^{-1}e^{-\frac{s}{\epsilon}}\epsilon^{-\frac{1}{m-1}}\leq e^{\frac{t-s}{\epsilon}}(y-v^{\alpha})+v^{\alpha}\leq (\frac{v}{2})^{\alpha}-C_{0}^{-1}\epsilon^{-\frac{1}{m-1}}e^{-\frac{s}{\epsilon}}\}}\der s\nonumber\\
      =  &\frac{C(\delta, \alpha,b)A^{3}C_{0}^{m-1} v^{\gamma+1-\alpha}\epsilon }{(1+v^{b})^{2}}\int_{0}^{t}\frac{e^{\frac{t}{\epsilon}}}{|e^{\frac{t}{\epsilon}}(y-v^{\alpha})+e^{\frac{s}{\epsilon}}(v^{\alpha}-(\frac{v}{2})^{\alpha})|} \nonumber\\
      &\times \mathbbm{1}_{\{-e^{\frac{s}{\epsilon}}v^{\alpha}-C_{0}^{-1}\epsilon^{-\frac{1}{m-1}}\leq e^{\frac{t}{\epsilon}}(y-v^{\alpha})\leq e^{\frac{s}{\epsilon}}((\frac{v}{2})^{\alpha}-v^{\alpha})-C_{0}^{-1}\epsilon^{-\frac{1}{m-1}}\}}\der s\nonumber\\
         = & \frac{C(\delta, \alpha,b)A^{3}C_{0}^{m-1} v^{\gamma+1-\alpha}\epsilon }{(1+v^{b})^{2}}\int_{0}^{t}\frac{e^{\frac{t}{\epsilon}}}{|e^{\frac{t}{\epsilon}}(v^{\alpha}-y)-e^{\frac{s}{\epsilon}}(v^{\alpha}-(\frac{v}{2})^{\alpha})|} \nonumber\\
         &\times \mathbbm{1}_{\{e^{\frac{s}{\epsilon}}v^{\alpha}+C_{0}^{-1}\epsilon^{-\frac{1}{m-1}}\geq e^{\frac{t}{\epsilon}}(v^{\alpha}-y)\geq e^{\frac{s}{\epsilon}}(v^{\alpha}-(\frac{v}{2})^{\alpha})+C_{0}^{-1}\epsilon^{-\frac{1}{m-1}}\}}\der s\nonumber\\
         &=:C(\delta, \alpha,b)L(y,v,t).\label{region of integration}
\end{align}
We make the change of variables $z=e^{\frac{t}{\epsilon}}(v^{\alpha}-y)-e^{\frac{s}{\epsilon}}(v^{\alpha}-(\frac{v}{2})^{\alpha})$. We thus have that
\begin{align}\label{definition of L}
    L(y,v,t)        =  &\frac{A^{3}C_{0}^{m-1} v^{\gamma+1-\alpha}\epsilon^{2} }{(1+v^{b})^{2}}\int_{e^{\frac{t}{\epsilon}}((\frac{v}{2})^{\alpha}-y)}^{e^{\frac{t}{\epsilon}}(v^{\alpha}-y)-(v^{\alpha}-(\frac{v}{2})^{\alpha})}\frac{e^{\frac{t}{\epsilon}}}{e^{\frac{t}{\epsilon}}(v^{\alpha}-y)-z}\frac{1}{|z|}\nonumber\\
    &\times \mathbbm{1}_{\{e^{\frac{s}{\epsilon}}(\frac{v}{2})^{\alpha}+C_{0}^{-1}\epsilon^{-\frac{1}{m-1}}\geq z\geq C_{0}^{-1}\epsilon^{-\frac{1}{m-1}}\}}\der z.
\end{align}
On the one hand, it holds that
\begin{align}
 &  \frac{ v^{\gamma+1-\alpha}\epsilon^{2} }{(1+v^{b})^{2}}\int_{e^{\frac{t}{\epsilon}}((\frac{v}{2})^{\alpha}-y), z\leq \frac{e^{\frac{t}{\epsilon}}(v^{\alpha}-y)}{2}}^{e^{\frac{t}{\epsilon}}(v^{\alpha}-y)-(v^{\alpha}-(\frac{v}{2})^{\alpha})}\frac{e^{\frac{t}{\epsilon}}}{e^{\frac{t}{\epsilon}}(v^{\alpha}-y)-z}\frac{1}{|z|} \mathbbm{1}_{\{e^{\frac{s}{\epsilon}}(\frac{v}{2})^{\alpha}+C_{0}^{-1}\epsilon^{-\frac{1}{m-1}}\geq z\geq C_{0}^{-1}\epsilon^{-\frac{1}{m-1}}\}}\der z\nonumber\\
   &\leq \frac{2 v^{\gamma+1-\alpha}\epsilon^{2} }{(1+v^{b})^{2}(v^{\alpha}-y)}\int_{e^{\frac{t}{\epsilon}}((\frac{v}{2})^{\alpha}-y), z\leq \frac{e^{\frac{t}{\epsilon}}(v^{\alpha}-y)}{2}}^{e^{\frac{t}{\epsilon}}(v^{\alpha}-y)-(v^{\alpha}-(\frac{v}{2})^{\alpha})}\frac{1}{|z|} \mathbbm{1}_{\{e^{\frac{s}{\epsilon}}(\frac{v}{2})^{\alpha}+C_{0}^{-1}\epsilon^{-\frac{1}{m-1}}\geq z\geq C_{0}^{-1}\epsilon^{-\frac{1}{m-1}}\}}\der z\nonumber\\
      &\leq \frac{2 v^{\gamma+1-\alpha}\epsilon^{2} }{(1+v^{b})^{2}(v^{\alpha}-y)}\int_{1}^{e^{\frac{t}{\epsilon}}(v^{\alpha}-y)-(v^{\alpha}-(\frac{v}{2})^{\alpha})}\frac{1}{z} \der z\leq \frac{C(\alpha) v^{\gamma+1-\alpha}\epsilon^{2} }{(1+v^{b})^{2}(v^{\alpha}-y)}\bigg[\frac{t}{\epsilon}+|\ln{v^{\alpha}}|\bigg]\label{estimate L part one before},
      \end{align}
      where in \eqref{estimate L part one before} we used that we are in the region $e^{\frac{s}{\epsilon}}v^{\alpha}+C_{0}^{-1}\epsilon^{-\frac{1}{m-1}}\geq e^{\frac{t}{\epsilon}}(v^{\alpha}-y)\geq e^{\frac{s}{\epsilon}}(v^{\alpha}-(\frac{v}{2})^{\alpha})+C_{0}^{-1}\epsilon^{-\frac{1}{m-1}}$ from \eqref{region of integration} and that $C_{0}^{-1}\epsilon^{-\frac{1}{m-1}}e^{-\frac{t}{\epsilon}}\leq v^{\alpha}\big[1-3^{-\alpha}\big]$. We then have that
      \begin{align}
   \frac{ v^{\gamma+1-\alpha}\epsilon^{2} }{(1+v^{b})^{2}(v^{\alpha}-y)}\bigg[\frac{t}{\epsilon}+|\ln{v^{\alpha}}|\bigg]   & \leq \frac{ v^{\gamma+1-\alpha}t \epsilon }{(1+v^{b})^{2}(v^{\alpha}-y)}+\frac{ v^{\gamma+1} \epsilon^{2} }{(1+v^{b})^{2}(v^{\alpha}-y)}; \qquad v\geq 1,\label{estimate L part one}\\
     \frac{ v^{\gamma+1-\alpha}\epsilon^{2} }{(1+v^{b})^{2}(v^{\alpha}-y)}\bigg[\frac{t}{\epsilon}+|\ln{v^{\alpha}}|\bigg]   & \leq \frac{ v^{\gamma+1-\alpha}t \epsilon }{(1+v^{b})^{2}(v^{\alpha}-y)}+\frac{ v^{\gamma-\alpha} \epsilon^{2} }{(1+v^{b})^{2}(v^{\alpha}-y)}; \qquad v\leq 1.
\end{align}
 We then use that $b\geq \gamma+1$ in \eqref{estimate L part one}.

On the other hand, it holds that
\begin{align}\label{estimate L part two}
 &  \frac{ v^{\gamma+1-\alpha}\epsilon^{2} }{(1+v^{b})^{2}}\int_{e^{\frac{t}{\epsilon}}((\frac{v}{2})^{\alpha}-y), z\geq  \frac{e^{\frac{t}{\epsilon}}(v^{\alpha}-y)}{2}}^{e^{\frac{t}{\epsilon}}(v^{\alpha}-y)-(v^{\alpha}-(\frac{v}{2})^{\alpha})}\frac{e^{\frac{t}{\epsilon}}}{e^{\frac{t}{\epsilon}}(v^{\alpha}-y)-z}\frac{1}{|z|} \mathbbm{1}_{\{e^{\frac{s}{\epsilon}}(\frac{v}{2})^{\alpha}+C_{0}^{-1}\epsilon^{-\frac{1}{m-1}}\geq z\geq C_{0}^{-1}\epsilon^{-\frac{1}{m-1}}\}}\der z\nonumber\\
   &\leq \frac{2 v^{\gamma+1-\alpha}\epsilon^{2} }{(1+v^{b})^{2}(v^{\alpha}-y)}\int_{e^{\frac{t}{\epsilon}}((\frac{v}{2})^{\alpha}-y), z\geq \frac{e^{\frac{t}{\epsilon}}(v^{\alpha}-y)}{2}}^{e^{\frac{t}{\epsilon}}(v^{\alpha}-y)-(v^{\alpha}-(\frac{v}{2})^{\alpha})}\frac{1}{e^{\frac{t}{\epsilon}}(v^{\alpha}-y)-z}\mathbbm{1}_{\{e^{\frac{s}{\epsilon}}(\frac{v}{2})^{\alpha}+C_{0}^{-1}\epsilon^{-\frac{1}{m-1}}\geq z\geq C_{0}^{-1}\epsilon^{-\frac{1}{m-1}}\}}\der z\nonumber\\
   &\leq \frac{2 v^{\gamma+1-\alpha}\epsilon^{2} }{(1+v^{b})^{2}(v^{\alpha}-y)}\int_{1}^{e^{\frac{t}{\epsilon}}(v^{\alpha}-y)-(v^{\alpha}-(\frac{v}{2})^{\alpha})}\frac{1}{e^{\frac{t}{\epsilon}}(v^{\alpha}-y)-z}\der z\nonumber\\
    &\leq \frac{2 v^{\gamma+1-\alpha}\epsilon^{2} }{(1+v^{b})^{2}(v^{\alpha}-y)}\int_{v^{\alpha}-(\frac{v}{2})^{\alpha}}^{e^{\frac{t}{\epsilon}}(v^{\alpha}-y)-1}\frac{1}{z}\der z\leq \frac{C(\alpha) v^{\gamma+1-\alpha}\epsilon^{2} }{(1+v^{b})^{2}(v^{\alpha}-y)}\bigg[\frac{t}{\epsilon}+|\ln{v^{\alpha}}|\bigg]
\end{align}
which concludes our proof.

 We now estimate the region where $\{e^{\frac{s}{\epsilon}}|y-(\frac{v}{2})^{\alpha}|\leq C_{0}^{-1}\epsilon^{-\frac{1}{m-1}}\}$. In this region, it holds that
 \begin{align*}
     \frac{\epsilon^{\frac{m}{m-1}}}{2}\leq \frac{C_{0}^{-m}}{1+e^{\frac{sm}{\epsilon}}|y-(\frac{v}{2})^{\alpha}|^{m}}.
 \end{align*}
 Thus
\begin{align*}
  \frac{e^{\frac{s}{\epsilon}}}{1+e^{\frac{ms}{\epsilon}}|y-w^{\alpha}|^{m}}&\frac{\epsilon\mathbbm{1}_{\{y\leq (v-w)^{\alpha}-C_{0}^{-1}\epsilon^{-\frac{1}{m-1}}e^{-\frac{s}{\epsilon}}\}}}{|y-(v-w)^{\alpha}|}\\
   &\leq  \frac{2\epsilon \mathbbm{1}_{\{y\leq (v-w)^{\alpha}-C_{0}^{-1}\epsilon^{-\frac{1}{m-1}}e^{-\frac{s}{\epsilon}}\}} }{|y-(v-w)^{\alpha}|}  \frac{e^{\frac{s}{\epsilon}}}{1+e^{\frac{ms}{\epsilon}}|y-w^{\alpha}|^{m}}\frac{C_{0}^{-m}\epsilon^{-\frac{m}{m-1}}}{1+e^{\frac{sm}{\epsilon}}|y-(\frac{v}{2})^{\alpha}|^{m}}\\
   &\leq  \frac{2 C_{0}^{-(m-1)}e^{\frac{s}{\epsilon}}}{1+e^{\frac{ms}{\epsilon}}|y-w^{\alpha}|^{m}}\frac{e^{\frac{s}{\epsilon}}}{1+e^{\frac{sm}{\epsilon}}|y-(\frac{v}{2})^{\alpha}|^{m}}.
\end{align*}
Denote by
      \begin{align*}
      I_{2}(y,v,s):=  v^{\gamma}\int_{\delta v}^{\frac{v}{2}} T_{2}(y,v-w,s)T_{1}(y,w,s)\der w\mathbbm{1}_{\{e^{\frac{s}{\epsilon}}|y-(\frac{v}{2})^{\alpha}|\leq C_{0}^{-1}\epsilon^{-\frac{1}{m-1}}\}}.
      \end{align*}
Thus, in this region it holds that
\begin{align*}
  I_{2}(y,v,s)&\leq \frac{C(\delta,b)C_{0}^{m-1}A^{3} C_{0}^{-(m-1)} v^{\gamma}}{(1+v^{b})^{2}} \frac{e^{\frac{s}{\epsilon}}}{1+e^{\frac{sm}{\epsilon}}|y-(\frac{v}{2})^{\alpha}|^{m}}\int_{\delta v}^{\frac{v}{2}} \frac{e^{\frac{s}{\epsilon}}}{1+e^{\frac{ms}{\epsilon}}|y-w^{\alpha}|^{m}}\der w\\
& \leq \frac{C(\delta,\alpha,b)C_{0}^{m-1}A^{3} e^{\frac{s}{\epsilon}}}{1+e^{\frac{sm}{\epsilon}}|y-(\frac{v}{2})^{\alpha}|^{m}} \frac{ C_{0}^{-(m-1)} v^{\gamma+1-\alpha}}{(1+v^{b})^{2}},
\end{align*}
 where in the last line we used Lemma \ref{ref lemma dirac v-w v}.

Thus, in this region, we have that 
\begin{align*}
\int_{0}^{t} & S_{\epsilon}(t-s)\big[ I_{2}(y,v,s)\big]\der s\\
&\leq \frac{C(\delta,\alpha,b)C_{0}^{m-1}A^{3}  C_{0}^{-(m-1)}v^{\gamma+1-\alpha} }{(1+v^{b})^{2}}\int_{0}^{t}\frac{e^{\frac{t}{\epsilon}}\mathbbm{1}_{\{|e^{\frac{t}{\epsilon}}(y-v^{\alpha})+e^{\frac{s}{\epsilon}}(v^{\alpha}-(\frac{v}{2})^{\alpha})|\leq C_{0}^{-1}\epsilon^{-\frac{1}{m-1}} \}}}{1+|e^{\frac{t}{\epsilon}}(y-v^{\alpha})+e^{\frac{s}{\epsilon}}(v^{\alpha}-(\frac{v}{2})^{\alpha})|^{m}} \der s.
      \end{align*}
We then use \eqref{needed later for t2 t1}, $m>2$, in order to deduce that
\begin{align*}
   \int_{0}^{t}S_{\epsilon}(t-s)\big[ I_{2}(y,v,s)\big]\der s\leq   \frac{ C(\delta,\alpha,b)C_{0}^{m-1}A^{3}  C_{0}^{-(m-1)} \epsilon}{(1+v^{b})|y-v^{\alpha}|}\min\{1,e^{\frac{t}{\epsilon}}|y-v^{\alpha}|\}.
\end{align*}
 By Remark \ref{remark nice region} we further obtain that 
 \begin{align}\label{complicated choice c0}
   \int_{0}^{t}S_{\epsilon}(t-s)\big[ I_{2}(y,v,s)\big]\der s\leq   \frac{ C(\delta,\alpha,b)C_{0}^{m-1}A^{3}  C_{0}^{-(m-1)} \epsilon}{(1+v^{b})|y-v^{\alpha}|}\mathbbm{1}_{\{y\leq v^{\alpha}-C_{0}^{-1}\epsilon^{-\frac{1}{m-1}}e^{-\frac{t}{\epsilon}}\}}.
\end{align}
\begin{rmk}
Notice that since $C_{0}^{m-1}$ is of order $\mathcal{O}(A)$ by \eqref{defc0}, this implies in particular that
\begin{align*}
 \int_{0}^{t}S_{\epsilon}(t-s)\big[ I_{2}(y,v,s)\big]\der s &\leq   \frac{ C(\delta,\alpha,b)A^{3}  \epsilon}{(1+v^{b})|y-v^{\alpha}|}\mathbbm{1}_{\{y\leq v^{\alpha}-C_{0}^{-1}\epsilon^{-\frac{1}{m-1}}e^{-\frac{t}{\epsilon}}\}}\\
 &=  \frac{ C(\delta,\alpha,\gamma,b,m)C_{0}^{m-1}A^{2}  \epsilon}{(1+v^{b})|y-v^{\alpha}|}\mathbbm{1}_{\{y\leq v^{\alpha}-C_{0}^{-1}\epsilon^{-\frac{1}{m-1}}e^{-\frac{t}{\epsilon}}\}}.
\end{align*}
The reason we keep the form \eqref{complicated choice c0} is because this estimate is part of the more general estimate in \eqref{estimate for t2 t1}.
\end{rmk}

We now estimate the region where $\{(\frac{v}{2})^{\alpha}+C_{0}^{-1}e^{-\frac{s}{\epsilon}}\epsilon^{-\frac{1}{m-1}}\leq y\leq v^{\alpha}+C_{0}^{-1}e^{-\frac{s}{\epsilon}}\epsilon^{-\frac{1}{m-1}}\}$. In this region, it holds that $|y-w^{\alpha}|\geq |y-(\frac{v}{2})^{\alpha}|$. Then 
\begin{align*}
  \frac{e^{\frac{s}{\epsilon}}}{1+e^{\frac{ms}{\epsilon}}|y-w^{\alpha}|^{m}}&\frac{\epsilon\mathbbm{1}_{\{y\leq (v-w)^{\alpha}-C_{0}^{-1}\epsilon^{-\frac{1}{m-1}}e^{-\frac{s}{\epsilon}}\}}}{|y-(v-w)^{\alpha}|}\\
   &\leq  \frac{\epsilon}{|y-(v-w)^{\alpha}|}  \frac{e^{\frac{s}{\epsilon}}}{1+e^{\frac{ms}{\epsilon}}|y-(\frac{v}{2})^{\alpha}|^{m}}\mathbbm{1}_{\{y\leq (v-w)^{\alpha}-C_{0}^{-1}\epsilon^{-\frac{1}{m-1}}e^{-\frac{s}{\epsilon}}\}}.
\end{align*}
We have that
\begin{align}\label{log growth}
    \int_{0}^{\frac{v}{2}} \frac{\der w}{|y-(v-w)^{\alpha}|} \mathbbm{1}_{\{y\leq (v-w)^{\alpha}-C_{0}^{-1}\epsilon^{-\frac{1}{m-1}}e^{-\frac{s}{\epsilon}}\}}&\leq  C(\alpha)v^{1-\alpha} \int_{y-v^{\alpha}}^{y-(\frac{v}{2})^{\alpha}} \frac{\der z}{|z|} \mathbbm{1}_{\{z\leq -C_{0}^{-1}\epsilon^{-\frac{1}{m-1}}e^{-\frac{s}{\epsilon}}\}}\nonumber\\
    &\leq C(\alpha)v^{1-\alpha} \int_{(\frac{v}{2})^{\alpha}-y}^{v^{\alpha}-y}\frac{\der z}{z} \mathbbm{1}_{\{z\geq C_{0}^{-1}\epsilon^{-\frac{1}{m-1}}e^{-\frac{s}{\epsilon}}\}}\nonumber\\
    &\leq C(\alpha)v^{1-\alpha}\bigg[\frac{t}{\epsilon}+|\ln{v^{\alpha}}|\bigg]
\end{align}
since $0\leq y\leq v^{\alpha}$. We can thus conclude that in this region, we have that 
\begin{align*}
      v^{\gamma}\int_{0}^{t}& S_{\epsilon}(t-s)\bigg[\int_{\delta v}^{\frac{v}{2}} T_{2}(y,v-w,s)T_{1}(y,w,s)\der w\mathbbm{1}_{\{(\frac{v}{2})^{\alpha}+C_{0}^{-1}e^{-\frac{s}{\epsilon}}\epsilon^{-\frac{1}{m-1}}\leq y\leq v^{\alpha}+C_{0}^{-1}e^{-\frac{s}{\epsilon}}\epsilon^{-\frac{1}{m-1}}\}}\bigg]\der s\\
      \leq &\frac{C(\delta,\alpha,b)A^{3}C_{0}^{m-1} [v^{\gamma+1-\alpha}t+v^{\gamma+1}\epsilon\mathbbm{1}_{\{v\geq 1\}} +v^{\gamma-\alpha}\epsilon\mathbbm{1}_{\{v\leq 1\}} ]}{(1+v^{b})^{2}}\\
      &\times \int_{0}^{t}\frac{e^{\frac{t}{\epsilon}}}{1+|e^{\frac{t}{\epsilon}}(y-v^{\alpha})+e^{\frac{s}{\epsilon}}(v^{\alpha}-(\frac{v}{2})^{\alpha})|^{m}} \mathbbm{1}_{\{e^{\frac{t}{\epsilon}}(y-v^{\alpha})+e^{\frac{s}{\epsilon}}(v^{\alpha}-(\frac{v}{2})^{\alpha})\geq C_{0}^{-1}\epsilon^{-\frac{1}{m-1}} \}}\der s.
      \end{align*}
      We can now conclude as before using \eqref{needed later for t2 t1}.
      This concludes this case.
      We now bound \begin{align*}
     v^{\gamma}\int_{0}^{t}S_{\epsilon}(t-s)\int_{0}^{\delta v}  T_{2}(y,v-w,s)T_{1}(y,w,s)\der w\der s.
  \end{align*}
  Notice that we only have to make sure that we have sufficiently fast decay in the $v$ variable. We then follow the computations for the region $w\in[\delta v, \frac{v}{2}]$ in order to conclude. More precisely, it suffices to show that
  \begin{align*}
      I_{3}(y,v,s):=&v^{\gamma}\int_{0}^{\delta v}  T_{2}(y,v-w,s)T_{1}(y,w,s)\der w \\
      &\leq \frac{C(\delta,\alpha,\gamma,b)A^{3}C_{0}^{m-1} v^{\gamma}}{(1+v^{b})(1+v^{\gamma+1})}\int_{0}^{\delta v}\frac{e^{\frac{s}{\epsilon}}\mathbbm{1}_{\{y\geq w^{\alpha}-C_{0}^{-1}\epsilon^{-\frac{1}{m-1}}e^{-\frac{s}{\epsilon}}\}}}{1+e^{\frac{(m-\frac{\gamma+1}{\alpha})s}{\epsilon}}|y-w^{\alpha}|^{m-\frac{\gamma+1}{\alpha}}}\frac{\epsilon\mathbbm{1}_{\{y\leq (v-w)^{\alpha}-C_{0}^{-1}\epsilon^{-\frac{1}{m-1}}e^{-\frac{s}{\epsilon}}\}}}{|y-(v-w)^{\alpha}|}\der w,
  \end{align*}
  for a sufficiently small $\delta<\frac{1}{2}$.
  We have that
  \begin{align*}
         I_{3}(y,v,s)&\leq \frac{C(b)A^{3}C_{0}^{m-1} v^{\gamma}}{1+v^{b}}\int_{0}^{\delta v}\frac{1}{1+w^{b}}\frac{e^{\frac{s}{\epsilon}}\mathbbm{1}_{\{y\geq w^{\alpha}-C_{0}^{-1}\epsilon^{-\frac{1}{m-1}}e^{-\frac{s}{\epsilon}}\}}}{1+e^{\frac{ms}{\epsilon}}|y-w^{\alpha}|^{m}}\frac{\epsilon\mathbbm{1}_{\{(\frac{v-w}{3})^{\alpha}\leq y\leq (v-w)^{\alpha}-C_{0}^{-1}\epsilon^{-\frac{1}{m-1}}e^{-\frac{s}{\epsilon}}\}}}{|y-(v-w)^{\alpha}|}\der w\\
        &\leq \frac{C(b)A^{3}C_{0}^{m-1} v^{\gamma}}{1+v^{b}}\int_{0}^{\delta v}\frac{e^{\frac{s}{\epsilon}}\mathbbm{1}_{\{y\geq w^{\alpha}-C_{0}^{-1}\epsilon^{-\frac{1}{m-1}}e^{-\frac{s}{\epsilon}}\}}}{1+e^{\frac{ms}{\epsilon}}|y-w^{\alpha}|^{m}}\frac{\epsilon\mathbbm{1}_{\{(\frac{v-w}{3})^{\alpha}\leq y\leq (v-w)^{\alpha}-C_{0}^{-1}\epsilon^{-\frac{1}{m-1}}e^{-\frac{s}{\epsilon}}\}}}{|y-(v-w)^{\alpha}|}\der  w.
  \end{align*}
We have that $(\frac{v-w}{3})^{\alpha}\leq y $, which implies that for $\delta$ sufficiently small it holds that $y-w^{\alpha}\geq (\frac{v-w}{3})^{\alpha}-w^{\alpha}\geq v^{\alpha}\bigg[\big(\frac{1-\delta}{3}\big)^{\alpha}-\delta^{\alpha}\bigg]\geq C(\delta,\alpha) v^{\alpha}$. We thus have that
\begin{align*}
      I_{3}(y,v,s) \leq &\frac{C(\delta,\alpha,\gamma,b)A^{3}C_{0}^{m-1} v^{\gamma}}{(1+v^{b})(1+v^{\gamma+1})}\int_{0}^{\delta v}\frac{e^{\frac{s}{\epsilon}}\mathbbm{1}_{\{y\geq w^{\alpha}-C_{0}^{-1}\epsilon^{-\frac{1}{m-1}}e^{-\frac{s}{\epsilon}}\}}}{1+e^{\frac{(m-\frac{\gamma+1}{\alpha})s}{\epsilon}}|y-w^{\alpha}|^{m-\frac{\gamma+1}{\alpha}}}\\
      &\times \frac{\epsilon\mathbbm{1}_{\{(\frac{v-w}{3})^{\alpha}\leq y\leq (v-w)^{\alpha}-C_{0}^{-1}\epsilon^{-\frac{1}{m-1}}e^{-\frac{s}{\epsilon}}\}}}{|y-(v-w)^{\alpha}|}\der  w
\end{align*}
and we can proceed as in the region $w\in [\delta v, \frac{v}{2}]$ since $m>\frac{\gamma+1}{\alpha}+2$.

\begin{enumerate}
    \item[ii)] We continue by bounding from above the term 
\end{enumerate}
\begin{align}
    v^{\gamma}\int_{0}^{t}S_{\epsilon}(t-s)\int_{0}^{\frac{v}{2}} T_{3}(y,v-w,s)T_{2}(y,w,s)\der w\der s.
\end{align}
We will show that
  \begin{align*}
    I_{4}(y,v,s):= & v^{\gamma}\int_{0}^{\frac{v}{2}}  T_{3}(y,v-w,s)T_{2}(y,w,s)\der w \\
      &\leq \frac{C(\alpha,\gamma,b)M_{2}A^{3}C_{0}^{m-1} v^{\gamma}}{(1+v^{b})(1+v^{\gamma+1})}\int_{0}^{\frac{v}{2}}\frac{e^{\frac{s}{\epsilon}}\mathbbm{1}_{\{y\leq (\frac{v-w}{3})^{\alpha}\}}}{1+e^{\frac{(m-\frac{\gamma+1}{\alpha})s}{\epsilon}}|y-(v-w)^{\alpha}|^{m-\frac{\gamma+1}{\alpha}}}\\
      &\times \frac{\epsilon\mathbbm{1}_{\{(\frac{w}{3})^{\alpha}\leq y\leq w^{\alpha}-C_{0}^{-1}\epsilon^{-\frac{1}{m-1}}e^{-\frac{s}{\epsilon}}\}}}{|y-w^{\alpha}|}\der w.
  \end{align*}
  Since $0\leq y\leq (\frac{v-w}{3})^{\alpha}$, we have that $|y-(v-w)^{\alpha}|\geq C(\alpha)(v-w)^{\alpha}\geq C(\alpha)v^{\alpha}.$ Thus
   \begin{align*}
     I_{4}(y,v,s)
      \leq &\frac{2M_{2}A^{3}C_{0}^{m-1} v^{\gamma}}{1+v^{b}}\int_{0}^{\frac{v}{2}}\frac{e^{\frac{s}{\epsilon}}\mathbbm{1}_{\{y\leq (\frac{v-w}{3})^{\alpha}\}}}{1+e^{\frac{m s}{\epsilon}}|y-(v-w)^{\alpha}|^{m}}\frac{\epsilon\mathbbm{1}_{\{(\frac{w}{3})^{\alpha}\leq y\leq w^{\alpha}-C_{0}^{-1}\epsilon^{-\frac{1}{m-1}}e^{-\frac{s}{\epsilon}}\}}}{|y-w^{\alpha}|}\der w\\ \leq &\frac{C(\alpha,\gamma,b)M_{2}A^{3}C_{0}^{m-1} v^{\gamma}}{(1+v^{b})(1+v^{\gamma+1})}\int_{0}^{\frac{v}{2}}\frac{e^{\frac{s}{\epsilon}}\mathbbm{1}_{\{y\leq (\frac{v-w}{3})^{\alpha}\}}}{1+e^{\frac{(m-\frac{\gamma+1}{\alpha})s}{\epsilon}}|y-(v-w)^{\alpha}|^{m-\frac{\gamma+1}{\alpha}}}\\
      &\times \frac{\epsilon\mathbbm{1}_{\{(\frac{w}{3})^{\alpha}\leq y\leq w^{\alpha}-C_{0}^{-1}\epsilon^{-\frac{1}{m-1}}e^{-\frac{s}{\epsilon}}\}}}{|y-w^{\alpha}|}\der w.
  \end{align*}
  We further have that
     \begin{align*}
     I_{4}(y,v,s)
      \leq & \frac{C(\alpha,\gamma,b)M_{2}A^{3}C_{0}^{m-1} v^{\gamma}}{(1+v^{b})(1+v^{\gamma+1})}\frac{e^{\frac{s}{\epsilon}}}{1+e^{\frac{(m-\frac{\gamma+1}{\alpha})s}{\epsilon}}|y-(\frac{v}{2})^{\alpha}|^{m-\frac{\gamma+1}{\alpha}}}\\
      &\times \int_{0}^{\frac{v}{2}}\frac{\epsilon\mathbbm{1}_{\{ y\leq w^{\alpha}-C_{0}^{-1}\epsilon^{-\frac{1}{m-1}}e^{-\frac{s}{\epsilon}}\}}}{|y-w^{\alpha}|}\der w\mathbbm{1}_{\{0\leq y\leq (\frac{v}{2})^{\alpha}-C_{0}^{-1}\epsilon^{-\frac{1}{m-1}}e^{-\frac{s}{\epsilon}}\}}.
\end{align*}
We bound $\int_{0}^{\frac{v}{2}}\frac{\mathbbm{1}_{\{ y\leq w^{\alpha}-C_{0}^{-1}\epsilon^{-\frac{1}{m-1}}e^{-\frac{s}{\epsilon}}\}}}{|y-w^{\alpha}|}\der w$ as in  \eqref{log growth}. Thus \begin{align}
\int_{0}^{\frac{v}{2}}\frac{\mathbbm{1}_{\{y\leq w^{\alpha}-C_{0}^{-1}\epsilon^{-\frac{1}{m-1}}e^{-\frac{s}{\epsilon}}\}}}{|y-w^{\alpha}|}\der w\leq C(\alpha)v^{1-\alpha}\bigg[\frac{t}{\epsilon}+|\ln{v^{\alpha}}|\bigg].
\end{align}
It follows that 
  \begin{align*}
     I_{4}(y,v,s)
      \leq & \frac{C(\alpha,\gamma,b)M_{2}A^{3}C_{0}^{m-1}\max\{v^{\gamma+1},v^{\gamma-\alpha}\}(\epsilon+t)}{(1+v^{b})(1+v^{\gamma+1})}\\
      &\times\frac{e^{\frac{s}{\epsilon}}}{1+e^{\frac{(m-\frac{\gamma+1}{\alpha})s}{\epsilon}}|y-(\frac{v}{2})^{\alpha}|^{m-\frac{\gamma+1}{\alpha}}}\mathbbm{1}_{\{0\leq y\leq (\frac{v}{2})^{\alpha}-C_{0}^{-1}\epsilon^{-\frac{1}{m-1}}e^{-\frac{s}{\epsilon}}\}}.
\end{align*}
We then use \eqref{needed later for t2 t1} since $m>\frac{\gamma+1}{\alpha}+2$ in order to deduce that
\begin{align*}
    \int_{0}^{t}S_{\epsilon}(t-s)\big[I_{4}(y,v,s)\big]\der s \leq  \frac{ C(\alpha,\gamma,b)M_{2}A^{3}C_{0}^{m-1}(\epsilon+t) \epsilon}{(1+v^{b})|y-v^{\alpha}|}\min\{1,e^{\frac{t}{\epsilon}}|y-v^{\alpha}|\}.
\end{align*}
\textbf{Proof of \eqref{estimate for t2 t2}.} We now bound from above the term 

\begin{align}
    v^{\gamma}\int_{0}^{t}S_{\epsilon}(t-s)\int_{0}^{\frac{v}{2}} T_{2}(y,v-w,s)T_{2}(y,w,s)\der w\der s.
\end{align}
We notice that $(\frac{v-w}{3})^{\alpha}\leq y\leq  w^{\alpha}-C_{0}^{-1}\epsilon^{-\frac{1}{m-1}}e^{-\frac{s}{\epsilon}}\leq w^{\alpha}$, which implies that $w\geq \frac{v}{4}$. In other words, it suffices to bound from above the following term
\begin{align}\label{semigroup t2 t2}
    v^{\gamma}\int_{0}^{t}S_{\epsilon}(t-s)\int_{\frac{v}{4}}^{\frac{v}{2}} T_{2}(y,v-w,s)T_{2}(y,w,s)\der w\der s.
\end{align}
We start by estimating
\begin{align*}
   I_{5}(y,v,s):=& v^{\gamma}\int_{\frac{v}{4}}^{\frac{v}{2}} T_{2}(y,v-w,s)T_{2}(y,w,s)\der w\\
    &   \leq 4A^{4}C_{0}^{2(m-1)}v^{\gamma}\int_{\frac{v}{4}}^{\frac{v}{2}}\frac{1}{(1+w^{b})(1+(v-w)^{b})}\frac{\epsilon\mathbbm{1}_{\{y\leq w^{\alpha}-C_{0}^{-1}\epsilon^{-\frac{1}{m-1}}e^{-\frac{s}{\epsilon}}\}}}{|y-w^{\alpha}|}\\
    &\times\frac{\epsilon\mathbbm{1}_{\{y\leq (v-w)^{\alpha}-C_{0}^{-1}\epsilon^{-\frac{1}{m-1}}e^{-\frac{s}{\epsilon}}\}}}{|y-(v-w)^{\alpha}|}\der w\\
& \leq  \frac{C(b)A^{4}C_{0}^{2(m-1)}v^{\gamma}}{(1+v^{b})^{2}}\int_{\frac{v}{4}}^{\frac{v}{2}}\frac{\epsilon\mathbbm{1}_{\{y\leq w^{\alpha}-C_{0}^{-1}\epsilon^{-\frac{1}{m-1}}e^{-\frac{s}{\epsilon}}\}}}{|y-w^{\alpha}|}\frac{\epsilon\mathbbm{1}_{\{y\leq (v-w)^{\alpha}-C_{0}^{-1}\epsilon^{-\frac{1}{m-1}}e^{-\frac{s}{\epsilon}}\}}}{|y-(v-w)^{\alpha}|}\der w\\
& \leq  \frac{C(b)A^{4}C_{0}^{2(m-1)}v^{\gamma}}{(1+v^{b})^{2}}\frac{1}{|y-(\frac{v}{2})^{\alpha}|}\int_{\frac{v}{4}}^{\frac{v}{2}}\frac{\epsilon^{2}\mathbbm{1}_{\{y\leq w^{\alpha}-C_{0}^{-1}\epsilon^{-\frac{1}{m-1}}e^{-\frac{s}{\epsilon}}\}}}{|y-w^{\alpha}|}\der w\mathbbm{1}_{\{y\leq (\frac{v}{2})^{\alpha}-C_{0}^{-1}\epsilon^{-\frac{1}{m-1}}e^{-\frac{s}{\epsilon}}\}}.
\end{align*}
Following the computations in \eqref{log growth}, it holds that
\begin{align}\label{y-w integral}
\int_{0}^{\frac{v}{2}}\frac{\mathbbm{1}_{\{y\leq w^{\alpha}-C_{0}^{-1}\epsilon^{-\frac{1}{m-1}}e^{-\frac{s}{\epsilon}}\}}}{|y-w^{\alpha}|}\der w\leq C(\alpha)v^{1-\alpha}\bigg[\frac{t}{\epsilon}+|\ln{v^{\alpha}}|\bigg].
\end{align}
Thus, we have that
\begin{align*}
       I_{5}(y,v,s)&\leq  \frac{C(\alpha,b)A^{4}C_{0}^{2(m-1)}\epsilon  [v^{\gamma+1-\alpha}t+v^{\gamma+\frac{1}{2}}\epsilon\mathbbm{1}_{\{v\geq 1\}} +v^{\gamma-\frac{\alpha}{2}}\epsilon\mathbbm{1}_{\{v\leq 1\}} ]}{(1+v^{b})^{2}}\frac{1}{|y-(\frac{v}{2})^{\alpha}|}\mathbbm{1}_{\{y\leq (\frac{v}{2})^{\alpha}-\epsilon^{-\frac{1}{m-1}}e^{-\frac{s}{\epsilon}}\}}\\
       &\leq  \frac{C(\alpha,b) \max\{v^{\gamma+\frac{1}{2}},v^{\gamma-\frac{\alpha}{2}}\}(t+\epsilon)}{(1+v^{b})^{2}}\frac{A^{4}C_{0}^{m-1}\epsilon}{|y-(\frac{v}{2})^{\alpha}|}\mathbbm{1}_{\{y\leq (\frac{v}{2})^{\alpha}-\epsilon^{-\frac{1}{m-1}}e^{-\frac{s}{\epsilon}}\}}
\end{align*}
and then we can bound  $ v^{\gamma}\int_{0}^{t}S_{\epsilon}(t-s)\big[ I_{5}(y,v,s)\big]\der s$ in the same manner as $ L(y,v,t),$ where $L(y,v,t)$ is as in \eqref{definition of L}. We then conclude that \eqref{estimate for t2 t2} holds using estimates similar to  \eqref{estimate L part one} and \eqref{estimate L part two}.

This concludes our proof.
\end{proof}

\section{Region $y\leq (\frac{v}{3})^{\alpha}$}\label{section four}
\subsection{Moment estimates}
In this section we prove that 
 \begin{prop}\label{main statement region below}
   There exists a sufficiently small $T\leq 1$ and a sufficiently small $\epsilon_{1}\in(0,1)$, such that for all $\epsilon\leq\epsilon_{1}$ and for all $n\in\mathbb{N}$ the following statement holds. Let $b\geq \overline{b}(\gamma,\alpha)$, $m\geq \frac{b}{\alpha}+1$ in \eqref{def t1}-\eqref{def t3}, where $\overline{b}(\gamma,\alpha)$ is as in \eqref{definition overline b}.  Let $H_{n}$ be as in \eqref{iteration supersol}. If 
     \begin{align*}
     H_{n}(y,v,t)\leq & \textup{ } T_{1}(y,v,t)+ T_{2}(y,v,t)+ T_{3}(y,v,t),
    \end{align*}
 where $T_{1},T_{2},T_{3}$ are as in \eqref{def t1}-\eqref{def t3}, then it holds that  $H_{n+1}(y,v,t)\leq T_{3}(y,v,t),$ for all $t\in[0,T]$ and $y\leq (\frac{v}{3})^{\alpha}$, $v\in(0,\infty)$.
 \end{prop}
 \begin{rmk}
 $T$ will be chosen in \eqref{definition time supersol} and depends on $A$ in \eqref{initial condition decay}.
 \end{rmk}
 Let
\begin{equation}\label{chirdef}
\xi\in C([0,\infty))\,, \quad \xi\colon [0,\infty)\rightarrow [0,1] \; \mbox{ such that } \xi(v)=1
\mbox{ if } v\geq 1 \mbox{ and } \xi(v)=0 \mbox{ if } v\leq \frac{1}{2}.
\end{equation}
The idea behind the proof of Proposition \ref{main statement region below} is to show that $G_{\epsilon}(y,v,t)=e^{\frac{t}{\epsilon}}\tilde{G}_{\epsilon}(y,v,t)$ (or a suitably modified version of this function) is a supersolution  for \eqref{iteration supersol}, where $\tilde{G}_{\epsilon}$ solves
\begin{align*}
\partial_{t}\tilde{G}_{\epsilon}(y,v,t)+\frac{1}{\epsilon}(v^{\alpha}-y)\partial_{y}\tilde{G}_{\epsilon}(y,v,t)+\frac{Lv^{\gamma}\xi(v)}{1+|y|^{d}}\partial_{v}\tilde{G}_{\epsilon}(y,v,t)&=0,\\
(1+v^{b})\tilde{G}_{\epsilon}(y,v,0)&=\frac{A}{1+|y-v^{\alpha}|^{m}},
\end{align*}
with $A$ as in \eqref{initial condition decay} and for some suitably chosen $L>0$ and $d\in\mathbb{N}$. We will define $d$ in \eqref{we need function to be even} and $L$ in \eqref{definition l} and notice that $L$ depends on $A$. We will then choose $\epsilon_{1}$ in Proposition \ref{main statement region below} to depend on $L$. For the choice of $\epsilon_{1}$, we refer to Proposition \ref{proposition about inequalities} and Proposition \ref{der v g negative}. Notice that $G_{\epsilon}$ satisfies the following equation
\begin{align}\label{equation supersolution}
      \partial_{t}G_{\epsilon}(y,v,t)+\frac{1}{\epsilon}(v^{\alpha}-y)\partial_{y}G_{\epsilon}(y,v,t)+\frac{Lv^{\gamma}\xi(v)}{1+|y|^{d}}\partial_{v}G_{\epsilon}(y,v,t)-\frac{1}{\epsilon}G_{\epsilon}(y,v,t)=0.
\end{align}
We first prove some moment estimates.

\begin{prop}[Moment estimates]\label{moment estimates prop}
Let $T\leq 1$, $k\geq 0$. Assume that $P:\mathbb{R}\times (0,\infty)\times [0,T]\rightarrow [0,\infty)$ is such that
\begin{align*}
 P(y,v,t)\leq T_{1}(y,v,t)+T_{2}(y,v,t)+T_{3}(y,v,t),
    \end{align*}
 for all $t\in[0,T]$,  where $T_{1},T_{2},T_{3}$ are as in \eqref{def t1}-\eqref{def t3} with $b>k+1$, $m\geq \frac{b}{\alpha}+1$ in \eqref{def t1}-\eqref{def t2}. It then holds that there exists a constant $K_{1}>0$, which depends on $\alpha,\gamma,b,m,$ such that
\begin{align}\label{desired moment estimate}
    \int_{(0,\infty)}w^{k}P(y,w,t)\der w\leq \frac{K_{1}A^{3}}{1+|y|^{\frac{b-k-1}{\alpha}}}, \textup{ for all } t\in[0,T], y\in\mathbb{R},
\end{align}
where $A$ is as in \eqref{initial condition decay}. In particular, for $k=1$, it holds that if $b>2$, then
\begin{align}
    \int_{(0,\infty)}wP(y,w,t)\der w\leq \frac{K_{1}A^{3}}{1+|y|^{\frac{b-2}{\alpha}}}, \textup{ for all } t\in[0,T], y\in\mathbb{R}.
\end{align}
\end{prop}
\begin{proof}

We begin by noticing that
\begin{align*}
  (1+v^{b})[T_{1}(y,v,t)+T_{2}(y,v,t)+T_{3}(y,v,t)] \leq &C(\alpha,\gamma,b,m)A^{3}\big[ e^{\frac{t}{\epsilon}}\psi(e^{\frac{t}{\epsilon}}(y-v^{\alpha}))     \chi_{1}(y,v,t)\\
  &+\frac{ \epsilon}{|y-v^{\alpha}|} \chi_{2}(y,v,t)+e^{\frac{t}{\epsilon}}\psi(e^{\frac{t}{\epsilon}}(y-v^{\alpha}))  \chi_{3}(y,v,t)\big],
\end{align*}
where $\chi_{1},\chi_{2},\chi_{3}$ are as in \eqref{chi1}-\eqref{chi3}. This is since by the definition of $C_{0}$ in \eqref{defc0}, we have that $C_{0}$ is of order $A^{\frac{1}{m-1}}$.

 Let $y\geq 0$. We distinguish between different regions. \begin{itemize}
\item Assume first that $\{y\geq w^{\alpha}-C_{0}^{-1}\epsilon^{-\frac{1}{m-1}} e^{-\frac{t}{\epsilon}}\}$. Then
\end{itemize}
\begin{align*}
    \int_{\{y\geq w^{\alpha}-C_{0}^{-1}\epsilon^{-\frac{1}{m-1}} e^{-\frac{t}{\epsilon}}\}}w^{k}P(y,w,t)\der w\leq \int_{(0,\infty)} \frac{2 A w^{k} e^{\frac{t}{\epsilon}}}{(1+w^{b})(1+e^{\frac{tm}{\epsilon}}|y-w^{\alpha}|^{m})}\der w.
\end{align*}
 We make the change of variables
 $z=e^{\frac{t}{\epsilon}}(y-w^{\alpha})$ and we obtain that
 \begin{align*}
        \int_{\{y\geq w^{\alpha}-C_{0}^{-1}\epsilon^{-\frac{1}{m-1}} e^{-\frac{t}{\epsilon}}\}}w^{k}P(y,w,t)\der w&\leq \int_{(-\infty,e^{\frac{t}{\epsilon}}y)} \frac{2A (y-e^{-\frac{t}{\epsilon}}z)^{\frac{k+1-\alpha}{\alpha}}}{(1+(y-e^{-\frac{t}{\epsilon}}z)^{\frac{b}{\alpha}})(1+|z|^{m})}\der z.
        \end{align*}
        On one side, if $|z|\leq \frac{e^{\frac{t}{\epsilon}}y}{2}$, it holds that $\frac{y}{2}\leq y-e^{-\frac{t}{\epsilon}}z\leq \frac{3 y}{2}$. Thus
        \begin{align}\label{step 1 moments} \int_{|z|\leq \frac{e^{\frac{t}{\epsilon}}y}{2}} \frac{(y-e^{-\frac{t}{\epsilon}}z)^{\frac{k+1-\alpha}{\alpha}}}{(1+(y-e^{-\frac{t}{\epsilon}}z)^{\frac{b}{\alpha}})(1+|z|^{m})}\der z\leq \frac{C(\alpha,b) |y|^{\frac{k+1-\alpha}{\alpha}}}{1+|y|^{\frac{b}{\alpha}}}\int_{(-\infty,\infty)} \frac{C(\alpha,b) }{1+|z|^{m}}\leq\frac{C(\alpha,b)| y|^{\frac{k+1-\alpha}{\alpha}}}{1+|y|^{\frac{b}{\alpha}}}.
 \end{align}
On the other hand, $\frac{(y-e^{-\frac{t}{\epsilon}}z)^{\frac{k+1-\alpha}{\alpha}}}{(1+(y-e^{-\frac{t}{\epsilon}}z)^{\frac{b}{\alpha}})}\leq 1$. Thus
\begin{align}\label{step 2 moments}
   \int_{ |z|\geq \frac{e^{\frac{t}{\epsilon}}y}{2}} \frac{(y-e^{-\frac{t}{\epsilon}}z)^{\frac{k+1-\alpha}{\alpha}}}{(1+(y-e^{-\frac{t}{\epsilon}}z)^{\frac{b}{\alpha}})(1+|z|^{m})}\der z\leq   \int_{ |z|\geq \frac{e^{\frac{t}{\epsilon}}y}{2}} \frac{\der z}{1+|z|^{m}}\leq \frac{C}{1+(e^{\frac{t}{\epsilon}}|y|)^{m-1}}\leq\frac{C}{1+|y|^{m-1}}.
\end{align}
Combining \eqref{step 1 moments} with \eqref{step 2 moments} and the fact that $m\geq \frac{b}{\alpha}+1$, we obtain that \eqref{desired moment estimate} holds in the region $\{y\geq w^{\alpha}-C_{0}^{-1}\epsilon^{-\frac{1}{m-1}} e^{-\frac{t}{\epsilon}}\}$.
\begin{itemize}
  \item  We now analyze the contribution of the region $(\frac{w}{3})^{\alpha}\leq y\leq w^{\alpha}-C_{0}^{-1}\epsilon^{-\frac{1}{m-1}}  e^{-\frac{t}{\epsilon}}$. Notice that in this region it holds that $y^{\frac{1}{\alpha}}\leq (y+C_{0}^{-1}\epsilon^{-\frac{1}{m-1}}  e^{-\frac{t}{\epsilon}})^{\frac{1}{\alpha}}\leq w\leq 3y^{\frac{1}{\alpha}}$. We remember that $C_{0}$ is of order $A^{\frac{1}{m-1}}$. Thus
\end{itemize} 
 \begin{align*}
      \int_{\{(\frac{w}{3})^{\alpha}\leq y\leq w^{\alpha}-C_{0}^{-1}\epsilon^{-\frac{1}{m-1}}  e^{-\frac{t}{\epsilon}}\}}w^{k}P(y,w,t)\der w&\leq \int^{3y^{\frac{1}{\alpha}}}_{(y+C_{0}^{-1}\epsilon^{-\frac{1}{m-1}}  e^{-\frac{t}{\epsilon}})^{\frac{1}{\alpha}}} \frac{C(\alpha,\gamma,b,m)A^{3}w^{k} \epsilon}{(1+w^{b})|y-w^{\alpha}|}\der w\\
      &\leq \frac{C(\alpha,\gamma,b,m)A^{3}\epsilon y^{\frac{k}{\alpha}}}{1+y^{\frac{b}{\alpha}}}\int^{3y^{\frac{1}{\alpha}}}_{(y+C_{0}^{-1}\epsilon^{-\frac{1}{m-1}}  e^{-\frac{t}{\epsilon}})^{\frac{1}{\alpha}}} \frac{1}{|y-w^{\alpha}|}\der w.
 \end{align*}
We make the change of variables $z=y-w^{\alpha}$ and obtain that 
\begin{align*}
  \frac{\epsilon y^{\frac{k}{\alpha}}}{1+y^{\frac{b}{\alpha}}}  \int^{3y^{\frac{1}{\alpha}}}_{(y+C_{0}^{-1}\epsilon^{-\frac{1}{m-1}}  e^{-\frac{t}{\epsilon}})^{\frac{1}{\alpha}}} \frac{1}{|y-w^{\alpha}|}\der w&\leq \frac{C(\alpha)\epsilon y^{\frac{k+1-\alpha}{\alpha}}}{1+y^{\frac{b}{\alpha}}}\int_{-(3^{\alpha}-1)y}^{-C_{0}^{-1}\epsilon^{-\frac{1}{m-1}}e^{-\frac{t}{\epsilon}}}\frac{\der z}{|z|}\\
  & \leq \frac{C(\alpha)\epsilon y^{\frac{k+1-\alpha}{\alpha}}}{1+y^{\frac{b}{\alpha}}}\big[\frac{t}{\epsilon}+|\ln{y}|\big]\leq \frac{C(\alpha)}{1+y^{\frac{b-k-1}{\alpha}}},
\end{align*}
where in the last inequality we used the fact that $t\leq 1$.
\begin{itemize}
 \item We now look at the contribution of the region $\{0\leq y \leq (\frac{w}{3})^{\alpha}\}$. In this region, we have that $|y-w^{\alpha}|=w^{\alpha}-y$ and $\big[1-\frac{1}{3^{\alpha}}\big]w^{\alpha}\leq w^{\alpha}-y\leq w^{\alpha}$. Then
\end{itemize}
\begin{align*}
        \int_{\{0\leq y\leq (\frac{w}{3})^{\alpha}\}}w^{k}P(y,w,t)\der w &\leq \int_{\{0\leq y\leq (\frac{w}{3})^{\alpha}\}}\frac{C(\alpha,\gamma,b,m)A e^{\frac{t}{\epsilon}}}{1+e^{\frac{tm}{\epsilon}}(w^{\alpha}-y)^{m}}\frac{ w^{k}}{1+w^{b}}\der w \\
        &\leq \int_{\{0\leq y\leq (\frac{w}{3})^{\alpha}\}}\frac{C(\alpha,\gamma,b,m)A e^{\frac{t}{\epsilon}}}{1+e^{\frac{tm}{\epsilon}}w^{\alpha m}}\frac{ w^{k}}{1+w^{b}}\der w.
        \end{align*}
        We make the change of variable $z=e^{\frac{t}{\epsilon}}w^{\alpha}$ and obtain that 
\begin{align*}
         \int_{\{0\leq y\leq (\frac{w}{3})^{\alpha}\}}w^{k}P(y,w,t)\der w &\leq C(\alpha,\gamma,b,m)A\int_{(3^{\alpha}e^{\frac{t}{\epsilon}}y,\infty)}\frac{(ze^{-\frac{t}{\epsilon}})^{\frac{k+1-\alpha}{\alpha}}\der z}{1+z^{m}}\\
         &\leq C(\alpha,\gamma,b,m)A\int_{(3^{\alpha} e^{\frac{t}{\epsilon}} y,\infty)}\frac{z^{\frac{k+1-\alpha}{\alpha}}\der z}{1+z^{m}}\leq \frac{C(\alpha,\gamma,b,m)A}{1+y^{m-\frac{k+1-\alpha}{\alpha}}},
\end{align*}
This concludes our proof in the case $y\geq 0$ since $m\geq \frac{b}{\alpha}$. If $y\leq 0$, by making the change of variable $z=e^{\frac{t}{\epsilon}}(w^{\alpha}-y)$, we obtain that
\begin{align*}
    \int w^{k}P(y,w,t)\der w&\leq \int_{(0,\infty)} \frac{C(\alpha,\gamma,b,m)Aw^{k} e^{\frac{t}{\epsilon}}}{(1+w^{b})(1+e^{\frac{tm}{\epsilon}}|y-w^{\alpha}|^{m})}\der w\\
    &\leq \int_{(e^{\frac{t}{\epsilon}}|y|,\infty)} \frac{C(\alpha,\gamma,b,m)A }{1+|z|^{m}}\der z\leq \frac{C(\alpha,\gamma,b,m)A}{1+(e^{\frac{t}{\epsilon}}|y|)^{m-1}}.
\end{align*}
This concludes our proof.
\end{proof}
We continue by analyzing properties of the transport equation \eqref{equation supersolution}. Define $d$ via
  \begin{equation}\label{we need function to be even}
\left\{\begin{aligned}
d&=\bigg\lfloor\frac{b-2}{\alpha}\bigg\rfloor-1; &\textup{ if }\bigg\lfloor\frac{b-2}{\alpha}\bigg\rfloor& \textup{ odd; }\\
d&=\bigg\lfloor\frac{b-2}{\alpha}\bigg\rfloor;&\textup{ if }\bigg\lfloor\frac{b-2}{\alpha}\bigg\rfloor& \textup{ even, }
   \end{aligned}\right.
   \end{equation}
   where $\lfloor t \rfloor$ denotes the floor function. We use the method of characteristics in \eqref{equation supersolution}
  \begin{equation}\label{the original characteristics}
\left\{\begin{aligned}
\partial_{t}{Y_{\epsilon}}(y,v,t)&=-\frac{1}{\epsilon}({V_{\epsilon}}^{\alpha}-Y), & \qquad {Y_{\epsilon}}(y,v,0)&=y\leq \big(\frac{v}{3}\big)^{\alpha}\,, \\
\partial_{t}{V_{\epsilon}}(y,v,t)&=-\frac{{L}{V_{\epsilon}}^{\gamma}\xi(V_{\epsilon})}{1+|{Y_{\epsilon}}|^{d}},& \qquad {V_{\epsilon}}(y,v,0)&=v.
   \end{aligned}\right.
   \end{equation}
 Notice that for $b>4$ then $d>1$ and even.  It immediately follows from \eqref{the original characteristics} that
  \begin{align}\label{der v der y}
      \frac{\der V_{\epsilon}}{\der Y_{\epsilon}}=\frac{\epsilon L V_{\epsilon}^{\gamma}\xi(V_{\epsilon})}{(1+|Y_{\epsilon}|^{d})(V_{\epsilon}^{\alpha}-Y_{\epsilon})}.
  \end{align}
 Notice that this is different from the analysis of the characteristics performed in \cite{cristianinhom} since in that case we could make use of the separation of variables in \eqref{der v der y} since the transport term in the space variable  contributed with $V_{\epsilon}^{\alpha}$ instead of $V_{\epsilon}^{\alpha}-Y_{\epsilon}$. We thus have to apply different methods than in \cite{cristianinhom} to prove the desired estimates in our case.
\subsection{Properties of the characteristics of the equation satisfied by the supersolution}
   \begin{prop}\label{proposition about inequalities}
Let $\alpha\in(0,1)$, $\gamma\in(1,1+\alpha)$, and $d>1$ even. Let $L>0$ and let $Y_{\epsilon}, V_{\epsilon}$ be as in \eqref{the original characteristics}. Then there exists a sufficiently small $\epsilon_{1}\in(0,1)$, which depends on $L$, such that for all $t \geq 0$ and $\epsilon\leq \epsilon_{1}$, the following estimates hold for all $v>0$ and $y\leq \big(\frac{v}{3}\big)^{\alpha}$.
      \begin{align}
        \frac{9}{10} v &\leq  V_{\epsilon}(y,v,t)\leq v;\label{ineq}\\
     Y_{\epsilon}(y,v,t)&\leq \big(\frac{v}{3}\big)^{\alpha};\label{ineq part two}\\
e^{\frac{t}{\epsilon}} (y-v^{\alpha}) & \leq   Y_{\epsilon}-V_{\epsilon}^{\alpha};\label{ineqq}\\
 Y_{\epsilon}-V_{\epsilon}^{\alpha}&\leq e^{\frac{t}{\epsilon}} (y-\Big(\frac{9}{10}\Big)^{\alpha}v^{\alpha}).\label{ineq 3}
      \end{align}
  \end{prop}
  \begin{rmk}
      We will later fix $L$ in \eqref{definition l} to depend on $A$ and other constants depending only on the parameters $\alpha,\gamma,b,m$. Thus $\epsilon_{1}$ will depend on $A$.
  \end{rmk}
  \begin{proof}[Proof of Proposition \ref{proposition about inequalities}]
In order to simplify the notation, we replace the notation of the pair $(Y_{\epsilon},V_{\epsilon})$ in \eqref{the original characteristics} with $(Y,V)$, while keeping in mind the dependence on $\epsilon$. Moreover, we will assume in all the following that $v\geq 1$ and will work with the following system of ODE's.
 \begin{equation}
\left\{\begin{aligned}\label{system ode without epsilon notation}
\partial_{t}{Y}(y,v,t)&=-\frac{1}{\epsilon}({V}^{\alpha}-Y), & \qquad {Y}(y,v,0)&=y\leq (\frac{v}{3})^{\alpha}\,, \\
\partial_{t}{V}(y,v,t)&=-\frac{{L}{V}^{\gamma}}{1+|{Y}|^{d}},& \qquad {V}(y,v,0)&=v.
   \end{aligned}\right.
   \end{equation} This is since there is nothing to prove in the case $v\leq \frac{1}{2}$ since by the definition of $\xi$ in \eqref{chirdef}, it holds that $\partial_{t}V_{\epsilon}=0$ when $v\leq \frac{1}{2}$. Moreover, assume that  \eqref{ineq}-\eqref{ineq 3} hold for $v\geq \frac{1}{2}$ and for \eqref{system ode without epsilon notation}. We notice that the following inequalities hold. We have that $\partial_{t}V_{\epsilon}(y,v,t)=-\frac{LV_{\epsilon}^{\gamma}\xi(V_{\epsilon})}{1+|Y_{\epsilon}|^{d}}\geq -\frac{LV_{\epsilon}^{\gamma}}{1+|Y_{\epsilon}|^{d}}$ and that $\partial_{t}V_{\epsilon}(y,v,t)\leq 0$. Thus, \eqref{ineq} follows. If \eqref{ineq} holds, then \eqref{ineq part two} follows from \eqref{the original characteristics}.  \eqref{ineqq}-\eqref{ineq 3} follow from \eqref{ineq}-\eqref{ineq part two}. Thus, it suffices to analyze the system \eqref{system ode without epsilon notation} with $v\geq \frac{1}{2}$. We assume in addition that $v\geq 1$ for simplicity of notation.

   From \eqref{the original characteristics}, we have that
   \begin{align}\label{lower bound der t v}
       \partial_{t}{V}(y,v,t)=-\frac{{L}{V}^{\gamma}}{1+|{Y}|^{d}}\geq-L V^{\gamma}.
   \end{align}Thus, there exists a time $t_{1}:=t_{1}(v)$ sufficiently small, which is independent of $\epsilon$ but dependent on $v$, such that
      \begin{align*}
      \frac{9}{10}v\leq V(y,v,s)\leq v,
      \end{align*}
      for all $0\leq s\leq t_{1}$. This implies that for all $s\leq t_{1}$ it holds that
  \begin{align}\label{iterative argument characteristics}
  \partial_{s}{Y}(y,v,s)&=\frac{1}{\epsilon}(Y-{V}^{\alpha})\leq \frac{1}{\epsilon}(Y-\Big(\frac{9}{10}\Big)^{\alpha}{v}^{\alpha}).
  \end{align}

  We analyze the sign of the following ODE.
  \begin{align}\label{ode keep sign}
  \partial_{z}{\overline{Y}}(y,v,z)&= \frac{1}{\epsilon}(\overline{Y}-\Big(\frac{9}{10}\Big)^{\alpha}{v}^{\alpha}), \quad \overline{Y}(y,v,0)=y\leq (\frac{v}{3})^{\alpha}.
  \end{align}
  Since $\overline{Y}=\Big(\frac{9}{10}\Big)^{\alpha}{v}^{\alpha}$ is the zero point of \eqref{ode keep sign} and $\overline{Y}(y,v,0)=y\leq (\frac{v}{3})^{\alpha}\leq \Big(\frac{9}{10}\Big)^{\alpha}{v}^{\alpha}$, then we have that $\overline{Y}(y,v,z)\leq \Big(\frac{9}{10}\Big)^{\alpha}{v}^{\alpha}$, for all $z\geq 0$ and furthermore $\overline{Y}(y,v,z)\leq y\leq (\frac{v}{3})^{\alpha}$, for all $z\geq 0$.

  By \eqref{iterative argument characteristics} and using a comparison principle, this implies that $Y(y,v,s)\leq (\frac{v}{3})^{\alpha}$, for all $s\leq t_{1}$ and in particular $Y(y,v,t_{1})\leq (\frac{v}{3})^{\alpha}$.
We thus have that $V(s)\geq \frac{9}{10}v$ and $Y(s)\leq (\frac{v}{3})^{\alpha}$, for all $s\leq t_{1}$. Then \eqref{iterative argument characteristics} becomes
  \begin{align}\label{better inequality for Y}
      \partial_{s}Y\leq \frac{v^{\alpha}}{\epsilon}\bigg[\frac{1}{3^{\alpha}}-\Big(\frac{9}{10}\Big)^{\alpha}\bigg]\leq -\frac{c v^{\alpha}}{\epsilon},
  \end{align}
for some $c>0$.

\begin{enumerate}\item[i)] Assume first that $Y(y,v,0)=y\in[0,(\frac{v}{3})^{\alpha}]$ and  $Y(y,v,t_{1})<0$.
\end{enumerate} We have that for any $y\in\mathbb{R}$ and $v>0$ there exists $\tau:=\tau(y,v)\in[0,t_{1})$ such that $Y(y,v,\tau(y,v))=0$. Thus, for all $s\in[0,t_{1}]$, it holds that
\begin{align*}
 |Y(y,v,s)|=   |Y(y,v,s)-Y(y,v,\tau(y,v))|\geq \frac{c v^{\alpha}}{\epsilon}|s- \tau(y,v)|.
\end{align*}

Let $s\in[0,t_{1}]$. Since $\frac{9}{10}v\leq V\leq v$, it holds that
\begin{align*}
    \partial_{s}{V}(y,v,s)=-\frac{{L}{V}^{\gamma}}{1+|{Y}|^{d}}\geq -\frac{{L}{V}^{\gamma}}{1+\frac{c v^{\alpha d}}{\epsilon^{d}}|s-\tau(y,v)|^{d}}\geq -\frac{C{L}{v}^{\gamma}}{1+\frac{c v^{\alpha d}}{\epsilon^{d}}|s-\tau(y,v)|^{d}}.
\end{align*}
In other words, we have that
\begin{align}\label{inequality for V}
    V(y,v,s)-v\geq -v^{\gamma}\int_{0}^{s}\frac{CL\der z}{1+\frac{c v^{\alpha d}}{\epsilon^{d}}|z-\tau(y,v)|^{d}}.
\end{align}
By making the change of variables $\xi=\frac{v^{\alpha}}{\epsilon}(z-\tau(y,v))$ in \eqref{inequality for V}, it follows that
\begin{align}\label{epsilon does not depend on v}
    V(y,v,s)\geq v -\epsilon v^{\gamma-\alpha}\int_{-\infty}^{\infty}\frac{CL\der \xi}{1+c|\xi|^{d}}\geq v\bigg[1-CL\epsilon v^{\gamma-\alpha-1}\bigg].
\end{align}
Since $v\geq 1$ and $\gamma<\alpha+1$, there exists $\epsilon_{1}\in(0,1)$ sufficiently small, depending on $L$, but independent of $v$ such that $\bigg[1-CL\epsilon v^{\gamma-\alpha-1}\bigg]\geq \frac{19}{20}$ and thus
\begin{align}\label{iteration in vv}
    V(y,v,s)\geq \frac{19}{20}v,
\end{align}
for all $s\in[0,t_{1}]$ and all $\epsilon\leq \epsilon_{1}$. 

From \eqref{iteration in vv}, \eqref{lower bound der t v}, and by continuity, we deduce that there exists $t_{2}:=t_{2}(v)>t_{1}$ such that
  \begin{align*}
      \frac{9}{10}v\leq V(y,v,s)\leq v,       \textup{ for all } s\in[t_{1},t_{2}].
      \end{align*}

 \begin{enumerate}\item[ii)]   Assume $Y(y,v,0)=y\in[0,(\frac{v}{3})^{\alpha}]$ and  $Y(y,v,t_{1})\geq 0$. 
 \end{enumerate}Since $Y(y,v,s)\geq Y(y,v,t_{1})\geq 0$, for all $s\in[0,t_{1}]$, it holds that
\begin{align*}
 |Y(y,v,s)|=Y(y,v,s)\geq Y(y,v,s)-Y(y,v,t_{1})\geq   |Y(y,v,s)-Y(y,v,t_{1})|\geq \frac{c v^{\alpha}}{\epsilon}|s- t_{1}|.
\end{align*}
    We can then perform similar computations as in the case $Y(y,v,0)=y\in[0,(\frac{v}{3})^{\alpha}]$ and  $Y(y,v,t_{1})< 0$ in order to deduce that $V(y,v,s)\geq \frac{19}{20}v,$ for all $s\leq t_{1}$ and all $\epsilon\leq \epsilon_{1}$. 
      
     \begin{enumerate}\item[iii)]    Assume $Y(y,v,0)=y<0$. \end{enumerate}Then for all $s\in[0,t_{1}]$, it holds by  \eqref{iterative argument characteristics} that
  \begin{align}
Y(y,v,s)\leq -\Big(\frac{9}{10}\Big)^{\alpha}\frac{v^{\alpha}s}{\epsilon}
  \end{align}
        and thus, as before, we can conclude that
        \begin{align*}
    \partial_{s}{V}(y,v,s)=-\frac{{L}{V}^{\gamma}}{1+|{Y}|^{d}}\geq -\frac{{L}{V}^{\gamma}}{1+\frac{c v^{\alpha d}}{\epsilon^{d}}s^{d}}\geq -\frac{C{L}{v}^{\gamma}}{1+\frac{c v^{\alpha d}}{\epsilon^{d}}s^{d}}.
\end{align*}
        As before, we obtain that \begin{align}
    V(y,v,s)\geq \frac{19}{20}v,
\end{align}
for all $s\in[0,t_{1}]$ and all $\epsilon\leq \epsilon_{1}$.

We are now able to iterate the argument in order to obtain that \eqref{ineq} holds for all $t\geq 0$. Notice that at each iteration the times $t_{i}$, $i\in\mathbb{N}$, depend on $v$, but after the iteration we obtain that the estimate \eqref{ineq} holds for all times $t\geq 0$. Moreover, $\epsilon_{1}$ in \eqref{epsilon does not depend on v} does not depend on $v$. The inequality \eqref{ineq part two} follows from \eqref{ineq} and \eqref{iterative argument characteristics}.

By \eqref{ineq} and \eqref{the original characteristics} it follows that
\begin{align*}
    \partial_{t}Y-\frac{Y}{\epsilon}=-\frac{V^{\alpha}}{\epsilon}\geq -\frac{v^{\alpha}}{\epsilon}
\end{align*}
     and thus 
     \begin{align*}
         \partial_{t}\bigg(e^{-\frac{t}{\epsilon}}Y\bigg )\geq \partial_{t}\bigg(e^{-\frac{t}{\epsilon}}v^{\alpha}\bigg).
     \end{align*}
    Integrating in time and using the fact that $V\leq v$ we obtain \eqref{ineqq}. \eqref{ineq 3} follows in a similar manner by noticing that
    \begin{align*}
           \partial_{t}Y-\frac{Y}{\epsilon}=-\frac{V^{\alpha}}{\epsilon}\leq  -\Big(\frac{9}{10}\Big)^{\alpha}\frac{v^{\alpha}}{\epsilon}.
    \end{align*}
This concludes our proof.
  \end{proof}
  \begin{rmk}\label{remark v alpha y}
     From \eqref{ineq 3} we immediately deduce that
   for all $y\leq (\frac{v}{3})^{\alpha},$ the following inequality holds.  \begin{align*}
         Y_{\epsilon}-V_{\epsilon}^{\alpha}\leq \bigg[\Big(\frac{9}{10}\Big)^{\alpha}-\frac{1}{3^{\alpha}}\bigg]e^{\frac{t}{\epsilon}}(y-v^{\alpha}).
     \end{align*}
  \end{rmk}
  \begin{proof}
      For $y\leq 0 $, it immediately follows that
      \begin{align*}
          Y_{\epsilon}-V_{\epsilon}^{\alpha}\leq e^{\frac{t}{\epsilon}} (y-\Big(\frac{9}{10}\Big)^{\alpha}v^{\alpha})\leq\Big(\frac{9}{10}\Big)^{\alpha} e^{\frac{t}{\epsilon}} (y-v^{\alpha})\leq  \bigg[\Big(\frac{9}{10}\Big)^{\alpha}-\frac{1}{3^{\alpha}}\bigg]e^{\frac{t}{\epsilon}}(y-v^{\alpha}).
      \end{align*}
      If $y\in(0,(\frac{v}{3})^{\alpha})$, then
       \begin{align*}
          Y_{\epsilon}-V_{\epsilon}^{\alpha}\leq e^{\frac{t}{\epsilon}} (y-\Big(\frac{9}{10}\Big)^{\alpha}v^{\alpha})\leq -\bigg[\Big(\frac{9}{10}\Big)^{\alpha}-\frac{1}{3^{\alpha}}\bigg] e^{\frac{t}{\epsilon}} v^{\alpha} \leq \bigg[\Big(\frac{9}{10}\Big)^{\alpha}-\frac{1}{3^{\alpha}}\bigg] e^{\frac{t}{\epsilon}} (y-v^{\alpha}).
          \end{align*}
  \end{proof}
  \subsection{Properties of the supersolution in the region $y\leq (\frac{v}{3})^{\alpha} $}
 \begin{prop}\label{der v g negative}
Fix $\alpha\in(0,1)$, $\gamma\in(1,1+\alpha)$, $L>0$ and $\epsilon\in(0,1)$. Let $d>1$ in \eqref{we need function to be even} even and sufficiently large such that $\gamma\leq \frac{d\alpha}{d+1}+1$. Let $G_{\epsilon,L}$ solve equation \eqref{equation supersolution}. Then there exists a sufficiently small $\epsilon_{1}\in(0,1)$ such that for all $t \geq 0$ and $\epsilon\leq \epsilon_{1}$, we have that for all $v>0$ and $y\leq \big(\frac{v}{3}\big)^{\alpha}$,
\begin{align*}
    \partial_{v}G_{\epsilon,L}(y,v,t)\leq 0.
\end{align*}
 \end{prop}
 \begin{proof}
  Let $Y_{\epsilon}, V_{\epsilon}$ be as in \eqref{the original characteristics}.    We have that 
     \begin{align*}
         G_{\epsilon,L}(y,v,t)=\frac{A e^{\frac{t}{\epsilon}}}{(1+V_{\epsilon}^{b})(1+|Y_{\epsilon}-V_{\epsilon}^{\alpha}|^{m})},
     \end{align*}
with $A$ as in \eqref{initial condition decay}. As before, in order to simplify the notation, we replace the notation of the pair $(Y_{\epsilon},V_{\epsilon})$ in \eqref{the original characteristics} with $(Y,V)$, while keeping in mind the dependence on $\epsilon$.    Moreover, from Proposition \ref{proposition about inequalities}, we know that $Y-V^{\alpha}\leq 0$, for all $t\geq 0$. Thus
     \begin{align*}
         \partial_{v}G_{\epsilon,L}(y,v,t)=-\frac{mA e^{\frac{t}{\epsilon}}(V^{\alpha}-Y)^{m-1}[\alpha V^{\alpha-1}\partial_{v}V-\partial_{v}Y]}{(1+V^{b})(1+(V^{\alpha}-Y)^{m})^{2}}-\frac{b Ae^{\frac{t}{\epsilon}}V^{b-1}\partial_{v}V}{(1+V^{b})^{2}(1+(V^{\alpha}-Y)^{m})}.
     \end{align*}
     Thus, in order to prove that $\partial_{v}G_{\epsilon,L}(y,v,t)\leq 0$ it suffices to prove that $\alpha V^{\alpha-1}\partial_{v}V-\partial_{v}Y\geq 0$ and $\partial_{v}V\geq 0$. Fix $t>0$. We know that $V\geq 0$ by Proposition \ref{proposition about inequalities}. We will prove that
     \begin{align*}
         \partial_{v}V(y,v,t)\geq 0, \textup{ for all } v>0, y\leq (\frac{v}{3})^{\alpha},
     \end{align*}
     and that
     \begin{align*}
         \partial_{v}Y(y,v,t)\leq 0, \textup{ for all } v>0, y\leq (\frac{v}{3})^{\alpha},
     \end{align*}
     thus concluding our proof. Since $\partial_{v}V(y,v,0)=1$, by continuity there exists $t_{1}>0$ such that $\partial_{v}V(y,v,s)\in[\frac{1}{2},1]$, for all $s\in[0,t_{1}]$. By \eqref{the original characteristics}, we have that
     \begin{align*}
         \partial_{s}\partial_{v}Y(y,v,s)=\frac{\partial_{v}Y}{\epsilon}-\frac{\alpha V^{\alpha-1}\partial_{v}V}{\epsilon}, \qquad \partial_{v}Y(y,v,0)=0.
     \end{align*}
     In other words, using Proposition \ref{proposition about inequalities} and since $\alpha<1$, it holds that for all $s\in[0,t_{1}]$ 
     \begin{align*}
             \partial_{s}\big(e^{-\frac{s}{\epsilon}}\partial_{v}Y(y,v,s)\big)=-e^{-\frac{s}{\epsilon}}\frac{\alpha V^{\alpha-1}\partial_{v}V}{\epsilon}\leq -e^{-\frac{s}{\epsilon}}\frac{\alpha v^{\alpha-1}}{2\epsilon}.
     \end{align*}
     Thus, by integrating in time, we obtain that
     \begin{align*}
         e^{-\frac{s}{\epsilon}}\partial_{v}Y\leq \frac{\alpha}{2} e^{-\frac{s}{\epsilon}}v^{\alpha-1}-\frac{\alpha}{2} v^{\alpha-1},
     \end{align*}
     for all $s\in[0,t_{1}],$ and thus
     \begin{align*}
         \partial_{v}Y(y,v,s)\leq \frac{\alpha}{2} v^{\alpha-1}(1-e^{\frac{s}{\epsilon}})\leq 0.
     \end{align*}
Using similar arguments, we obtain that
\begin{align}\label{exponential bound for der y}
    \partial_{v}Y(y,v,s)\geq \frac{10^{1-\alpha}\alpha}{9^{1-\alpha}} v^{\alpha-1}(1-e^{\frac{s}{\epsilon}}).
\end{align}

By \eqref{the original characteristics}, we also have that for all $s\in[0,t_{1}]$,
\begin{align}\label{der v v}
    \partial_{s}\partial_{v}V(y,v,s)=-\frac{\gamma L V^{\gamma-1}\partial_{v}V}{1+|Y|^{d}}+\frac{d L V^{\gamma}|Y|^{d-2}Y\partial_{v}Y}{(1+|Y|^{d})^{2}}, \qquad \partial_{v}V(y,v,0)=1.
\end{align}
We will prove that 
\begin{align}\label{linear ode part 1}
    \int_{0}^{s}\frac{ V^{\gamma-1}(z)\der z}{1+|Y|^{d}(z)}\leq C\epsilon
\end{align}
and that 
\begin{align}\label{linear ode part 2}
    \int_{0}^{s}\frac{V^{\gamma}|Y|^{d-1}|\partial_{v}Y|}{(1+|Y|^{d})^{2}}\der z \leq C\epsilon,
\end{align}
for all $s\in[0,t_{1}]$. We first prove that \eqref{linear ode part 1} holds. By \eqref{ineq} we have that
\begin{align*}
   \frac{V^{\gamma-1}}{1+|Y|^{d}}\leq \frac{ v^{\gamma-1}}{1+|Y|^{d}}.
\end{align*}
Moreover, we remember that 
\begin{align*}
    \partial_{s}Y(y,v,s)=\frac{Y}{\epsilon}-\frac{V^{\alpha}}{\epsilon}\leq \frac{v^{\alpha}}{\epsilon}\bigg[\frac{1}{3^{\alpha}}-\Big(\frac{9}{10}\Big)^{\alpha}\bigg]\leq -\frac{c v^{\alpha}}{\epsilon}
\end{align*}
and then as before we have the following cases. \begin{enumerate}\item[i)]  $Y(y,v,0)=y\in[0,(\frac{v}{3})^{\alpha}]$ and $Y(y,v,t_{1})< 0$.
\end{enumerate}
Then there exists $\tau:=\tau(y,v)\in[0,t_{1})$ such that $Y(y,v,\tau(y,v))=0$ and, for all $s\in[0,t_{1}]$, it holds that
\begin{align}\label{linear bound for y}
 |Y(y,v,s)|\geq \frac{c v^{\alpha}}{\epsilon}|s- \tau(y,v)|.
\end{align}
Then 
\begin{align*}
        \int_{0}^{s}\frac{ V^{\gamma-1}(z)\der z}{1+|Y|^{d}(z)}\leq      \int_{0}^{s}\frac{ Cv^{\gamma-1}\der z}{1+ \frac{ v^{\alpha d}}{\epsilon^{d}}|z- \tau(y,v)|^{d}}.
\end{align*}
We make the change of variable $\xi=\frac{v^{\alpha}}{\epsilon}(z-\tau)$ and use the fact that we can consider without loss of generality that $v\geq 1$ in order to obtain that
\begin{align*}
        \int_{0}^{s}\frac{ V^{\gamma-1}(z)\der z}{1+|Y|^{d}(z)}\leq      \int_{-\infty}^{\infty}\frac{ C\epsilon v^{\gamma-\alpha-1}\der \xi}{1+ |\xi|^{d}}\leq C \epsilon.
\end{align*}
\begin{enumerate}
    \item[ii)]If $Y(y,v,0)=y\in[0,(\frac{v}{3})^{\alpha}]$ and $Y(y,v,t_{1})\geq 0$, we have that \begin{align*}
    |Y(y,v,s)|=Y(y,v,s)\geq Y(y,v,s)-Y(y,v,t_{1})=|Y(y,v,s)-Y(y,v,t_{1})|\geq \frac{c v^{\alpha}}{\epsilon}|s- t_{1}|
\end{align*}\end{enumerate}and the result follows as before. 

\begin{enumerate}\item[iii)]Otherwise $Y(y,v,0)=y<0$. Then for all $s\in[0,t_{1}]$, it holds by  \eqref{iterative argument characteristics} that
  \begin{align*}
Y(y,v,s)\leq -\Big(\frac{9}{10}\Big)^{\alpha}\frac{v^{\alpha}s}{\epsilon}
  \end{align*}\end{enumerate}and proceed as in the previous case in order to show that \eqref{linear ode part 1} holds.

We now prove \eqref{linear ode part 2}. We only analyze the case when $Y(y,v,0)=y\in[0,(\frac{v}{3})^{\alpha}]$ and $Y(y,v,t_{1})<0$ as the other two cases can be treated in a similar, more straightforward manner.
Thus, let us assume $Y(y,v,0)=y\in[0,(\frac{v}{3})^{\alpha}]$ and $Y(y,v,t_{1})<0$. Then, for every $y$ and every $v>0$, there exists $\tau:=\tau(y,v)\in[0,t_{1})$ such that $Y(y,v,\tau(y,v))=0$. We can write 
\begin{align*}
     \int_{0}^{s}\frac{V^{\gamma}|Y|^{d-1}|\partial_{v}Y|}{(1+|Y|^{d})^{2}}\der z = \int_{0}^{\tau(y,v)}\frac{V^{\gamma}|Y|^{d-1}|\partial_{v}Y|}{(1+|Y|^{d})^{2}}\der z + \int_{\tau(y,v)}^{s}\frac{V^{\gamma}|Y|^{d-1}|\partial_{v}Y|}{(1+|Y|^{d})^{2}}\der z =:(I)+(II).
\end{align*}
We bound $(II)$.
We have that
\begin{align*}
    \partial_{s}Y=\frac{Y}{\epsilon}-\frac{V^{\alpha}}{\epsilon}\leq \frac{Y}{\epsilon}-(\frac{9}{10})^{\alpha}\frac{ v^{\alpha}}{\epsilon}
\end{align*}
and thus
\begin{align*}
    \partial_{s}(e^{-\frac{s}{\epsilon}}Y)\leq -e^{-\frac{s}{\epsilon}}(\frac{9}{10})^{\alpha}\frac{v^{\alpha}}{\epsilon}.
\end{align*}
We integrate from $\tau(y,v)$ to $s\in[\tau(y,v),t_{1}]$ and use the fact that $Y(y,v,\tau(y,v))=0$ in order to obtain that
\begin{align*}
   e^{-\frac{s}{\epsilon}}Y(s)\leq e^{-\frac{s}{\epsilon}}(\frac{9}{10})^{\alpha}v^{\alpha}-e^{-\frac{\tau}{\epsilon}}(\frac{9}{10})^{\alpha}v^{\alpha}
\end{align*}
or in other words
\begin{align}\label{exponential bound for y}
    Y(y,v,s)\leq (\frac{9}{10})^{\alpha}v^{\alpha}(1-e^{\frac{s-\tau}{\epsilon}})\leq 0.
\end{align}
Combining \eqref{exponential bound for y} with \eqref{exponential bound for der y}, we have that
\begin{align*}
 \bigg|\frac{V^{\gamma}|Y|^{d-1}\partial_{v}Y}{(1+|Y|^{d})^{2}}\bigg|&\leq  \frac{C  v^{\gamma}|\partial_{v}Y|}{1+|Y|^{d+1}}\leq  \frac{C  v^{\gamma}|\partial_{v}Y|}{1+C v^{\alpha (d+1)}(e^{\frac{s-\tau}{\epsilon}}-1)^{d+1}}\\
 &\leq \frac{C  v^{\gamma+\alpha-1}(e^{\frac{s}{\epsilon}}-1)}{1+v^{\alpha (d+1)}(e^{\frac{s-\tau}{\epsilon}}-1)^{d+1}}\\
 &\leq \frac{C  v^{\gamma-1} (e^{\frac{s}{\epsilon}}-1)}{1+ v^{\alpha d}(e^{\frac{s-\tau}{\epsilon}}-1)^{d+1}}.
\end{align*}
Thus
\begin{align*}
\int_{\tau}^{s} \bigg|\frac{V^{\gamma}|Y|^{d-2}Y\partial_{v}Y}{(1+|Y|^{d})^{2}}\bigg|\der z \leq \int_{\tau}^{s}\frac{C  v^{\gamma-1} (e^{\frac{z}{\epsilon}}-1)}{1+ v^{\alpha d}(e^{\frac{z-\tau}{\epsilon}}-1)^{d+1}}\der z.
\end{align*}
We make the change of variables $\xi=v^{\frac{d\alpha }{d+1}}(e^{\frac{z-\tau}{\epsilon}}-1)$ and thus
\begin{align*}
    \der \xi = \frac{v^{\frac{\alpha d}{d+1}}}{\epsilon}e^{\frac{z-\tau}{\epsilon}}\der z.
\end{align*}
We have that  $\tau\leq C\epsilon$. This is since if we choose $s=0$ in \eqref{linear bound for y}, then $|y|\geq \frac{c v^{\alpha}}{\epsilon}| \tau(y,v)|$ and we are in the case when $y\in[0, (\frac{v}{3})^{\alpha}]$. Using that $\tau\leq C\epsilon$ and $v\geq 1$, we further obtain that
\begin{align}\label{bound ii}
\int_{\tau}^{s} \bigg|\frac{V^{\gamma}|Y|^{d-2}Y\partial_{v}Y}{(1+|Y|^{d})^{2}}\bigg|\der z &\leq \int_{0}^{\infty}\frac{C \epsilon  v^{\gamma-\frac{\alpha d}{d+1}-1}e^{-\frac{z-\tau}{\epsilon}}(e^{\frac{z}{\epsilon}}-1) }{1+ |\xi|^{d+1}}\der \xi\leq \int_{0}^{\infty}\frac{C \epsilon  v^{\gamma-\frac{\alpha d}{d+1}-1}e^{\frac{\tau}{\epsilon}}(1-e^{-\frac{z}{\epsilon}}) }{1+ |\xi|^{d+1}}\der \xi\nonumber\\
&\leq \int_{0}^{\infty}\frac{C \epsilon  v^{\gamma-\frac{\alpha d}{d+1}-1}e^{\frac{\tau}{\epsilon}}}{1+ |\xi|^{d+1}}\der \xi\leq\int_{0}^{\infty}\frac{C \epsilon  v^{\gamma-\frac{\alpha d}{d+1}-1}}{1+ |\xi|^{d+1}}\der \xi\leq  C\epsilon
\end{align}
since we chose $d$ sufficiently large such that $\gamma\leq \frac{d\alpha}{d+1}+1$.

We now bound $(I)$. For $s\leq \tau,$ we only have that \eqref{linear bound for y} holds. Using in addition \eqref{exponential bound for der y}, we have that 
\begin{align*}
    \int_{0}^{\tau(y,v)}\frac{V^{\gamma}|Y|^{d-1}|\partial_{v}Y|} {(1+|Y|^{d})^{2}}\der z\leq   \int_{0}^{\tau}\frac{Cv^{\gamma}|\partial_{v}Y|}{1+|Y|^{d+1}}\der z&\leq \int_{0}^{\tau}\frac{C v^{\gamma+\alpha-1}(e^{\frac{z}{\epsilon}}-1)}{1+\frac{v^{\alpha(d+1)}}{\epsilon^{d+1}}|z-\tau|^{d+1}}\der z\\
    &\leq \int_{0}^{\tau}\frac{C v^{\gamma-1}(e^{\frac{z}{\epsilon}}-1)}{1+\frac{v^{d\alpha}}{\epsilon^{d+1}}|z-\tau|^{d+1}}\der z.
\end{align*}
We now make the change of variables 
\begin{align*}
    \xi&=\frac{v^{\frac{\alpha d}{d+1}}}{\epsilon}(\tau-z);\\
     \der \xi&=-\frac{v^{\frac{\alpha d}{d+1}}}{\epsilon}\der z;\\
     \frac{z}{\epsilon}&=\frac{\tau-\epsilon v^{-\frac{\alpha d}{d+1}}\xi}{\epsilon}\leq \frac{\epsilon \xi +\tau}{\epsilon}= \xi+\frac{\tau}{\epsilon}\leq\xi+1.
\end{align*}
Thus, using that $e^{x}-1\leq Cx,$ when $x\in[0,1]$, we further obtain that
\begin{align}\label{bound i}
    (I)\leq \int_{0}^{\infty}\frac{C\epsilon v^{\gamma-\frac{d\alpha}{d+1}-1}\frac{z}{\epsilon}}{1+|\xi|^{d+1}}\der \xi\leq \int_{0}^{\infty}\frac{C\epsilon v^{\gamma-\frac{d\alpha}{d+1}-1}(\xi+1)}{1+|\xi|^{d+1}}\der \xi\leq C\epsilon,
\end{align}
since $v\geq 1$ and $\gamma\leq \frac{d\alpha}{d+1}+1$. Combining \eqref{bound i} with \eqref{bound ii}, we obtain that \eqref{linear ode part 2} holds.
We denote by 
\begin{align*}
    B(s):=\int_{0}^{s}\frac{L\gamma V^{\gamma-1}}{1+|Y|^{d}}\der z.
\end{align*}
Notice that by \eqref{linear ode part 1}, it holds that $0\leq B(s)\leq CL$, for all $s\in[0,t_{1}]$.
From \eqref{der v v}, we obtain that
\begin{align*}
    \partial_{s}\big(e^{B(s)}\partial_{v}V(s)\big)\geq e^{B(s)}\frac{LdV^{\gamma}|Y|^{d-2}Y\partial_{v}Y}{(1+|Y|^{d})^{2}}\geq -e^{B(s)}\bigg|\frac{LdV^{\gamma}|Y|^{d-2}Y\partial_{v}Y}{(1+|Y|^{d})^{2}}\bigg|\geq -C(L)\bigg|\frac{V^{\gamma}|Y|^{d-2}Y\partial_{v}Y}{(1+|Y|^{d})^{2}}\bigg|,
\end{align*}
for some constant $C(L)>0$ which depends on $L$. Making use of \eqref{linear ode part 2}, we further obtain that
\begin{align*}
e^{B(s)}\partial_{v}V(s)-1\geq -C(L)\epsilon
\end{align*}
and since by \eqref{linear ode part 1} it holds that $0\leq B(s)\leq CL\epsilon$, for all $s\in[0,t_{1}]$, then there exists $\epsilon_{1}\in(0,1)$, which depends on $L$, such that for all $\epsilon\leq \epsilon_{1}$, the following holds
\begin{align*}
    \partial_{v}V(s)\geq e^{-B(s)}(1-C(L)\epsilon)\geq e^{-CL\epsilon}(1-C(L)\epsilon)\geq \frac{3}{4},
\end{align*}
for all $s\in[0,t_{1}]$. We can thus iterate the argument in order to obtain that for all times $t\geq 0$, we have that $\partial_{v}V(y,v,t)\geq \frac{1}{2} \geq 0$ and $\partial_{v}Y(y,v,t)\leq \frac{\alpha}{2}v^{\alpha-1}(1-e^{\frac{t}{\epsilon}})\leq 0$. This concludes our proof.
 \end{proof}
 \begin{rmk}\label{bounds on the derivatives remark}
    As an immediate consequence, we have that
    \begin{align*}
        \frac{10^{1-\alpha}\alpha}{9^{1-\alpha}} v^{\alpha-1}(1-e^{\frac{t}{\epsilon}})&\leq \partial_{v}Y(y,v,t)\leq \frac{\alpha v^{\alpha-1}}{2}(1-e^{\frac{t}{\epsilon}});\\
        \frac{1}{2}&\leq \partial_{v}V(y,v,t)\leq 1;\\
    \frac{\alpha v^{\alpha-1}}{2}e^{\frac{t}{\epsilon}}&\leq    \alpha V^{\alpha-1}\partial_{v}{V}-\partial_{v}Y\leq  \frac{10^{1-\alpha}\alpha}{9^{1-\alpha}} v^{\alpha-1}e^{\frac{t}{\epsilon}}.
    \end{align*}
 \end{rmk}
 \begin{prop}\label{bounds derivative prop}
Let $L>0$. Let $G_{\epsilon,L}$ be as in \eqref{equation supersolution} and $\epsilon_{1}\in(0,1)$ and $d>1$ be as in Proposition \ref{der v g negative}. It holds that there exists $K_{2}>0$, which is independent of $\epsilon\leq \epsilon_{1}$,$L>0$, but which can depend on $\alpha,\gamma,b,m$, such that for all $t\geq 0$ and $\epsilon\leq\epsilon_{1}$, we have that
 \begin{align*}
    -\partial_{v}G_{\epsilon,L}(y,w,t)\leq -K_{2} \partial_{v}G_{\epsilon,L}(y,v,t), \textup{ for all } w\in\Big[\frac{v}{2},v\Big] \textup{ and all } y\leq \Big(\frac{v}{3}\Big)^{\alpha}.
 \end{align*}
 \end{prop}
 \begin{proof}
We remember that
     \begin{align*}
           \partial_{v}G_{\epsilon,L}(y,v,t)=-\frac{mA e^{\frac{t}{\epsilon}}(V^{\alpha}-Y)^{m-1}[\alpha V^{\alpha-1}\partial_{v}V-\partial_{v}Y]}{(1+V^{b})(1+(V^{\alpha}-Y)^{m})^{2}}-\frac{b Ae^{\frac{t}{\epsilon}}V^{b-1}\partial_{v}V}{(1+V^{b})^{2}(1+(V^{\alpha}-Y)^{m})},
     \end{align*}
    with $A$ as in \eqref{initial condition decay}. By Proposition \ref{proposition about inequalities} and Remark \ref{remark v alpha y}, we have that
     \begin{align}\label{step 1}
      0\leq \bigg[\Big(\frac{9}{10}\Big)^{\alpha}-\frac{1}{3^{\alpha}}\bigg]e^{\frac{t}{\epsilon}}(v^{\alpha}-y)\leq V^{\alpha}-Y\leq   e^{\frac{t}{\epsilon}}(v^{\alpha}-y).
     \end{align}
     By Remark \ref{bounds on the derivatives remark}, it holds that 
     \begin{align}\label{step 2}
          \frac{\alpha v^{\alpha-1}}{2}e^{\frac{t}{\epsilon}}\leq    \alpha V^{\alpha-1}\partial_{v}{V}-\partial_{v}Y\leq  \frac{10^{1-\alpha}\alpha}{9^{1-\alpha}} v^{\alpha-1}e^{\frac{t}{\epsilon}}.
     \end{align}
     By \eqref{step 1} and \eqref{step 2}, we have that
     \begin{align*}
           \frac{m e^{\frac{t}{\epsilon}}(V^{\alpha}-Y)^{m-1}[\alpha V^{\alpha-1}\partial_{v}V-\partial_{v}Y]}{(1+(V^{\alpha}-Y)^{m})^{2}}\geq C(\alpha,m)\frac{e^{\frac{t(m+1)}{\epsilon}}(v^{\alpha}-y)^{m-1}v^{\alpha-1}}{(1+e^{\frac{tm}{\epsilon}}(v^{\alpha}-y)^{m})^{2}}.
     \end{align*}
     It also holds that
     \begin{align*}
         \frac{b e^{\frac{t}{\epsilon}}V^{b-1}\partial_{v}V}{(1+V^{b})^{2}(1+(V^{\alpha}-Y)^{m})}\geq C(\alpha,b,m)\frac{e^{\frac{t}{\epsilon}} v^{b-1}}{(1+v^{b})^{2}(1+(v^{\alpha}-y)^{m})}.
     \end{align*}
     Thus
     \begin{align*}
        -\partial_{v}G_{\epsilon,L}(y,v,t)\geq C(\alpha,b,m)A\bigg[\frac{e^{\frac{t(m+1)}{\epsilon}}(v^{\alpha}-y)^{m-1}v^{\alpha-1}}{(1+v^{b})(1+e^{\frac{tm}{\epsilon}}(v^{\alpha}-y)^{m})^{2}}+\frac{e^{\frac{t}{\epsilon}} v^{b-1}}{(1+v^{b})^{2}(1+(v^{\alpha}-y)^{m})}\bigg].
     \end{align*}
  We now look at the following system of ODE's.
  \begin{equation}\label{ode system different initial data}
\left\{\begin{aligned}
\partial_{t}{Y}(y,w,t)&=-\frac{1}{\epsilon}({V}^{\alpha}-Y), & \qquad {Y}(y,w,0)&=y\leq (\frac{v}{3})^{\alpha}\,, \\
\partial_{t}{V}(y,w,t)&=-\frac{{L}{V}^{\gamma}}{1+|{Y}|^{d}},& \qquad {V}(y,w,0)&=w\in\Big[\frac{v}{2},v\Big].
   \end{aligned}\right.
   \end{equation}
   We are able to reproduce the computations from Proposition \ref{proposition about inequalities} and Proposition \ref{der v g negative} in order to obtain similar estimates for the solution $(\overline{Y},\overline{V})$ of the system \eqref{ode system different initial data}. This is since $y\leq (\frac{v}{3})^{\alpha}\leq (\frac{2}{3})^{\alpha}w^{\alpha}$. This implies that \eqref{better inequality for Y} becomes
   \begin{align}
      \partial_{s}\overline{Y}\leq \frac{w^{\alpha}}{\epsilon}\bigg[\frac{2^{\alpha}}{3^{\alpha}}-\Big(\frac{9}{10}\Big)^{\alpha}\bigg]\leq -\frac{\tilde{c} w^{\alpha}}{\epsilon}.
  \end{align}
  The rest of the proof follows similarly, but some constants need to be modified. We obtain that
  \begin{align*}
        \frac{9}{10} w &\leq  \overline{V}(y,w,t)\leq w;\\
     \overline{Y}(y,w,t)&\leq \frac{2^{\alpha}}{3^{\alpha}}w^{\alpha};\\
e^{\frac{t}{\epsilon}} (y-w^{\alpha}) & \leq   \overline{Y}-\overline{V}^{\alpha};\\
 \overline{Y}-\overline{V}^{\alpha}&\leq \bigg[\Big(\frac{9}{10}\Big)^{\alpha}-\frac{2^{\alpha}}{3^{\alpha}}\bigg] e^{\frac{t}{\epsilon}} (y-w^{\alpha}).
  \end{align*}
  The bounds for $\partial_{v}\overline{Y}(y,w,t)$ and $\partial_{v}\overline{V}(y,w,t)$ do not change. We thus obtain that
      \begin{align*}
           -\partial_{v}G_{\epsilon,L}(y,w,t)\leq C(\alpha,b,m)A\bigg[\frac{e^{\frac{t(m+1)}{\epsilon}}(w^{\alpha}-y)^{m-1}w^{\alpha-1}}{(1+w^{b})(1+e^{\frac{tm}{\epsilon}}(w^{\alpha}-y)^{m})^{2}}+\frac{e^{\frac{t}{\epsilon}}w^{b-1}}{(1+w^{b})^{2}(1+e^{\frac{tm}{\epsilon}}(w^{\alpha}-y)^{m})}\bigg].
     \end{align*}
     If $y\leq 0$, we have that
 \begin{align*}
  v^{\alpha}-y\geq w^{\alpha}-y\geq     \Big(\frac{v}{2}\Big)^{\alpha}-y\geq \frac{v^{\alpha}-y}{2^{\alpha}}. 
 \end{align*}
 If $y\in[0,(\frac{v}{3})^{\alpha}]$, then
  \begin{align*}
  v^{\alpha}-y\geq w^{\alpha}-y\geq     \Big(\frac{v}{2}\Big)^{\alpha}-y\geq \frac{3^{\alpha}-2^{\alpha}}{6^{\alpha}}v^{\alpha}\geq  \frac{3^{\alpha}-2^{\alpha}}{6^{\alpha}}(v^{\alpha}-y).
 \end{align*}
 Concluding, we have that
 \begin{align*}
            \frac{e^{\frac{t(m+1)}{\epsilon}}(w^{\alpha}-y)^{m-1}w^{\alpha-1}}{(1+e^{\frac{tm}{\epsilon}}(w^{\alpha}-y)^{m})^{2}}\leq C(\alpha,b,m) \frac{e^{\frac{t(m+1)}{\epsilon}}(v^{\alpha}-y)^{m-1}v^{\alpha-1}}{(1+e^{\frac{tm}{\epsilon}}(v^{\alpha}-y)^{m})^{2}}
 \end{align*}
 and further that
 \begin{align*}
     -\partial_{v}G_{\epsilon,L}(y,w,t)\leq  -K_{2}\partial_{v}G_{\epsilon,L}(y,v,t).
 \end{align*}
 \end{proof}
 \subsection{Proof of existence of a supersolution in the region $y\leq (\frac{v}{3})^{\alpha}$}\label{subsection approximation}
 We are now in a position to prove the main statement of this section.

 \begin{proof}[Proof of Proposition \ref{main statement region below}]
Assume  $H_{n}(y,v,t)\leq  T_{1}(y,v,t)+ T_{2}(y,v,t)+ T_{3}(y,v,t)$. Let $C_{\gamma}:=1+\frac{1}{2^{\gamma}}$, $K_{0}$ as in \eqref{upper bound coagulation kernel}, $K_{1}$ as in Proposition \ref{moment estimates prop}, $K_{2}>0$ be as in Proposition \ref{bounds derivative prop}, and $A$ as in \eqref{initial condition decay}. Let 
\begin{align}\label{defk3}
    G_{\epsilon}:=e^{K_{3}(\alpha,b,m)L t}G_{\epsilon,L}, \textup{ where } G_{\epsilon,L} \textup{ solves } \eqref{equation supersolution} 
\end{align}with 
\begin{align}\label{definition l}
L=C_{\gamma}K_{0}K_{1}K_{2}A^{3}.
\end{align}Moreover, $d$ is chosen as in \eqref{we need function to be even} such that $d>1$, even, and sufficiently large such that $\gamma\leq \frac{d\alpha}{d+1}+1$, and $K_{3}(\alpha,b,m)$ is a fixed constant depending on the parameters $\alpha, b,m$ which will be determined later. Notice that $d$ depends on $b$ and  the minimum value of $b$ for which these conditions are satisfied determines the value of $\overline{b}(\gamma,\alpha)$ in \eqref{definition overline b} as noticed in Remark \ref{we need function to be even minimum recovered}.

We prove Proposition \ref{main statement region below} by first proving that $H_{n+1}(y,v,t)\leq G_{\epsilon}(y,v,t),$ for all $t\in[0,T]$ and $y\leq (\frac{v}{3})^{\alpha}$, $v\in(0,\infty),$ and then that  $G_{\epsilon}(y,v,t)\leq T_{3}(y,v,t)$ for all $t\in[0,T]$ and $y\leq (\frac{v}{3})^{\alpha}$, $v\in(0,\infty)$.

Let $n\in\mathbb{N}$. We prove that $G_{\epsilon}$ is a supersolution for \eqref{iteration supersol}. In other words, we prove that
\begin{align}\label{supersol main ineq}
    \partial_{t}G_{\epsilon}(y,v,t)+\frac{1}{\epsilon}\partial_{y}[(v^{\alpha}-y)G_{\epsilon}(y,v,t)]-&\int_{0}^{\frac{v}{2}}K(v-w,w)G_{\epsilon}(y,v-w,t)H_{n}(y,w,t)\der w \nonumber\\
   &+\int_{0}^{\infty}K(v,w)G_{\epsilon}(y,v,t)H_{n}(y,w,t)\der w\geq 0.
   \end{align}
Since $K(v-w,w)\leq K(v,w),$ for all $w\leq v$, it follows that
\begin{align*}
    \partial_{t}G_{\epsilon}(y,v,t)+\frac{1}{\epsilon}\partial_{y}[(v^{\alpha}-y)G_{\epsilon}(y,v,t)]-&\int_{0}^{\frac{v}{2}}K(v-w,w)G_{\epsilon}(y,v-w,t)H_{n}(y,w,t)\der w \\
   &+\int_{0}^{\infty}K(v,w)G_{\epsilon}(y,v,t)H_{n}(y,w,t)\der w \\
   \geq   \partial_{t}G_{\epsilon}(y,v,t)+\frac{1}{\epsilon}\partial_{y}[(v^{\alpha}-y)G_{\epsilon}(y,v,t)]-&\int_{0}^{\frac{v}{2}}K(v-w,w)G_{\epsilon}(y,v-w,t)H_{n}(y,w,t)\der w \\
   &+\int_{0}^{\frac{v}{2}}K(v,w)G_{\epsilon}(y,v,t)H_{n}(y,w,t)\der w\\
  \geq \partial_{t}G_{\epsilon}(y,v,t)+\frac{1}{\epsilon}\partial_{y}[(v^{\alpha}-y)G_{\epsilon}(y,v,t)]-&\int_{0}^{\frac{v}{2}}K(v,w)G_{\epsilon}(y,v-w,t)H_{n}(y,w,t)\der w \\
   &+\int_{0}^{\frac{v}{2}}K(v,w)G_{\epsilon}(y,v,t)H_{n}(y,w,t)\der w
   \\
= \partial_{t}G_{\epsilon}(y,v,t)+\frac{1}{\epsilon}\partial_{y}[(v^{\alpha}-y)G_{\epsilon}(y,v,t)]-&\int_{0}^{\frac{v}{2}}K(v,w)[G_{\epsilon}(y,v-w,t)-G_{\epsilon}(y,v,t)]H_{n}(y,w,t)\der w 
  \\
= \partial_{t}G_{\epsilon}(y,v,t)+\frac{1}{\epsilon}\partial_{y}[(v^{\alpha}-y)G_{\epsilon}(y,v,t)]+&\int_{0}^{\frac{v}{2}}K(v,w)\int_{v-w}^{v}\partial_{v}G_{\epsilon}(y,z,t)\der z H_{n}(y,w,t)\der w =: (\ast).
\end{align*}

From Proposition \ref{proposition about inequalities}, Proposition \ref{der v g negative}, and Proposition \ref{bounds derivative prop},   we have that $G_{\epsilon}$ has the following properties.
\begin{itemize}
    \item $\partial_{v}G_{\epsilon}(y,z,t)\geq     K_{2}\partial_{v}G_{\epsilon}(y,v,t)$, for all $z\in[v-w,v],$ $w\in[0,\frac{v}{2}];$
      \item $\partial_{v}G_{\epsilon}(y,v,t)\leq 0$, for all $y\leq (\frac{v}{3})^{\alpha}$.
\end{itemize} 

Since also $K(v,w)\leq C_{\gamma}K_{0}v^{\gamma}$ by \eqref{upper bound coagulation kernel} when $w\in[0,\frac{v}{2}]$, we obtain that
\begin{align}\label{supersol formal}
    (\ast)\geq\partial_{t}G_{\epsilon}(y,v,t)+\frac{1}{\epsilon}\partial_{y}[(v^{\alpha}-y)G_{\epsilon}(y,v,t)]+&C_{\gamma}K_{0}K_{2}v^{\gamma}\partial_{v}G_{\epsilon}(y,v,t)\int_{0}^{\frac{v}{2}}w H_{n}(y,w,t)\der w.
\end{align}

Making use of Proposition \ref{moment estimates prop}, the non-positivity of $\partial_{v}G_{\epsilon}$ in \eqref{supersol formal}, and since by assumption $H_{n}\leq T_{1}+T_{2}+T_{3}$, we obtain that
\begin{align*}
    (\ast)&\geq\partial_{t}G_{\epsilon}(y,v,t)+\frac{1}{\epsilon}\partial_{y}[(v^{\alpha}-y)G_{\epsilon}(y,v,t)]\\
    &+C_{\gamma}K_{0}K_{2}v^{\gamma}\partial_{v}G_{\epsilon}(y,v,t)\int_{0}^{\frac{v}{2}}w [T_{1}+T_{2}+T_{3}](y,w,t)\der w\\
    &\geq \partial_{t}G_{\epsilon}(y,v,t)+\frac{1}{\epsilon}\partial_{y}[(v^{\alpha}-y)G_{\epsilon}(y,v,t)]+\frac{C_{\gamma}K_{0}K_{1}K_{2}A^{3}v^{\gamma}}{1+|y|^{\frac{b-2}{\alpha}}}\partial_{v}G_{\epsilon}(y,v,t)\\
    &\geq\partial_{t}G_{\epsilon}(y,v,t)+\frac{1}{\epsilon}(v^{\alpha}-y)\partial_{y}G_{\epsilon}(y,v,t)+\frac{Lv^{\gamma}}{1+|y|^{d}}\partial_{v}G_{\epsilon}(y,v,t)-\frac{1}{\epsilon}G_{\epsilon}(y,v,t),
\end{align*}
where in the last line we used that $\frac{b-2}{\alpha}\geq d$. Notice that this is equation \eqref{equation supersolution} when $v\geq 1$ and thus  $G_{\epsilon}$ is a supersolution for \eqref{iteration supersol}. Thus
\begin{align*}
    (\ast)
    &\geq \partial_{t}G_{\epsilon}(y,v,t)+\frac{1}{\epsilon}(v^{\alpha}-y)\partial_{y}G_{\epsilon}(y,v,t)+\frac{L\xi_{R}(v)v^{\gamma}}{1+|y|^{d}}\partial_{v}G_{\epsilon}(y,v,t)-\frac{1}{\epsilon}G_{\epsilon}(y,v,t)\\
    &-\bigg|\frac{L[1-\xi_{R}](v)v^{\gamma}}{1+|y|^{d}}\partial_{v}G_{\epsilon}(y,v,t)\bigg|\\
     &\geq \partial_{t}G_{\epsilon}(y,v,t)+\frac{1}{\epsilon}(v^{\alpha}-y)\partial_{y}G_{\epsilon}(y,v,t)+\frac{L\xi_{R}(v)v^{\gamma}}{1+|y|^{d}}\partial_{v}G_{\epsilon}(y,v,t)-\frac{1}{\epsilon}G_{\epsilon}(y,v,t)\\
    &-\bigg|\frac{L\mathbbm{1}_{[0,1]}(v)v^{\gamma}}{1+|y|^{d}}\partial_{v}G_{\epsilon}(y,v,t)\bigg|.
\end{align*}
Assume that 
\begin{align}\label{bound derivative by function}
    \bigg|\frac{L\mathbbm{1}_{[0,1]}(v)v^{\gamma}}{1+|y|^{d}}\partial_{v}G_{\epsilon}(y,v,t)\bigg|\leq K_{3}(\alpha,b,m)L G_{\epsilon}(y,v,t).
\end{align}
Then 
\begin{align*}
    (\ast) \geq &\partial_{t}G_{\epsilon}(y,v,t)+\frac{1}{\epsilon}(v^{\alpha}-y)\partial_{y}G_{\epsilon}(y,v,t)+\frac{L\xi_{R}(v)v^{\gamma}}{1+|y|^{d}}\partial_{v}G_{\epsilon}(y,v,t)-\frac{1}{\epsilon}G_{\epsilon}(y,v,t)\\
    &-K_{3}(\alpha,b,m)LG_{\epsilon}(y,v,t).
\end{align*}
By \eqref{defk3} , it holds that $\partial_{t}G_{\epsilon}(y,v,t)=e^{K_{3}(\alpha,b,m)Lt}\partial_{t}G_{\epsilon,L}(y,v,t)+K_{3}(\alpha,b,m)LG_{\epsilon}(y,v,t)$. Thus
\begin{align*}
    (\ast) \geq &e^{K_{3}Lt}\bigg[\partial_{t}G_{\epsilon,L}(y,v,t)+\frac{1}{\epsilon}(v^{\alpha}-y)\partial_{y}G_{\epsilon,L}(y,v,t)+\frac{L\xi_{R}(v)v^{\gamma}}{1+|y|^{d}}\partial_{v}G_{\epsilon,L}(y,v,t)\\
    &-\frac{1}{\epsilon}G_{\epsilon,L}(y,v,t)\bigg]= 0.
\end{align*}
This concludes our proof. 

It remains to prove that \eqref{bound derivative by function} holds. We have proved in Proposition \ref{bounds derivative prop} that
  \begin{align*}
           -\partial_{v}G_{\epsilon,L}(y,v,t)\leq C(m,\alpha,b)A\bigg[\frac{e^{\frac{t(m+1)}{\epsilon}}(v^{\alpha}-y)^{m-1}v^{\alpha-1}}{(1+v^{b})(1+e^{\frac{tm}{\epsilon}}(v^{\alpha}-y)^{m})^{2}}+\frac{e^{\frac{t}{\epsilon}}v^{b-1}}{(1+v^{b})^{2}(1+e^{\frac{tm}{\epsilon}}(v^{\alpha}-y)^{m})}\bigg].
     \end{align*}
Thus
\begin{align*}
     \bigg|\frac{L\mathbbm{1}_{[0,1]}(v)v^{\gamma}}{1+|y|^{d}}\partial_{v}G_{\epsilon}(y,v,t)\bigg|\leq C(\alpha,b,m)Lv^{\gamma}\mathbbm{1}_{[0,1]}(v) G_{\epsilon}(y,v,t)\bigg[\frac{e^{\frac{tm}{\epsilon}}(v^{\alpha}-y)^{m-1}v^{\alpha-1}}{1+e^{\frac{tm}{\epsilon}}(v^{\alpha}-y)^{m}}+\frac{v^{b-1}}{1+v^{b}}\bigg].
\end{align*}
Since $y\leq (\frac{v}{3})^{\alpha}$ and thus $v^{\alpha}-y\geq C(\alpha) v^{\alpha}$, we further have that
\begin{align}
     \bigg|\frac{L\mathbbm{1}_{[0,1]}(v)v^{\gamma}}{1+|y|^{d}}\partial_{v}G_{\epsilon}(y,v,t)\bigg|&\leq C(\alpha,b,m)L v^{\gamma-1}\mathbbm{1}_{[0,1]}(v) G_{\epsilon}(y,v,t)\bigg[\frac{e^{\frac{tm}{\epsilon}}(v^{\alpha}-y)^{m}}{1+e^{\frac{tm}{\epsilon}}(v^{\alpha}-y)^{m}}+\frac{v^{b}}{1+v^{b}}\bigg]\nonumber\\
     &\leq  C(\alpha,b,m)L v^{\gamma-1}\mathbbm{1}_{[0,1]}(v) G_{\epsilon}(y,v,t) \leq C(\alpha,b,m)L G_{\epsilon}(y,v,t) ,\label{def m2 supersol}
\end{align}
since $\gamma>1$ and thus \eqref{bound derivative by function} holds if we choose $K_{3}$ as above. Thus, we have that $H_{n+1}\leq G_{\epsilon}$ when $y\leq (\frac{v}{3})^{\alpha}$.

We now prove that  $G_{\epsilon}(y,v,t)\leq T_{3}(y,v,t)$ for all $t\in[0,T]$ and $y\leq (\frac{v}{3})^{\alpha}$, $v\in(0,\infty)$. Let $y\leq (\frac{v}{3})^{\alpha}$. If 
\begin{align}\label{definition time supersol}
    t\leq \frac{1}{K_{3}(\alpha,b,m)L },
\end{align} where $K_{3}(\alpha,b,m)$ is as in \eqref{def m2 supersol} and $L$ as in \eqref{definition l}, then $e^{K_{3}(\alpha,b,m)L t}\leq 3$.  From Proposition \ref{proposition about inequalities}, we have that
\begin{align}
  (1+v^{b})  G_{\epsilon}(y,v,t)&\leq \frac{C(\alpha,b,m)e^{K_{3}(\alpha,b,m)L t}Ae^{\frac{t}{\epsilon}}}{1+e^{\frac{tm}{\epsilon}}|v^{\alpha}-y|^{m}}\leq \frac{3 C(\alpha,b,m)A e^{\frac{t}{\epsilon}}}{1+e^{\frac{tm}{\epsilon}}|v^{\alpha}-y|^{m}}\nonumber\\
  &= \frac{M_{2}Ae^{\frac{t}{\epsilon}}}{1+e^{\frac{tm}{\epsilon}}|v^{\alpha}-y|^{m}}=T_{3}(y,v,t)\label{def m2}
\end{align}
by the definition of $T_{3}$ is \eqref{def t3}. We also remark that the constant $M_{2}$ in the above inequality depends only on the choice of parameters $\alpha,  b,m$ since it appears due to the bounds found in Proposition \ref{proposition about inequalities}  and does not depend on $n$. This concludes our proof.
\end{proof}

\section{Proof of Theorem \ref{proposition existence via fixed point} - Existence}\label{section five}
This section is dedicated to the proof of Theorem 
\ref{proposition existence via fixed point}. We remember the following.

Let   $b\geq \max\{\overline{b}(\gamma,\alpha),\gamma+1\}$ and $m>\max\{\frac{2(\gamma+1)}{\alpha},\frac{b}{\alpha}+1\}$, where  $\overline{b}(\gamma,\alpha)$ is as in \eqref{definition overline b}. By Proposition \ref{supersol prop above region},  there exists $\epsilon_{1}\in(0,1)$ and $t\in[0,1]$ sufficiently small such that for all $\epsilon\leq \epsilon_{1}$, such that $H_{n}(y,v,s)\leq T_{1}(s)+T_{2}(s),$ for all $s\in[0,t]$, $y\geq (\frac{v}{3})^{\alpha}$, and $n\in\mathbb{N}$. By Proposition \ref{main statement region below}, there exist $\epsilon_{1}\in(0,1)$ and $T\leq 1$ sufficiently small such that for all $\epsilon\leq \epsilon_{1}$, $H_{n}(y,v,t)\leq T_{3}(y,v,t),$ for all $t\in[0,T]$, $y\leq (\frac{v}{3})^{\alpha}$, $v\in(0,\infty)$, and for all $n\in\mathbb{N}$. In other words, we have that
\begin{align}
H_{n}(y,v,t)\leq T_{1}(y,v,t)+T_{2}(y,v,t)+T_{3}(y,v,t), \textup{ for all } n\in\mathbb{N},
\end{align}
for all $t\in[0,T]$, with $T\leq 1$, sufficiently small, and all $y\in\mathbb{R}$, $v\in(0,\infty)$. We will prove that

\vspace{0.1cm}

\textbf{\underline{Claim:}} Assume in addition that $b>2\gamma+1$. For all $n\in\mathbb{N}$
\begin{align}\label{sequence is cauchy}
\sup_{t\in[0,T], y\in\mathbb{R}, v\in(0,\infty)}|H_{n+1}(y,v,t)-H_{n}(y,v,t)|\leq (CT)^{n}\rightarrow 0
 \textup{ as } n\rightarrow 0 \textup{ for } T< \frac{1}{2C}.
 \end{align}
 A similar strategy can then be employed in order to prove that for all $n,\tilde{n}\in\mathbb{N}$, $\tilde{n}\geq n+1$, it holds that
 \begin{align*}
\sup_{t\in[0,T], y\in\mathbb{R}, v\in(0,\infty)}|H_{\tilde{n}}(y,v,t)-H_{n}(y,v,t)|\leq (CT)^{n}\rightarrow 0
 \textup{ as } n\rightarrow 0 \textup{ for } T< \frac{1}{2C}.
 \end{align*}
 Assuming that \eqref{sequence is cauchy} is true, we have that for fixed $\epsilon$ there exists a limit $H_{\epsilon}(y,v,t)\in \textup{C}([0,T];L^{\infty}(\mathbb{R}\times(0,\infty))$ such that $H_{n}\rightarrow H_{\epsilon}$ as $n\rightarrow \infty$ and moreover
 \begin{align*}
H_{\epsilon}(y,v,t)\leq T_{1}(y,v,t)+T_{2}(y,v,t)+T_{3}(y,v,t).
 \end{align*}
 Thus, assuming \eqref{sequence is cauchy} holds, what is left to prove is that $H_{\epsilon}$ solves \eqref{mild solution equation}. We have that $H_{n}$ satisfies
\begin{align*}
    &H_{n+1}(y,v,t)-S_{\epsilon}(t)[H_{n+1}(y,v,0) ] D[H_{n}](y,v,0,t)=\\
&\frac{1}{2}\int_{0}^{t}\int_{(0,v)}D[H_{n}](y,v,s,t)
S_{\epsilon}(t-s)\bigg[K(v-v',v')H_{n+1}(y,v-w,s) H_{n}(y,w,s)\bigg]\der s
\end{align*}
and wish to pass to the limit in the equation. Notice that
\begin{align}\label{bound hn new}
     H_{n}(e^{\frac{t-s}{\epsilon}}(y-v^{\alpha})+v^{\alpha},w,s)
      \leq [T_{1}+T_{2}+T_{3}](e^{\frac{t-s}{\epsilon}}(y-v^{\alpha})+v^{\alpha},w,s),
    \end{align}
where $T_{1},T_{2}, T_{3}$ are as in \eqref{def t1}-\eqref{def t3}.    Denote by
    \begin{align*}
        X_{\epsilon}(y,v,t,s):=e^{\frac{t-s}{\epsilon}}(y-v^{\alpha})+v^{\alpha}.
    \end{align*}
    We can then replace $y$ by $X_{\epsilon}$ in the proof of Proposition \ref{moment estimates prop} in order to obtain the following.
\begin{prop}[Moment estimates]\label{moment estimates limit part}

 Let $b>\gamma+1$ and $m\geq \frac{b}{\alpha}+1$ in \eqref{def t1}-\eqref{def t3}. Let $T\leq 1$ and $H_{n}$ be as in \eqref{bound hn new}. It then holds that
\begin{align}
    \int_{(0,\infty)}& w^{k}H_{n}(X_{\epsilon}(y,v,t,s),w,s)\der w\leq \frac{K_{1}A^{3}}{1+|X_{\epsilon}|^{d}}, \textup{ for some } d\in\mathbb{N}
\end{align}
and all $k\in[0,\gamma+1]$, $t\in[0,T], s\in[0,t], y\in\mathbb{R}, v\in(0,\infty).$
\end{prop}
\begin{prop}
$H_{\epsilon}$ solves \eqref{mild solution equation}.
\end{prop}
\begin{proof}
We prove first that 
\begin{align*}
D[H_{n}](y,v,s,t)\rightarrow D[H_{\epsilon}](y,v,s,t).
\end{align*}
Using $|e^{-x}-e^{-y}|\leq |x-y|, $ for $x,y\geq 0$, we have that
\begin{align*}
\bigg| & \textup{e}^{-\int_{s}^{t} a[H_{n}](e^{\frac{t-\tau}{\epsilon}}(y-v^{\alpha})+v^{\alpha},v,\tau)\der\tau}-\textup{e}^{-\int_{s}^{t}a[H_{\epsilon}](e^{\frac{t-\tau}{\epsilon}}(y-v^{\alpha})+v^{\alpha},v,\tau)\der\tau}\bigg|\\
&\leq \bigg|\int_{s}^{t}a[H_{n}](e^{\frac{t-\tau}{\epsilon}}(y-v^{\alpha})+v^{\alpha},v,\tau)\der\tau-\int_{s}^{t}a[H_{\epsilon}](e^{\frac{t-\tau}{\epsilon}}(y-v^{\alpha})+v^{\alpha},v,\tau)\der\tau\bigg|,
\end{align*}
where $a[H_{\epsilon}]$ is as in \eqref{general a f def}. However
\begin{align*}
    a[H_{n}](e^{\frac{t-\tau}{\epsilon}}(y-v^{\alpha})+v^{\alpha},v,\tau)= \int_{(0,\infty)}K(v,w)H_{n}(X_{\epsilon}(y,v,t,\tau),w,\tau)\der w.
\end{align*}
Let $M>1$. Since $H_{n}\rightarrow H_{\epsilon}$ as $\epsilon\rightarrow 0,$ we have that
\begin{align*}
    \int_{(0,M)}K(v,w)H_{n}(X_{\epsilon}(y,v,t,\tau),w,\tau)\der w \rightarrow \int_{(0,M)}K(v,w)H_{\epsilon}(X_{\epsilon}(y,v,t,\tau),w,\tau)\der w
\end{align*}
as $n\rightarrow\infty$, while by Proposition \ref{moment estimates limit part}, it follows that

\begin{align*}
    \int_{(M,\infty)}K(v,w)H_{n}(X_{\epsilon}(y,v,t,\tau),w,\tau)\der w &\leq  M^{-1}\int (v^{\gamma}w+w^{\gamma+1})H_{n}(X_{\epsilon}(y,v,t,\tau),w,\tau)\der w \\
    & \leq C M^{-1}(1+v^{\gamma})\rightarrow 0
 \textup{ as } M\rightarrow\infty.\end{align*}
 The rest of the proof follows the same idea. For example, we have that
 \begin{align*}
\bigg|\int_{0}^{t}\int_{(0,v)}
&S_{\epsilon}(t-s)\bigg[K(v-w,w)H_{n+1}(y,v-w,s) H_{n}(y,w,s)\bigg]\der s\\
&-
\int_{0}^{t}\int_{(0,v)}
S_{\epsilon}(t-s)\bigg[K(v-w,w)H_{\epsilon}(y,v-w,s) H_{\epsilon}(y,w,s)\bigg]\der s\bigg|\\
&\leq \int_{0}^{t}\int_{(0,v)}S_{\epsilon}(t-s)\bigg[K(v-w,w)|H_{n+1}-H_{\epsilon}|(y,v-w,s) H_{n}(y,w,s)\bigg]\der s\\
&+\int_{0}^{t}\int_{(0,v)}
S_{\epsilon}(t-s)\bigg[K(v-w,w)H_{\epsilon}(y,v-w,s)|H_{n}- H_{\epsilon}|(y,w,s)\bigg]\der s.
\end{align*}
We can use the convergence of $H_{n}$ to $H_{\epsilon}$ together with the moment estimates from Proposition \ref{moment estimates limit part}  in order to conclude as above. More precisely, by Proposition \ref{moment estimates limit part}, we have
\begin{align*}
& \int_{0}^{t}\int_{(0,v)}S_{\epsilon}(t-s)\bigg[K(v-w,w)|H_{n+1}-H_{\epsilon}|(y,v-w,s) H_{n}(y,w,s)\bigg]\der s\\
   & = \int_{0}^{t}\int_{(0,v)}e^{\frac{t-s}{\epsilon}}K(v-w,w)|H_{n+1}-H_{\epsilon}|(X_{\epsilon}(y,v,t,s),v-w,s) H_{n}(X_{\epsilon}(y,v,t,s),w,s)\der w\der s\\
   & \leq C(\epsilon,v) ||H_{n+1}-H_{\epsilon}||_{L^{\infty}}\int_{0}^{t}\int_{(0,v)} H_{n}(X_{\epsilon}(y,v,t,s),w,s)\der w\der s\\
     & \leq C(\epsilon,v) ||H_{n+1}-H_{\epsilon}||_{L^{\infty}}\rightarrow 0 \textup{ as } n\rightarrow\infty.
\end{align*}

The rest of the steps are done in the same manner using in addition that
\begin{align*}
(1+v^{b})H_{\epsilon}(y,v,t)\leq Ce^{\frac{t}{\epsilon}}.
\end{align*}
 This concludes our proof.
\end{proof}

 The rest of this subsection is dedicated to proving that \eqref{sequence is cauchy} is true.

\begin{proof}[Proof of \eqref{sequence is cauchy}]

{\bf Step 1: (Set-up)}
Let $t\geq 0$, $y\in\mathbb{R}$ and $v> 0$. For $n\in\mathbb{N}$, we denote by 
\begin{align}\label{def rn}
R_n(y,v,t):=H_{n+1}(y,v,t)-H_n(y,v,t).
\end{align}
For two functions $f,g$, we denote by
\begin{align}
\mathcal{C}_{1}[f,g]&:=\int_{0}^{t}\int_{0}^{\frac{v}{2}}e^{\frac{t-s}{\epsilon}}K(v-w,w)f(e^{\frac{t-s}{\epsilon}}(y-v^{\alpha})+v^{\alpha},w,s)g(e^{\frac{t-s}{\epsilon}}(y-v^{\alpha})+v^{\alpha},v-w,s)\der w \der s,\label{k1 term}\\
\mathcal{K}_{1}[f,g]&:=\int_{0}^{\frac{v}{2}}K(v-w,w)f(y,w,t)g(y,v-w,t)\der w,\label{k1 term 2}\\
\mathcal{C}_{2}[f,g]&:=\int_{0}^{t}\int_{0}^{\infty}e^{\frac{t-s}{\epsilon}}K(v,w)f(e^{\frac{t-s}{\epsilon}}(y-v^{\alpha})+v^{\alpha},w,s)g(e^{\frac{t-s}{\epsilon}}(y-v^{\alpha})+v^{\alpha},v,s)\der w\der s,\label{k2 term}\\
\mathcal{K}_{2}[f,g]&:=\int_{0}^{\infty}K(v,w)f(y,w,t)g(y,v,t)\der w,\label{k2 term 2}
\end{align}
and
\begin{align}
\mathcal{C}[f,g]:=\mathcal{C}_{1}[f,g]-\mathcal{C}_{2}[f,g],\label{full k term}\\
\mathcal{K}[f,g]:=\mathcal{K}_{1}[f,g]-\mathcal{K}_{2}[f,g].\label{full k term 2}
\end{align}
Using this notation, it holds that
\begin{align}\label{induction step with distance one}
 R_n(y,v,t)&=\mathcal{C}[H_{n},R_n]+\mathcal{C}[R_{n-1},H_{n}].
\end{align}

We notice that $\mathcal{C}[H_{n},R_n]$ is linear in $H_{n}$, while $\mathcal{C}[R_{n-1},H_{n}]$ does not depend on $H_{n}$. We can thus use Duhamel's principle in order to deduce that in order to find estimates for $R_{n}$ which solves
\begin{equation}\label{duhamel 1}
\left\{\begin{aligned}
\partial_t R_n(y,v,t)+\frac{1}{\epsilon}\partial_y[(v^{\alpha}-y) R_n(y,v,t)]&=\mathcal{K}[H_{n},R_n]+\mathcal{K}[R_{n-1},H_{n}]; \\
R_{n}(y,v,0)&=0
   \end{aligned}\right.
   \end{equation}
   it is sufficient to find estimates for the system
    \begin{equation}\label{duhamel 2}
\left\{\begin{aligned}
\partial_t R^{s}_n(x,v,t)+\frac{1}{\epsilon}\partial_y[(v^{\alpha}-y) R_n(y,v,t)]&=\mathcal{K}[H_{n},R_n], \textup{ for } t>s; \\
R^{s}_{n}(y,v,s)&=\mathcal{K}[R_{n-1},H_{n}].
   \end{aligned}\right.
   \end{equation}

{\bf Step 2: (Bound for inhomogeneous term)}

We wish to bound $\mathcal{C}[R_{n-1},H_{n}]$ by induction. We only  bound from above $\mathcal{C}_{2}[R_{n-1},H_{n}]$ as $\mathcal{C}_{1}[R_{n-1},H_{n}]$ is bounded using similar computations. 
\vspace{0.1cm}

\textit{\underline{Induction basis:}} By Proposition \ref{moment estimates limit part},  \eqref{h0 form}-\eqref{h0 pde}, and since  $(1+v^{b})|H_{n}(y,v,t)|\leq Ce^{\frac{t}{\epsilon}}$, for all $n\in\mathbb{N}$, we have that
\begin{align*}
\mathcal{C}_{2}[H_{0},H_{1}]\leq &(1+v^{\gamma})\int_{0}^{t}e^{\frac{t-s}{\epsilon}}H_{1}(X_{\epsilon}(y,v,t,s),v,s)\int_{0}^{\infty}(1+w^{\gamma}) |H_{0}|(X_{\epsilon}(y,v,t,s),w,s)\der w\\
&\leq  (1+v^{\gamma})\int_{0}^{t} e^{\frac{t-s}{\epsilon}}H_{1}(X_{\epsilon}(y,v,t,s),v,s)\frac{K_{1}A^{3}}{1+|X|^{d}}\leq C t e^{\frac{t}{\epsilon}}\leq C t e^{\frac{1}{\epsilon}},
\end{align*}
where $C$ does not depend on $n$ or $\epsilon$.

\vspace{0.1cm}

\textit{\underline{Induction step:}} Assume by induction that 
\begin{align*}
|R_{n-1}|(y,v,t)\leq (Cte^{\frac{1}{\epsilon}})^{n-1}\frac{1+v^{\gamma}}{1+v^{b}},
\end{align*}
 where $C$ does not depend on $n$ or on $\epsilon$. Then, since $(1+v^{b})|H_{n}(y,v,t)|\leq Ce^{\frac{t}{\epsilon}}$ and since $b>2\gamma+1$, it holds that
\begin{align}\label{same estimates iteration}
\mathcal{C}_{2}[R_{n-1},H_{n}]\leq &(1+v^{\gamma})\int_{0}^{t}e^{\frac{t-s}{\epsilon}}H_{n}(X_{\epsilon}(y,v,t,s),v,s)\int_{0}^{\infty}(1+w^{\gamma}) |R_{n-1}|(X_{\epsilon}(y,v,t,s),w,s)\der w\nonumber\\
&\leq (1+v^{\gamma})\int_{0}^{t}e^{\frac{t-s}{\epsilon}}H_{n}(X_{\epsilon}(y,v,t,s),v,s)\int_{0}^{\infty}\frac{1+w^{2\gamma}}{1+w^{b}} \der w(Cse^{\frac{1}{\epsilon}})^{n-1}\der s\nonumber\\
&\leq  \int_{0}^{t}e^{\frac{t-s}{\epsilon}}(1+v^{\gamma})H_{n}(X_{\epsilon}(y,v,t,s),v,s)(Cse^{\frac{1}{\epsilon}})^{n-1}\der s\nonumber \\
&\leq C^{n} t^{n}e^{\frac{n-1}{\epsilon}} e^{\frac{t}{\epsilon}}\frac{1+v^{\gamma}}{1+v^{b}}\leq (Cte^{\frac{1}{\epsilon}})^{n}\frac{1+v^{\gamma}}{1+v^{b}}.
\end{align}

\vspace{0.2cm}

{\bf Step 3: (Iteration in time in strong formulation)}
By Step $2$, we deduce that for every $n\in\mathbb{N}$, it holds that
\begin{align}\label{upper bound rn}
    R_{n}(y,v,t)\leq (Cte^{\frac{1}{\epsilon}})^{n}.
\end{align}
Thus, for $t< \frac{e^{-\frac{1}{\epsilon}}}{C}$, we have that
\begin{align}\label{sequence rn goes to zero}
      R_{n}(y,v,t)\rightarrow 0 \textup{ as } n \rightarrow \infty.
\end{align}
A similar argument has been used in \cite{cristianinhom}.

We now wish to extend \eqref{sequence rn goes to zero} to hold for times of order $\mathcal{O}(1)$. Let us assume for simplicity of notation that $C=1$ in \eqref{upper bound rn}. The idea is to iterate the argument in Step $2.$ to hold for longer times. For simplicity, we first show the argument using the strong form of 
\eqref{mild form iteration in n}. In order to obtain a rigorous proof, the same idea needs to be applied directly to the mild formulation \eqref{mild form iteration in n}. This will be done in Step $4.$

\vspace{0.2cm}

\textit{\underline{Translation in time:}} We first do a translation in time. For $t\in [0,\frac{e^{-\frac{1}{\epsilon}}}{2}]$, define 
\begin{align*}
 \tilde{R}_{n}(t):=R_{n}\bigg(t+\frac{e^{-\frac{1}{\epsilon}}}{2}\bigg).
    \end{align*}
    Notice that $\tilde{R}_{n}(t):=R_{n}\bigg(t+\frac{e^{-\frac{1}{\epsilon}}}{2}\bigg)=H_{n+1}\bigg(t+\frac{e^{-\frac{1}{\epsilon}}}{2}\bigg)-H_{n}\bigg(t+\frac{e^{-\frac{1}{\epsilon}}}{2}\bigg)$ and thus
    \begin{equation}\label{tilde rn}
\left\{\begin{aligned}
\partial_t \tilde{R}_n(y,v,t)+\frac{1}{\epsilon}\partial_y[(v^{\alpha}-y) \tilde{R}_n(y,v,t)]&=\mathcal{K}\bigg[H_{n}\bigg(t+\frac{e^{-\frac{1}{\epsilon}}}{2}\bigg),\tilde{R}_n\bigg]+\mathcal{K}\bigg[\tilde{R}_{n-1},H_{n}\bigg(t+\frac{e^{-\frac{1}{\epsilon}}}{2}\bigg)\bigg]; \\
\tilde{R}_{n}(y,v,0)&=H_{n+1}\bigg(\frac{e^{-\frac{1}{\epsilon}}}{2}\bigg)-H_{n}\bigg(\frac{e^{-\frac{1}{\epsilon}}}{2}\bigg).
   \end{aligned}\right.
   \end{equation}
    \textit{\underline{Initial condition becomes zero:}}    Let now
   \begin{align*}
       P_{n}(y,v,t)&:=\tilde{R}_{n}(y,v,t)-e^{\frac{t}{\epsilon}}R_{n}\bigg(e^{\frac{t}{\epsilon}}(y-v^{\alpha})+v^{\alpha},v,\frac{e^{-\frac{1}{\epsilon}}}{2}\bigg);  \\
       \tilde{P}_{n}(y,v,t)&:=e^{\frac{t}{\epsilon}}R_{n}\bigg(e^{\frac{t}{\epsilon}}(y-v^{\alpha})+v^{\alpha},v,\frac{e^{-\frac{1}{\epsilon}}}{2}\bigg).
   \end{align*}
   Then $P_{n}$ satisfies
   \begin{equation}\label{equation pn}
\left\{\begin{aligned}
\partial_t P_n(y,v,t)+\frac{1}{\epsilon}\partial_{y}[(v^{\alpha}-y) P_n(y,v,t)]&=\mathcal{K}\bigg[H_{n}\bigg(t+\frac{e^{-\frac{1}{\epsilon}}}{2}\bigg),P_n\bigg]+\mathcal{K}\bigg[P_{n-1},H_{n}\bigg(t+\frac{e^{-\frac{1}{\epsilon}}}{2}\bigg)\bigg] \\
&+\mathcal{K}\bigg[H_{n}\bigg(t+\frac{e^{-\frac{1}{\epsilon}}}{2}\bigg), \tilde{P}_{n}\bigg]+\mathcal{K}\bigg[ \tilde{P}_{n-1},H_{n}\bigg(t+\frac{e^{-\frac{1}{\epsilon}}}{2}\bigg)\bigg]\\
P(y,v,0)&=0.
   \end{aligned}\right.
   \end{equation}

    By Duhamel's principle, this becomes equivalent to solving
    \begin{equation}
\left\{\begin{aligned}
\partial_t P_n(y,v,t)+\frac{1}{\epsilon}\partial_{y}[(v^{\alpha}-y) P_n(y,v,t)]&=\mathcal{K}\bigg[H_{n}\bigg(t+\frac{e^{-\frac{1}{\epsilon}}}{2}\bigg),P_n\bigg]\textup{ for } t>s\\
P_{n}(y,v,s)&=\mathcal{K}\bigg[P_{n-1},H_{n}\bigg(s+\frac{e^{-\frac{1}{\epsilon}}}{2}\bigg)\bigg] +\mathcal{K}\bigg[H_{n}\bigg(s+\frac{e^{-\frac{1}{\epsilon}}}{2}\bigg), \tilde{P}_{n}\bigg]\\
&+\mathcal{K}\bigg[ \tilde{P}_{n-1},H_{n}\bigg(s+\frac{e^{-\frac{1}{\epsilon}}}{2}\bigg)\bigg]=:I_{1}+I_{2}+I_{3}.
   \end{aligned}\right.
   \end{equation}
   We are now able to reproduce the estimates we performed in Step $2$ in order to obtain that there exists a constant $C>0$ such that
\begin{align}\label{bound pn}
P_{n}(t)\leq     (Ct e^{\frac{1}{\epsilon}})^{n}+2^{-n} \textup{ for all } n\in\mathbb{N}.
\end{align}
\textit{\underline{Estimates for the iteration in time:}} The term $I_{1}$ we bound by induction as in Step $2$, \eqref{same estimates iteration}, using in addition that
\begin{align*}
P_{n}(t)=& R_{n}\bigg(y,v,t+\frac{e^{-\frac{1}{\epsilon}}}{2}\bigg)- e^{\frac{t}{\epsilon}}R_{n}\bigg(\frac{e^{-\frac{1}{\epsilon}}}{2}\bigg)\leq R_{n}\bigg(y,v,t+\frac{e^{-\frac{1}{\epsilon}}}{2}\bigg)\\
&\leq H_{n+1}\bigg(y,v,t+\frac{e^{-\frac{1}{\epsilon}}}{2}\bigg)+H_{n}\bigg(y,v,t+\frac{e^{-\frac{1}{\epsilon}}}{2}\bigg).
\end{align*}
$I_{2}$ and $I_{3}$ can then bounded in a similar manner. We thus only show $I_{2}$. More precisely, we have that
\begin{align*}
    \mathcal{C}_{2}& \bigg[H_{n}\bigg(t+\frac{e^{-\frac{1}{\epsilon}}}{2}\bigg),\tilde{P}_{n}\bigg]\\
    &:=\int_{0}^{t}\int_{0}^{\infty}e^{\frac{t-s}{\epsilon}}K(v,w)H_{n}(X_{\epsilon}(y,v,t,s),w,s+\frac{e^{-\frac{1}{\epsilon}}}{2})e^{\frac{s}{\epsilon}}R_{n}\bigg(X_{\epsilon}(y,v,t,0),v,\frac{e^{-\frac{1}{\epsilon}}}{2}\bigg)\der w\der s.
\end{align*}
By the upper bound in Theorem \ref{proposition existence via fixed point}, we have that $(1+v^{b})|H_{n}(y,v,s)|\leq Ce^{\frac{s}{\epsilon}}$, for all $t\leq 1$. Moreover, by Step $2$ it holds that $R_{n}(t)\leq (te^{\frac{1}{\epsilon}})^{n}\frac{1+v^{\gamma}}{1+v^{b}}$. Thus, we have that $R_{n}\big(X_{\epsilon}(y,v,t,0),v,\frac{e^{-\frac{1}{\epsilon}}}{2}\big)\leq \frac{1}{2^{n}}\frac{1+v^{\gamma}}{1+v^{b}}$.This implies that as long as $2t\leq 1$, it holds that
\begin{align*}
    \mathcal{C}_{2}& \bigg[H_{n}\bigg(t+\frac{e^{-\frac{1}{\epsilon}}}{2}\bigg),\tilde{P}_{n}\bigg]\leq\frac{1}{2^{n}}\int_{0}^{t}\int_{0}^{\infty}e^{\frac{t+s+\frac{e^{-\frac{1}{\epsilon}}}{2}}{\epsilon}}\der s\leq e^{\frac{1}{2}}\frac{te^{\frac{1}{\epsilon}}}{2^{n}}\leq \frac{\sqrt{e}}{2^{n+1}}.
\end{align*}
We bound $I_{3}$ similarly and conclude that \eqref{bound pn} holds. Since
\begin{align*}
    R_{n}\bigg(y,v,t+\frac{e^{-\frac{1}{\epsilon}}}{2}\bigg)=P_{n}(t)+e^{\frac{t}{\epsilon}}R_{n}\bigg(X_{\epsilon}(y,v,t,0),v,\frac{e^{-\frac{1}{\epsilon}}}{2}\bigg)
\end{align*}
we then obtain that
\begin{align*}
     R_{n}\bigg(y,v,t+\frac{e^{-\frac{1}{\epsilon}}}{2}\bigg)\leq   e^{\frac{t}{\epsilon}}(Ct)^{n}+2^{-n+1}\rightarrow 0 \textup{ as } n\rightarrow \infty, \textup{ for all } t\leq \frac{e^{-\frac{1}{\epsilon}}}{2}.
\end{align*}
We can then iterate Step $3$ to extend the result to hold for all times as long as $t\leq\min\{2^{-1},T\}$ with $T$ as in Theorem \ref{proposition existence via fixed point}.

\vspace{0.2cm}

\vspace{0.2cm}

{\bf Step 4: (Iteration in time in mild formulation)}
As mentioned in Step $3,$ we now show the translation in time argument directly for the mild formulation \eqref{mild form iteration in n}.
\vspace{0.2cm}

\textit{\underline{Mild formulation. Translation in time}}
Since $R_{n}(y,v,0)=0,$ $R_{n}$ satisfies the following mild formulation 
\begin{align*}
R_n(y,v,t)&=\mathcal{C}[H_{n},R_n]+\mathcal{C}[R_{n-1},H_{n}].
\end{align*}
Define now
\begin{align*}
 \tilde{R}_{n}(t):=R_{n}\bigg(t+\frac{e^{-\frac{1}{\epsilon}}}{2}\bigg).
    \end{align*}
    Notice as before that $\tilde{R}_{n}(t):=R_{n}\bigg(t+\frac{e^{-\frac{1}{\epsilon}}}{2}\bigg)=H_{n+1}\bigg(t+\frac{e^{-\frac{1}{\epsilon}}}{2}\bigg)-H_{n}\bigg(t+\frac{e^{-\frac{1}{\epsilon}}}{2}\bigg)$. We show the argument for  $\mathcal{C}_{2}[H_{n},R_n]$ in \eqref{k2 term 2} as the other terms are obtained in the same manner.
\begin{align*}
&\mathcal{C}_{2}[H_{n},R_n]  \bigg(t+\frac{e^{-\frac{1}{\epsilon}}}{2}\bigg)\\
&=\int_{0}^{t+\frac{e^{-\frac{1}{\epsilon}}}{2}}\int_{0}^{\infty}e^{\frac{t+\frac{e^{-\frac{1}{\epsilon}}}{2}-s}{\epsilon}}K(v,w)H_{n}(e^{\frac{t+\frac{e^{-\frac{1}{\epsilon}}}{2}-s}{\epsilon}}(y-v^{\alpha})+v^{\alpha},w,s)R_{n}(e^{\frac{t+\frac{e^{-\frac{1}{\epsilon}}}{2}-s}{\epsilon}}(y-v^{\alpha})+v^{\alpha},v,s)\der w\der s\\
&=\int_{0}^{\frac{e^{-\frac{1}{\epsilon}}}{2}}\int_{0}^{\infty}e^{\frac{t+\frac{e^{-\frac{1}{\epsilon}}}{2}-s}{\epsilon}}K(v,w)H_{n}(e^{\frac{t+\frac{e^{-\frac{1}{\epsilon}}}{2}-s}{\epsilon}}(y-v^{\alpha})+v^{\alpha},w,s)R_{n}(e^{\frac{t+\frac{e^{-\frac{1}{\epsilon}}}{2}-s}{\epsilon}}(y-v^{\alpha})+v^{\alpha},v,s)\der w\der s\\
&+\int_{\frac{e^{-\frac{1}{\epsilon}}}{2}}^{t+\frac{e^{-\frac{1}{\epsilon}}}{2}}\int_{0}^{\infty}e^{\frac{t+\frac{e^{-\frac{1}{\epsilon}}}{2}-s}{\epsilon}}K(v,w)H_{n}(e^{\frac{t+\frac{e^{-\frac{1}{\epsilon}}}{2}-s}{\epsilon}}(y-v^{\alpha})+v^{\alpha},w,s)R_{n}(e^{\frac{t+\frac{e^{-\frac{1}{\epsilon}}}{2}-s}{\epsilon}}(y-v^{\alpha})+v^{\alpha},v,s)\der w\der s\\
&:=L_{1}+L_{2}.
\end{align*}
For $L_{2}$ we make the following change of variable $\tilde{s}=s-\frac{e^{-\frac{1}{\epsilon}}}{2}$. It thus holds that
\begin{align*}
L_{2}=\int_{0}^{t}\int_{0}^{\infty}e^{\frac{t-\tilde{s}}{\epsilon}}K(v,w)H_{n}\bigg(e^{\frac{t-\tilde{s}}{\epsilon}}(y-v^{\alpha})+v^{\alpha},w,\tilde{s}+\frac{e^{-\frac{1}{\epsilon}}}{2}\bigg)R_{n}\bigg(e^{\frac{t-\tilde{s}}{\epsilon}}(y-v^{\alpha})+v^{\alpha},v,\tilde{s}+\frac{e^{-\frac{1}{\epsilon}}}{2}\bigg)\der w\der \tilde{s}.
\end{align*}
Thus, $\tilde{R}_{n}$ satisfies the following equation.
\begin{align*}
    \tilde{R}_{n}(y,v,t)=e^{\frac{t}{\epsilon}}R_{n}\bigg(e^{\frac{t}{\epsilon}}(y-v^{\alpha})+v^{\alpha},v,\frac{e^{-\frac{1}{\epsilon}}}{2}\bigg)+\mathcal{C}\bigg[H_{n}\bigg(\cdot+\frac{e^{-\frac{1}{\epsilon}}}{2}\bigg),\tilde{R}_n\bigg]+\mathcal{C}\bigg[\tilde{R}_{n-1},H_{n}\bigg(\cdot+\frac{e^{-\frac{1}{\epsilon}}}{2}\bigg)\bigg].
\end{align*}
  \textit{\underline{Mild formulation. Initial condition becomes $0$:}}    Lastly, let 
   \begin{align*}
       P_{n}(y,v,t)&:=\tilde{R}_{n}(y,v,t)-e^{\frac{t}{\epsilon}}R_{n}\bigg(e^{\frac{t}{\epsilon}}(y-v^{\alpha})+v^{\alpha},v,\frac{e^{-\frac{1}{\epsilon}}}{2}\bigg)\\ \tilde{P}_{n}(y,v,t)&:=e^{\frac{t}{\epsilon}}R_{n}\bigg(e^{\frac{t}{\epsilon}}(y-v^{\alpha})+v^{\alpha},v,\frac{e^{-\frac{1}{\epsilon}}}{2}\bigg).
   \end{align*}
   Then $P_{n}$ satisfies
   \begin{align*}
    \tilde{P}_{n}(y,v,t)=\mathcal{C}\bigg[H_{n}\bigg(\cdot+\frac{e^{-\frac{1}{\epsilon}}}{2}\bigg),\tilde{R}_n\bigg]+\mathcal{C}\bigg[\tilde{R}_{n-1},H_{n}\bigg(\cdot+\frac{e^{-\frac{1}{\epsilon}}}{2}\bigg)\bigg].
   \end{align*}
   In other words, $P_{n}$ satisfies
   \begin{equation}
\left\{\begin{aligned}
 P_n(y,v,t)&=\mathcal{C}\bigg[H_{n}\bigg(\cdot+\frac{e^{-\frac{1}{\epsilon}}}{2}\bigg),\tilde{R}_n\bigg]+\mathcal{C}\bigg[\tilde{R}_{n-1},H_{n}\bigg(\cdot+\frac{e^{-\frac{1}{\epsilon}}}{2}\bigg)\bigg] \\
&+\mathcal{C}\bigg[H_{n}\bigg(\cdot+\frac{e^{-\frac{1}{\epsilon}}}{2}\bigg), \tilde{P}_{n}\bigg]+\mathcal{C}\bigg[ \tilde{P}_{n-1},H_{n}\bigg(\cdot+\frac{e^{-\frac{1}{\epsilon}}}{2}\bigg)\bigg]\\
P(y,v,0)&=0.
   \end{aligned}\right.
   \end{equation}
We are now able to reproduce the estimates from Step $3.$  This concludes our proof.\end{proof}
\section{Proof of Theorem \ref{teo convergence} - Convergence to a Dirac measure}
\begin{prop}[Moment estimates - general version]
    Let $b>\gamma+3$ and $m>\gamma+3-\alpha$. Let $T\leq 1$. There exist $r>0$ and $\omega:=\min\{m-\gamma-1+\alpha,b-\gamma-1\}>2$ such that
\begin{align}\label{integral in y}
    \int_{(0,\infty)}(w^{-r}+w^{\gamma})H_{\epsilon}(y,w,t)\der w\leq \frac{C}{1+|y|^{\omega}}, \textup{ for all } t\in[0,T], y\in\mathbb{R}.
\end{align}and thus 
    \begin{align}\label{integral large y}
    \int_{\mathbb{R}}\int_{(0,\infty)}(1+|y|)H_{\epsilon}(y,w,t)\der w\der y\leq C, \textup{ for all } t\in[0,T].
\end{align}
\end{prop}
\begin{proof}
Since $\alpha<1$, there exists $r\in(0,1)$ such that $\alpha+2r\leq 1$.

\eqref{integral in y} is a more general version of the proof of Proposition \ref{moment estimates prop}. We only have to prove that $\int_{(0,\infty)}w^{-r}H_{\epsilon}(y,w,t)\der w\leq \frac{C}{1+|y|^{\omega}}$. Let $C_{0}$ be as in Proposition \ref{supersol prop above region}. Let $y\geq 0$. We distinguish between different regions. \begin{itemize}
\item Assume first that $\{y\geq w^{\alpha}-C_{0}^{-1}\epsilon^{-\frac{1}{m-1}} e^{-\frac{t}{\epsilon}}\}$. Then
\end{itemize}
\begin{align*}
    \int_{\{y\geq w^{\alpha}-C_{0}^{-1}\epsilon^{-\frac{1}{m-1}} e^{-\frac{t}{\epsilon}}\}}w^{-r}P(y,w,t)\der w\leq \int_{(0,\infty)} \frac{Cw^{-r}e^{\frac{t}{\epsilon}}}{(1+w^{b})(1+e^{\frac{tm}{\epsilon}}|y-w^{\alpha}|^{m})}\der w.
\end{align*}
 We make the change of variables
 $z=e^{\frac{t}{\epsilon}}(y-w^{\alpha})$ and we obtain that
 \begin{align*}
        \int_{\{y\geq w^{\alpha}-C_{0}^{-1}\epsilon^{-\frac{1}{m-1}} e^{-\frac{t}{\epsilon}}\}}w^{-r}P(y,w,t)\der w&\leq \int_{(-\infty,e^{\frac{t}{\epsilon}}y)} \frac{C(y-e^{-\frac{t}{\epsilon}}z)^{\frac{1-\alpha-r}{\alpha}}}{(1+(y-e^{-\frac{t}{\epsilon}}z)^{\frac{b}{\alpha}})(1+|z|^{m})}\der z.
        \end{align*}
      On one side, if $|z|\leq \frac{e^{\frac{t}{\epsilon}}y}{2}$, it holds that $\frac{y}{2}\leq y-e^{-\frac{t}{\epsilon}}z\leq \frac{3 y}{2}$. Thus
        \begin{align}\label{step 1 moments general} \int_{|z|\leq \frac{e^{\frac{t}{\epsilon}}y}{2}} \frac{C(y-e^{-\frac{t}{\epsilon}}z)^{\frac{1-\alpha-r}{\alpha}}}{(1+(y-e^{-\frac{t}{\epsilon}}z)^{\frac{b}{\alpha}})(1+|z|^{m})}\der z&\leq \frac{C  y^{\frac{1-\alpha-r}{\alpha}}}{1+|y|^{\frac{b}{\alpha}}}\int_{(-\infty,\infty)} \frac{C}{1+|z|^{m}}\leq\frac{C}{1+|y|^{\frac{b+r-1+\alpha}{\alpha}}}.
 \end{align}
On the other hand, since  $1-\alpha-r\geq 0$, it holds that $\frac{(y-e^{-\frac{t}{\epsilon}}z)^{\frac{1-\alpha-r}{\alpha}}}{(1+(y-e^{-\frac{t}{\epsilon}}z)^{\frac{b}{\alpha}})}\leq C$. Thus
\begin{align}\label{step 2 moments general}
   \int_{ |z|\geq \frac{e^{\frac{t}{\epsilon}}y}{2}} \frac{C(y-e^{-\frac{t}{\epsilon}}z)^{\frac{1-\alpha-r}{\alpha}}}{(1+(y-e^{-\frac{t}{\epsilon}}z)^{\frac{b}{\alpha}})(1+|z|^{m})}\der z&\leq   \int_{ |z|\geq \frac{e^{\frac{t}{\epsilon}}y}{2}} \frac{C}{1+|z|^{m}}\der z\leq \frac{C}{1+(e^{\frac{t}{\epsilon}}y)^{m-1}}\leq\frac{C}{1+y^{m-1}}.
\end{align}
Combining \eqref{step 1 moments general} and \eqref{step 2 moments general}, we obtain that \eqref{integral in y} holds in the region $\{y\geq w^{\alpha}-C_{0}^{-1}\epsilon^{-\frac{1}{m-1}} e^{-\frac{t}{\epsilon}}\}$.
\begin{itemize}
  \item  We now analyze the region $(\frac{w}{3})^{\alpha}\leq y\leq w^{\alpha}-C_{0}^{-1}\epsilon^{-\frac{1}{m-1}}  e^{-\frac{t}{\epsilon}}$. Notice that in this region it holds that $y^{\frac{1}{\alpha}}\leq (y+C_{0}^{-1}\epsilon^{-\frac{1}{m-1}}  e^{-\frac{t}{\epsilon}})^{\frac{1}{\alpha}}\leq w\leq 3y^{\frac{1}{\alpha}}$. Thus
\end{itemize} 
 \begin{align*}
      \int_{\{(\frac{w}{3})^{\alpha}\leq y\leq w^{\alpha}-C_{0}^{-1}\epsilon^{-\frac{1}{m-1}}  e^{-\frac{t}{\epsilon}}\}}w^{-r}P(y,w,t)\der w&\leq \int^{3y^{\frac{1}{\alpha}}}_{(y+C_{0}^{-1}\epsilon^{-\frac{1}{m-1}}  e^{-\frac{t}{\epsilon}})^{\frac{1}{\alpha}}} \frac{Cw^{-r}\epsilon}{(1+w^{b})|y-w^{\alpha}|}\der w\\
      &\leq \frac{C\epsilon y^{\frac{-r}{\alpha}}}{1+y^{\frac{b}{\alpha}}}\int^{3y^{\frac{1}{\alpha}}}_{(y+C_{0}^{-1}\epsilon^{-\frac{1}{m-1}}  e^{-\frac{t}{\epsilon}})^{\frac{1}{\alpha}}} \frac{1}{|y-w^{\alpha}|}\der w.
 \end{align*}
We make the change of variables $z=y-w^{\alpha}$ and obtain that 
\begin{align*}
  \frac{C\epsilon  y^{\frac{-r}{\alpha}}}{1+y^{\frac{b}{\alpha}}}  \int^{3y^{\frac{1}{\alpha}}}_{(y+C_{0}^{-1}\epsilon^{-\frac{1}{m-1}}  e^{-\frac{t}{\epsilon}})^{\frac{1}{\alpha}}} \frac{1}{|y-w^{\alpha}|}\der w&\leq \frac{C\epsilon y^{\frac{1-\alpha-r}{\alpha}}}{1+y^{\frac{b}{\alpha}}}\int_{-(3^{\alpha}-1)y}^{-C_{0}^{-1}\epsilon^{-\frac{1}{m-1}}e^{-\frac{t}{\epsilon}}}\frac{\der z}{|z|}\\
  & \leq \frac{C\epsilon  y^{\frac{1-\alpha-r}{\alpha}}}{1+y^{\frac{b}{\alpha}}}\big[\frac{t}{\epsilon}+|\ln{y}|\big].
\end{align*}
If $y\geq 1$, then
\begin{align*}
  \frac{C\epsilon  y^{\frac{1-\alpha-r}{\alpha}}}{1+y^{\frac{b}{\alpha}}}\big[\frac{t}{\epsilon}+\ln{y}\big]  \leq\frac{Ct  y^{\frac{1-\alpha-r}{\alpha}}}{1+y^{\frac{b}{\alpha}}}+\frac{C\epsilon  y^{\frac{1-r}{\alpha}}}{1+y^{\frac{b}{\alpha}}}y^{-1}\ln{y}\leq \frac{C}{1+y^{\frac{b+r-1}{\alpha}}},
\end{align*}
where in the last inequality we used the fact that $t\leq 1$.
If $y\leq 1$, we use the fact that $1-2r-\alpha\geq 0$, together with the fact that $y^{r}\ln \frac{1}{y}\leq C$, in order to conclude in the same manner that
\begin{align*}
     \frac{C\epsilon y^{\frac{1-\alpha-r}{\alpha}}}{1+y^{\frac{b}{\alpha}}}\big[\frac{t}{\epsilon}-\ln{y}\big] \leq  \frac{Ct y^{\frac{1-\alpha-r}{\alpha}}}{1+y^{\frac{b}{\alpha}}}+\frac{C\epsilon y^{\frac{1-\alpha-2r}{\alpha}}}{1+y^{\frac{b}{\alpha}}}y^{r}\ln{\frac{1}{y}}  \leq \frac{C}{1+y^{\frac{b+2r-1}{\alpha}}}.
\end{align*}

The region  $\{ y \leq (\frac{w}{3})^{\alpha}\}$ follows as in the proof of Proposition \ref{moment estimates prop} using in addition the fact that $1-\alpha-r\geq 0$ as in \eqref{step 2 moments general}. This concludes our proof.

\end{proof}
We notice that $H_{\epsilon}$ found in Theorem \ref{proposition existence via fixed point} satisfies the weak version of \eqref{mild solution equation}.

\vspace{0.2cm}

\noindent\underline{\textbf{\textit{Weak formulation:}}} From \eqref{mild solution equation}, it immediately follows that for every $t\in[0,T]$ and for every $\varphi\in\textup{C}^{1}_{c}(\mathbb{R}\times(0,\infty)) $, it holds that
\begin{align}
    \partial_{t}\int_{\mathbb{R}}\int_{(0,\infty)} &  H_{\epsilon}(y,v,t)\varphi(y,v)\der v \der y= \frac{1}{\epsilon}\int_{\mathbb{R}}\int_{(0,\infty)}H_{\epsilon}(y,v,t)(v^{\alpha}-y)\partial_{y}\varphi(y,v)\der v \der y\nonumber\\
    &+\frac{1}{2}\int_{\mathbb{R}}\int_{(0,\infty)} \int_{(0,\infty)} K(v,w)H_{\epsilon}(y,v,t)H_{\epsilon}(y,w,t)\chi_{\varphi}(y,v,w)\der w\der v \der y, \label{weak formulation}
\end{align}
where we denoted
\begin{align*}
    \chi_{\varphi}(y,v,w):=\varphi(y,v+w)-\varphi(y,v)-\varphi(y,w).
\end{align*}

We are now in a position to prove Theorem \ref{teo convergence}.
\begin{proof}[Proof of Theorem \ref{teo convergence}] 

We first prove that for every fixed $t>0$, $H_{\epsilon}$ converges to a Dirac measure.

\vspace{0.2cm}

\noindent\underline{\textbf{\textit{Convergence to a Dirac measure for fixed time:}}} If $|y-v^{\alpha}|\geq \delta$ and $t>0$, then by \eqref{def t1}-\eqref{def t3}, it holds that
\begin{align*}
    T_{1}(y,v,t)&=  \frac{2A}{1+v^{b}}e^{\frac{t}{\epsilon}}\psi(e^{\frac{t}{\epsilon}}(y-v^{\alpha}))\mathbbm{1}_{\{y\geq v^{\alpha}-C_{0}^{-1}\epsilon^{-\frac{1}{m-1}}  e^{-\frac{t}{\epsilon}}\}}\leq \frac{ Ce^{\frac{t}{\epsilon}}}{1+e^{\frac{tm}{\epsilon}}|y-v^{\alpha}|^{m}}\leq C(\delta)e^{-\frac{t(m-1)}{\epsilon}};\\
   T_{2}(y,v,t)&=\frac{2C_{0}^{m-1}\max\{1,A^{2}\}\epsilon}{(1+v^{b})|y-v^{\alpha}|}\mathbbm{1}_{\{(\frac{v}{3})^{\alpha}\leq y\leq v^{\alpha}-C_{0}^{-1}\epsilon^{-\frac{1}{m-1}}  e^{-\frac{t}{\epsilon}}\}}\leq C(\delta)\epsilon;\\
     T_{3}(y,v,t)&=  \frac{M_{2}A}{1+v^{b}}e^{\frac{t}{\epsilon}}\psi(e^{\frac{t}{\epsilon}}(y-v^{\alpha}))\mathbbm{1}_{\{y\leq (\frac{v}{3})^{\alpha}\}}\leq C(\delta)e^{-\frac{t(m-1)}{\epsilon}}.
\end{align*}
Thus, if $t>0$, $\supp \varphi\subseteq $$\{|y-v^{\alpha}|\geq \delta\}$, and $\varphi$ is compactly supported, it holds that
\begin{align}\label{convergence to dirac for fixed time}
    \int_{\mathbb{R}}\int_{(0,\infty)}H_{\epsilon}(y,v,t)\varphi(y,v)\der v \der y\leq C(\delta)[\epsilon+e^{-\frac{t(m-1)}{\epsilon}}]\rightarrow 0 \textup{ as } \epsilon\rightarrow 0.
\end{align}

We continue by proving equicontinuity in time. 

\vspace{0.2cm}

\underline{\textbf{\textit{Equicontinuity:}}} Let $\varphi\in\textup{C}_{0}(\mathbb{R}\times(0,\infty))$ and $\sigma>0$. We want to prove that for any $\delta\in(0,1)$, there exists $\delta_{1}$ such that if $t,s\in[\sigma,T]$ with $|t-s|\leq \delta_{1}$, then 
\begin{align}\label{equicontinuity}
    \bigg|\int_{\mathbb{R}}\int_{(0,\infty)} & [H_{\epsilon}(y,v,t)-H_{\epsilon}(y,v,s)]\varphi(y,v)\der v \der y\bigg|\leq \delta.
    \end{align}
\underline{\textit{Estimates for the transport term:}} We first prove some estimates for the transport term in \eqref{weak formulation}. Let $\varphi\in\textup{C}^{1}_{c}(\mathbb{R}\times(0,\infty)) $ and $t\geq \sigma$. It holds that
\begin{align}
    \int_{\mathbb{R}}\int_{(0,\infty)}H_{\epsilon}(y,v,t)|y-v^{\alpha}|\partial_{y}\varphi(y,v)\der v \der y\leq  C\int_{\mathbb{R}}\int_{(0,\infty)}[T_{1}+T_{2}+T_{3}](y,v,t)|y-v^{\alpha}|\der v \der y.
\end{align}
We estimate first $T_{1}+T_{3}$. We have that
\begin{align}
    \int_{\mathbb{R}}  \int_{(0,\infty)}[T_{1}+T_{3}](y,v,t)|y-v^{\alpha}|\der v \der y &\leq   \int_{\mathbb{R}}\int_{(0,\infty)}\frac{C|y-v^{\alpha}|}{1+v^{b}}e^{\frac{t}{\epsilon}}\psi(e^{\frac{t}{\epsilon}}(y-v^{\alpha}))\der v \der y\nonumber\\
    &\leq \int_{\mathbb{R}}\int_{(0,\infty)}\frac{Ce^{\frac{t}{\epsilon}}|y-v^{\alpha}|}{1+v^{b}}\frac{1}{1+e^{\frac{tm}{\epsilon}}|y-v^{\alpha}|^{m}}\der v \der y\nonumber\\
     &\leq \int_{\mathbb{R}}\int_{(0,\infty)}\frac{C}{1+v^{b}}\frac{1}{1+e^{\frac{t(m-1)}{\epsilon}}|y-v^{\alpha}|^{m-1}}\der v \der y.\label{transport term vanishes in y}
\end{align}
We make the change of variables $\xi=e^{\frac{t}{\epsilon}}(y-v^{\alpha})$. \eqref{transport term vanishes in y} then becomes for all $t\geq \sigma$
\begin{align*}
    \int_{\mathbb{R}}  \int_{(0,\infty)}[T_{1}+T_{3}](y,v,t)|y-v^{\alpha}|\der v \der y&\leq \int_{\mathbb{R}}\int_{(0,\infty)}\frac{C}{1+v^{b}}\frac{e^{-\frac{t}{\epsilon}}}{1+|\xi|^{m-1}}\der v \der \xi\\
    &\leq e^{-\frac{\sigma}{\epsilon}}\int_{\mathbb{R}}\int_{(0,\infty)}\frac{C}{1+v^{b}}\frac{1}{1+|\xi|^{m-1}}\der v \der \xi\leq C e^{-\frac{\sigma}{\epsilon}}.
\end{align*}
We continue by estimating $T_{2}$. In this case, it follows that
\begin{align*}
    \int_{\mathbb{R}}\int_{(0,\infty)}T_{2}(y,v,t)|y-v^{\alpha}|\der v \der y &\leq  \int_{\mathbb{R}}\int_{(0,\infty)}\frac{C\epsilon|y-v^{\alpha}|}{(1+v^{b})|y-v^{\alpha}|}\mathbbm{1}_{\{(\frac{v}{3})^{\alpha}\leq y\leq v^{\alpha}\}}\der v\der y\\
    &\leq \int_{(0,\infty)}\frac{C \epsilon}{1+v^{b-\alpha}}\der v\leq C\epsilon.
\end{align*}
Combining all the above estimates, we have that
\begin{align}\label{estimates for the transport term}
     \int_{\mathbb{R}}\int_{(0,\infty)}H_{\epsilon}(y,v,t)|y-v^{\alpha}|\partial_{y}\varphi(y,v)\der v \der y\leq C e^{-\frac{\sigma}{\epsilon}}+C\epsilon.
\end{align}

\textit{\underline{Equicontinuity, test functions in $\textup{C}^{1}_{c}(\mathbb{R}\times(0,\infty))$:}} Let $M>1$. We first prove \eqref{equicontinuity} for a fixed function $\varphi\in\textup{C}^{1}_{c}(\mathbb{R}\times(0,\infty)) $ such that $\supp\varphi\in\{[-M,M]\times [\frac{1}{M},M]\}$, $||\varphi||_{L^{\infty}}\leq 1$. Let $s,t\in [\sigma,T]$. By \eqref{weak formulation}, it holds that
\begin{align*}
\bigg|\int_{\mathbb{R}}\int_{(0,\infty)} &[H_{\epsilon}(y,v,t)-H_{\epsilon}(y,v,s)]\varphi(y,v)\der v \der y\bigg|\\
\leq &\frac{1}{\epsilon}\int_{s}^{t}\int_{\mathbb{R}}\int_{(0,\infty)} H_{\epsilon}(y,v,z)|v^{\alpha}-y||\partial_{y}\varphi(y,v)|\der v \der y\der z\\
&+\frac{1}{2}\int_{s}^{t}\int_{\mathbb{R}}\int_{(0,\infty)}\int_{(0,\infty)}K(v,w)H_{\epsilon}(y,v,z)H_{\epsilon}(y,w,z)|\chi_{\varphi}(y,v,w)|\der w\der v \der y\der z\\
\leq &\frac{1}{\epsilon}\int_{s}^{t}\int_{\mathbb{R}}\int_{(0,\infty)} H_{\epsilon}(y,v,z)|v^{\alpha}-y||\partial_{y}\varphi(y,v)|\der v \der y\der z\\
&+C\int_{s}^{t}\int_{\mathbb{R}}\int_{(0,\infty)}(1+v^{\gamma})H_{\epsilon}(y,v,z)\der v\int_{(0,\infty)}(1+w^{\gamma})H_{\epsilon}(y,w,z)\der w \der y\der z\\
&=: J_{1}+J_{2}.
\end{align*}
By \eqref{estimates for the transport term} and since $t,s\geq \sigma$, it follows that
\begin{align}\label{j1 j1}
    J_{1}\leq \frac{C}{\epsilon}[e^{-\frac{\sigma}{\epsilon}}+\epsilon]|t-s|\leq C|t-s|,
\end{align}
for all $\epsilon\leq \epsilon_{\sigma}$, where $\epsilon_{\sigma}$ only depends on $\sigma$. 
By \eqref{integral in y} it follows that
\begin{align}\label{j2 j2}
J_{2}\leq \int_{s}^{t}\int_{\mathbb{R}} \frac{C}{1+|y|^{2\omega}}\der y\der z\leq C|t-s|.
\end{align}
By \eqref{j1 j1} and \eqref{j2 j2}, we obtain that
\begin{align}\label{equicontinuity c1}
    \bigg|\int_{\mathbb{R}}\int_{(0,\infty)} &[H_{\epsilon}(y,v,t)-H_{\epsilon}(y,v,s)]\varphi(y,v)\der v \der y\bigg|\leq C|t-s|,
\end{align}
when $\varphi\in\textup{C}^{1}_{c}(\mathbb{R}\times(0,\infty)) $.

\textit{\underline{From $\textup{C}^{1}_{c}(\mathbb{R}\times(0,\infty))$ to  $\textup{C}_{c}(\mathbb{R}\times(0,\infty)):$}}  Let $\varphi\in\textup{C}_{c}(\mathbb{R}\times(0,\infty))$ such that $\supp\varphi\in[-M,M]\times [\frac{1}{M},M]$. There exists $\varphi_{n}\in\textup{C}^{1}_{c}(\mathbb{R}\times(0,\infty))$ such that $||\varphi-\varphi_{n}||_{\infty}\leq \frac{1}{n}$. It then holds that
\begin{align*}
    \bigg|\int_{\mathbb{R}}\int_{(0,\infty)} & [H_{\epsilon}(y,v,t)-H_{\epsilon}(y,v,s)]\varphi(y,v)\der v \der y\bigg|\leq \bigg|\int_{\mathbb{R}}\int_{(0,\infty)} [H_{\epsilon}(y,v,t)-H_{\epsilon}(y,v,s)]\varphi_{n}(y,v)\der v \der y\bigg|\\
    &+\int_{\mathbb{R}}\int_{(0,\infty)} [H_{\epsilon}(y,v,t)+H_{\epsilon}(y,v,s)]|\varphi(y,v)-\varphi_{n}(y,v)|\der v \der y.
\end{align*}
By \eqref{equicontinuity c1}, it follows that 
\begin{align}\label{equi compact supp}
    \bigg|\int_{\mathbb{R}}\int_{(0,\infty)} & [H_{\epsilon}(y,v,t)-H_{\epsilon}(y,v,s)]\varphi(y,v)\der v \der y\bigg|\nonumber\\
    &\leq C|t-s|+\frac{1}{n}\int_{\mathbb{R}}\int_{(0,\infty)} [H_{\epsilon}(y,v,t)+H_{\epsilon}(y,v,s)]\der v \der y\leq C|t-s|+\frac{C}{n}.
\end{align}
We can then conclude by choosing $n\in\mathbb{N}$ sufficiently large.

\textit{\underline{From $\textup{C}_{c}(\mathbb{R}\times(0,\infty))$ to  $\textup{C}_{0}(\mathbb{R}\times(0,\infty)):$}} 
Let $\varphi\in\textup{C}_{0}(\mathbb{R}\times(0,\infty))$. Let $M>1$, then 
\begin{align*}
    \bigg|\int_{\mathbb{R}}\int_{(0,\infty)} & [H_{\epsilon}(y,v,t)-H_{\epsilon}(y,v,s)]\varphi(y,v)\der v \der y\bigg|\\
    &\leq \bigg|\int_{[-M,M]}\int_{[\frac{1}{M},M]} [H_{\epsilon}(y,v,t)-H_{\epsilon}(y,v,s)]\varphi(y,v)\der v \der y\bigg|\\
    &+\int_{\mathbb{R}}\int_{\{v\geq M\}\cup \{v\leq\frac{1}{M}\}} [H_{\epsilon}(y,v,t)+H_{\epsilon}(y,v,s)]|\varphi(y,v)-\varphi_{n}(y,v)|\der v \der y\\
     &+\int_{|y|\geq M}\int_{(M,\infty)} [H_{\epsilon}(y,v,t)+H_{\epsilon}(y,v,s)]|\varphi(y,v)-\varphi_{n}(y,v)|\der v \der y=: L_{1}+L_{2}+L_{3}.
\end{align*}
We have that the following inequalities hold by choosing $M$ sufficiently large.
\begin{align*}
    L_{1}&\leq \frac{\delta}{3} \qquad  \textup{ by } \eqref{equi compact supp};\\
    L_{2}&\leq M^{-1}\int_{\mathbb{R}}\int_{\{v\geq M\}} v [H_{\epsilon}(y,v,t)+H_{\epsilon}(y,v,s)]\der v \der y\\
    &+M^{-r}\int_{\mathbb{R}}\int_{\{v\leq\frac{1}{M}\}}v^{-r} [H_{\epsilon}(y,v,t)+H_{\epsilon}(y,v,s)]\der v \der y\leq \frac{\delta}{3} \qquad \textup{ by } \eqref{integral in y};\\
    L_{3}&\leq M^{-1}\int_{|y|\geq M}\int_{(M,\infty)} |y|[H_{\epsilon}(y,v,t)+H_{\epsilon}(y,v,s)]\der v \der y\leq \frac{\delta}{3}  \qquad \textup{ by } \eqref{integral large y}.
\end{align*}
This concludes the proof of equicontinuity. Similar computations for the passage from $\varphi\in\textup{C}^{1}_{c}(\mathbb{R}\times(0,\infty))$ to  $\varphi\in\textup{C}_{0}(\mathbb{R}\times(0,\infty))$ have also been used in \cite{cristianspherical}.

\vspace{0.3cm}

\noindent\underline{\textbf{\textit{Convergence to a Dirac measure:}}} 
 By  \eqref{integral large y}, we have that for $t\in[0,T]$, there exists a convergent subsequence of $H_{\epsilon}(t)$ in the weak-$^{\ast}$ topology as $\epsilon\rightarrow 0$.  Let $\sigma>0$. Together with \eqref{equicontinuity}, we deduce from Arzelà–Ascoli theorem that there exists a subsequence of $\{H_{\epsilon}\}$, which we do not relabel, and an $H\in\textup{C}([\sigma,T];\mathscr{M}_{+}(\mathbb{R}\times(0,\infty))),$ such that $H_{\epsilon}(t)$ converge to $H(t)$ in the weak-$^{\ast}$ topology as $\epsilon\rightarrow 0$, for every $t\in[\sigma,T].$

By \eqref{convergence to dirac for fixed time}, we have that for every $t>0$ and $\varphi$ such that $\supp \varphi\subseteq $$\{|y-v^{\alpha}|\geq \delta\}$, $\varphi$ is compactly supported, it holds that
\begin{align}
    \int_{\mathbb{R}}\int_{(0,\infty)}H_{\epsilon}(y,v,t)\varphi(y,v)\der v \der y\rightarrow 0 \textup{ as } \epsilon\rightarrow 0.
\end{align}The extension to functions $\varphi\in\textup{C}_{0}(\mathbb{R}\times(0,\infty))$ is as before. In other words, for any $\delta>0$ and for every  $\varphi\in\textup{C}_{0}(\mathbb{R}\times(0,\infty))$ such that $\supp \varphi\subseteq \{|y-v^{\alpha}|\geq \delta\}$, we have that
\begin{align*}
 0=  \lim_{\epsilon\rightarrow 0} \int_{\mathbb{R}}\int_{(0,\infty)}H_{\epsilon}(y,v,t)\varphi(y,v)\der v \der y= \int_{\mathbb{R}}\int_{(0,\infty)}H(y,v,t)\varphi(y,v)\der v \der y. 
\end{align*}
This concludes our proof.
\end{proof}
\appendix
\section{Rescaled version of the results}\label{appendix rescaled version of the results}
Existence of mass-conserving solutions has been proved in \cite{cristianinhom} for the rain model \eqref{original equation} for times of order one.  Theorem \ref{proposition existence via fixed point} extends the existence of mass-conserving solutions in \cite{cristianinhom} to hold for times of order $\mathcal{O}(e^{\frac{1}{\epsilon}})$ in the case of coagulation kernels with negligible contribution. More precisely, making use of \eqref{rain model to our model}, \eqref{our model to rain model}, Table \ref{tab:my_label}, and Theorem \ref{proposition existence via fixed point}, we obtain that
\begin{prop}
 Let $H_{\epsilon,\textup{in}}$ be as in Assumption \ref{assumption initial condition} and $K$ as in Assumption \ref{assumption kernel}.  Let   $b
\geq \max\{\overline{b}(\gamma,\alpha),2\gamma+1\}$ and $m>\max\{\frac{2(\gamma+1)}{\alpha},\frac{b}{\alpha}+1\}$, where  $\overline{b}(\gamma,\alpha)$ is as in \eqref{definition overline b}.  There exist $M_{1}(\alpha,\gamma,b,m), M_{2}(\alpha,\gamma,b,m),$ depending only on $\alpha,\gamma,b,m$, such that for any $A\geq 1$, there exist $\overline{\epsilon}\in(0,1)$, sufficiently small, and  $T>0$, sufficiently small, which is independent of $\epsilon\in(0,1),$ such that  for all $\epsilon\in(0,\overline{\epsilon}]$ there exists a mass-conserving mild solution $f_{\epsilon}\in\textup{C}([0,e^{\frac{T}{\epsilon}}-1];L^{\infty}(\mathbb{R}\times(0,\infty)))$ to (\ref{original equation}). Moreover, we have that $f_{\epsilon}$ satisfies the following estimates for $C_{0}$ as in \eqref{defc0}.
   
If $v^{\alpha}(\tau+1)\big[1-(\frac{1}{3})^{\alpha}\big]\geq C_{0}^{-1}\epsilon^{-\frac{1}{m-1}}  $, then
    \begin{align}\label{converted bounds}
   0\leq (1+v^{b}) f_{\epsilon}(y,v,\tau)\leq &\frac{2A}{1+|x-(\tau+1)v^{\alpha}|^{m}}\mathbbm{1}_{ A_{\epsilon,y,v,\tau}}+\frac{2M_{1}A^{3}\epsilon}{|x-(\tau+1)v^{\alpha}|}\mathbbm{1}_{ \big(A_{\epsilon,y,v,\tau}\cup B_{\epsilon,y,v,\tau}\big)^{c}}\\
   &+\frac{M_{2}A}{1+|x-(\tau+1)v^{\alpha}|^{m}}\mathbbm{1}_{ B_{\epsilon,y,v,\tau}},\nonumber
    \end{align}
where
   \begin{align*}
       A_{\epsilon,y,v,\tau}&:=\{x-(\tau+1)v^{\alpha}\geq -C_{0}^{-1}\epsilon^{-\frac{1}{m-1}} \};\\
         B_{\epsilon,y,v,\tau}&:=\bigg\{\frac{x}{\tau+1}\leq \Big(\frac{v}{3}\Big)^{\alpha}\bigg\}.
   \end{align*}  Otherwise, if  $0\leq v^{\alpha}(\tau+1)\big[1-(\frac{1}{3})^{\alpha}\big]\leq C_{0}^{-1}\epsilon^{-\frac{1}{m-1}}  $, then
\begin{align}
   0\leq (1+v^{b}) f_{\epsilon}(x,v,\tau)\leq \frac{2A}{1+|x-(\tau+1)v^{\alpha}|^{m}}.
    \end{align}
\end{prop}
With the form of the bounds in \eqref{converted bounds}, the reason behind the choice of \eqref{existence space fixed point} becomes clear. More precisely,
\begin{rmk}
Denote by $r(x,v,\tau):=\frac{1}{1+|x-\tau v^{\alpha}|^{m}+v^{p}}$ and by $p=\alpha m$. Then
\begin{align*}
\partial_{v}r(x,v,\tau)=-\frac{pv^{p-1}-p(x-v^{\alpha}\tau)^{m-1}v^{\alpha-1}\tau}{(1+|x-\tau v^{\alpha}|^{m}+v^{p})^{2}}.
\end{align*}
In other words, $\partial_{v}r(x,v,\tau)=0$ if and only if $x \tau^{\frac{1}{m-1}}=v^{\alpha}(1+\tau^{\frac{m}{m-1}})$. Thus, for fixed $x$ and $\tau$, the point where the function $r$ changes its sign can be approximated for large times by $v^{\alpha}\approx \frac{x}{\tau}$ and for small times by $v^{\alpha}\approx x \tau^{\frac{1}{m-1}}$. The coagulation operator can be approximated at a formal level by a term of the form $v^{\gamma}\partial_{v}r(x,v,\tau)\int_{0}^{\infty}r(x,v',\tau)\der v'$. The change of sign of $r$ was dealt with in the proof made in \cite{cristianinhom} by taking a cut-off at $v^{\alpha}_{\textup{max}}\approx x\tau ^{\frac{1}{m-1}}.$ While at a first glance there is no similarity between the methods to prove existence in \cite{cristianinhom} with the methods in the current paper, the bounds in \eqref{existence space fixed point} and \eqref{converted bounds}, respectively, can be understood at an intuitive level as taking a cut-off at $v^{\alpha}\approx \frac{x}{\tau}$. Thus, our existence result can be understood as the natural continuation of the proof in \cite{cristianinhom} for large times.
\end{rmk}
\section{Construction of solutions with compact support}\label{appendix a}
We now prove that $H_{n+1}$ defined in \eqref{iteration supersol} has compact support assuming that it has compact support at initial time. This is in order to find a time $\tilde{\delta}$ as in Proposition \ref{continuity argument prop} that is independent of $y$ and $v$.

\begin{proof}[Proof of Proposition \ref{prop compact support functions}]
We first analyze the following equation.
\begin{align}
    \partial_{t}f_{n+1}(y,v,t)+\frac{1}{\epsilon}v^{\alpha}\partial_{x}f(x,v,t)=&\mathbb{K}[f_{n+1},f_{n}](y,v,t),
\end{align}
where $\mathbb{K}_{N}[f_{n+1},f_{n}]$ is as in \eqref{equation appendix}. Existence of $f_{n}\in C^{1}([0,T]\times\mathbb{R}\times (0,\infty))$ for small times independent of $\epsilon$ follows via Banach's fixed point theorem.

We then define
\begin{align*}
H_{n}(y,v,\tau):=e^{\tau}f_{n}(x,v,t), \quad x:=e^{\tau}y, \quad t:=e^{\tau}-1, \textup{ for all } n\in\mathbb{N}.
\end{align*}
Thus, $H_{n}\in C^{1}([0,T]\times\mathbb{R}\times (0,\infty))$ and solves \eqref{equation appendix}. Compact support follows directly from \eqref{equation appendix} and the choice of $K_{N}$ in \eqref{truncation kernel in m}.

\end{proof}

\subsection*{Acknowledgements}
 The authors sincerely thank B. Niethammer  for reviewing a draft of the manuscript and contributing to the improvement of its presentation. The authors would like to thank  N. Fournier and  B. Perthame for fruitful discussions about gelation theory and for explaining the arguments in \cite{gelfournier2025}.

I.C. was supported by the European Research Council (ERC) under the European Union's Horizon 2020 research and innovation program Grant No. 637653, project BLOC ``Mathematical Study of Boundary Layers in Oceanic Motion'' and by the project BOURGEONS ``Boundaries, Congestion and Vorticity in Fluids: A connection with environmental issues '', grant ANR-23-CE40-0014-01 of the French National Research Agency (ANR).

J.J.L.V. gratefully acknowledges the financial support of the collaborative
research centre The mathematics of emerging effects (CRC 1060, Project-ID 211504053) and of the Hausdorff Center for Mathematics (EXC 2047/1, Project-ID 390685813) funded through the Deutsche Forschungsgemeinschaft (DFG, German Research Foundation).

\subsubsection*{Statements and Declarations}

\textbf{Conflict of interest} The authors declare that they have no conflict of interest.

\textbf{Data availability} Data sharing not applicable to this article as no datasets were generated or analyzed during the current study.

\addcontentsline{toc}{section}{References}
\printbibliography

\end{document}